\documentclass[11pt, english]{article}
\usepackage{amsmath,amssymb}
\usepackage{bm}
\usepackage{graphicx}
\usepackage{ascmac}
\usepackage{tikz}
\usepackage{tikz-cd}
\usepackage{stmaryrd} 
\usepackage{etoolbox}
\usepackage{amsmath,amsfonts,amssymb}   
\usepackage{amsthm}
\usepackage{enumitem}   
\usepackage{hyperref}
\usepackage{geometry}
\usepackage{color}
\usepackage{diagbox}

\newtheorem{theorem}{Theorem}[subsection]
\newtheorem{proposition}[theorem]{Proposition}
\newtheorem{corollary}[theorem]{Corollary}
\newtheorem{lemma}[theorem]{Lemma}

\theoremstyle{definition}
\newtheorem{definition}[theorem]{Definition}
\newtheorem{example}[theorem]{Example}
\newtheorem{notation}[theorem]{Notation}

\theoremstyle{remark}
\newtheorem{remark}[theorem]{Remark}

\newcommand{\Aut}{\mathop{\mathrm{Aut}}\nolimits}

\newcommand{\coker}{\mathop{\mathrm{coker}}\nolimits}

\newcommand{\Ext}{\mathop{\mathrm{Ext}}\nolimits}

\newcommand{\Fn}{\mathop{\widehat{F}_n^{(\ell)}}\nolimits}
\newcommand{\Gal}{\mathop{\mathrm{Gal}}\nolimits}

\newcommand{\SL}{\mathop{\mathrm{SL}}\nolimits}

\newcommand{\proab}{\mathop{\mathsf{pro}\mathchar`- \mathsf{Ab}}\nolimits}
\newcommand{\proc}{\mathop{\mathsf{pro}\mathchar`- \mathcal{C}}\nolimits}
\newcommand{\progrp}{\mathop{\mathsf{pro}\mathchar`- \mathsf{Grp}}\nolimits}
\newcommand{\proh}{\mathop{\mathsf{pro}\mathchar`- \mathsf{H}}\nolimits}
\newcommand{\prohp}{\mathop{\mathsf{pro}\mathchar`- \mathsf{H}_{\ast}}\nolimits}
\newcommand{\prochp}{\mathop{\mathsf{pro}\mathchar`- \mathcal{C}\mathsf{H}_{\ast}}\nolimits}
\newcommand{\prosets}{\mathop{\mathsf{pro}\mathchar`- \mathsf{Sets}}\nolimits}

\newcommand{\gr}{\mathop{\mathrm{gr}}\nolimits}

\newcommand{\Hom}{\mathop{\mathrm{Hom}}\nolimits}

\renewcommand{\Im}{\mathop{\mathrm{Im}}\nolimits}

\newcommand{\Ker}{\mathop{\mathrm{Ker}}\nolimits}

\newcommand{\lpbracket}{\mathop{(\!(}\nolimits}

\newcommand{\Ob}{\mathop{\mathrm{Ob}}\nolimits}

\newcommand{\Path}{\mathop{\mathrm{Path}}\nolimits}

\newcommand{\Piet}{\mathop{\pi_1^{\mathrm{\acute{e}t}}}\nolimits}

\newcommand{\rpbracket}{\mathop{)\!)}\nolimits}
\newcommand{\sign}{\mathop{\mathrm{sign}}\nolimits}
\newcommand{\sgn}{\mathop{\mathrm{sgn}}\nolimits}

\newcommand{\Spec}{\mathop{\mathrm{Spec}}\nolimits}

\newcommand{\Id}{\mathop{\mathrm{Id}}\nolimits}

\newcommand{\etalchar}[1]{$^{#1}$}

\def\stackbelow#1#2{\underset{\displaystyle\overset{\displaystyle\shortparallel}{#2}}{#1}}

\makeatletter
\newcommand{\colim@}[2]{%
  \vtop{\m@th\ialign{##\cr
    \hfil$#1\operator@font colim$\hfil\cr
    \noalign{\nointerlineskip\kern1.5\ex@}#2\cr
    \noalign{\nointerlineskip\kern-\ex@}\cr}}%
}
\newcommand{\colim}{%
  \mathop{\mathpalette\colim@{\rightarrowfill@\scriptscriptstyle}}\nmlimits@
}
\renewcommand{\varprojlim}{%
  \mathop{\mathpalette\varlim@{\leftarrowfill@\scriptscriptstyle}}\nmlimits@
}
\renewcommand{\varinjlim}{%
  \mathop{\mathpalette\varlim@{\rightarrowfill@\scriptscriptstyle}}\nmlimits@
}

\makeatother

\usetikzlibrary{decorations.markings}
\tikzset{
    set arrow inside/.code={\pgfqkeys{/tikz/arrow inside}{#1}},
    set arrow inside={end/.initial=>, opt/.initial=},
    /pgf/decoration/Mark/.style={
        mark/.expanded=at position #1 with
        {
            \noexpand\arrow[\pgfkeysvalueof{/tikz/arrow inside/opt}]{\pgfkeysvalueof{/tikz/arrow inside/end}}
        }
    },
    arrow inside/.style 2 args={
        set arrow inside={#1},
        postaction={
            decorate,decoration={
                markings,Mark/.list={#2}
            }
        }
    },
}

\makeatletter
\newcommand{\subjclass}[2][2020]{%
  \let\@oldtitle\@title%
  \gdef\@title{\@oldtitle\footnotetext{#1 \emph{Mathematics Subject Classification.} #2}}%
}
\newcommand{\keywords}[1]{%
  \let\@@oldtitle\@title%
  \gdef\@title{\@@oldtitle\footnotetext{\emph{Key words and phrases:} #1.}}%
}
\newcommand{\dates}[1]{%
  \let\@@@oldtitle\@title%
  \gdef\@title{\@@@oldtitle\footnotetext{\emph{Date:} #1.}}%
}
\makeatother

\title{Arithmetic Orr invariants of absolute Galois groups}
\dates{first version, December 15, 2020; this version, \today}
{}

\author{Hisatoshi Kodani and Yuji Terashima}
\date{}

\newcommand{\Addresses}{{
  \bigskip
  \footnotesize

  H.~Kodani, \textsc{Advanced Institute for Materials Research, Tohoku University, 2-1-1 Katahira, Aoba-ku, Sendai, 980-8577, Japan}\par\nopagebreak
  \textit{E-mail address}: \texttt{hisatoshi.kodani@tohoku.ac.jp}

  \medskip

  Y.~Terashima, \textsc{Graduate School of Science, Tohoku University, Aramaki-aza-Aoba 6-3, Aoba-ku, Sendai, 980-8578, Japan}\par\nopagebreak
  \textit{E-mail address}: \texttt{yujiterashima@tohoku.ac.jp}

}}

\begin{document} 
 \maketitle

\begin{abstract}
	Based on the analogies between mapping class groups and absolute Galois groups, we introduce an arithmetic  pro-$\ell$ analogue of Orr invariants for a Galois element associated with Galois action on \'etale fundamental groups of punctured projective lines. At the same time, we also introduce pro-$\ell$ Orr space as an arithmetic analogue of Orr space whose third homotopy group is a target group of Orr invariant. We then determine its rank as $\mathbb{Z}_{\ell}$-module following Igusa-Orr's computation. Moreover, we investigate its relation with Ellenberg's obstruction to $\pi_1$-sections associated with lower central series filtration in the context of Grothendieck's section conjecture.%
\end{abstract}

\tableofcontents

\section{Introduction} \label{sec:1}

In low dimensional topology, Milnor invariants are important invariants to distinguish string links up to link homotopy (\cite{HaLi}). Milnor invariants are known as finite type (Vassiliev) invariants of string links. Habegger and Masbaum \cite{HaMa} give an explicit formula relating the first non-vanishing Milnor invariants and Kontsevich invariants of string links.

Orr introduced an invariant $\theta_k$ for a based link in  \cite{O1}. The invariant takes values in the third homotopy group $\pi_3(K_k)$ of certain space $K_k$, we call it the  Orr space, obtained from an Eilenberg-MacLane space of a $k$-th nilpotent quotient of the free group. He showed that his invariants are defined for based link whose Milnor invariants of length $\leq k$ vanish, and the invariant $\theta_k$ contains all information of Milnor invariants of lengths $k+1$  to $2k-1$. 

In this paper, we introduce and investigate an arithmetic pro-$\ell$ analogue of Orr invariants $\theta_k^{(\ell)}$ and Orr space $K_k^{(\ell)}$ for a based Galois element of an absolute Galois group based on analogies between mapping class groups and absolute Galois groups as in Kodani-Morishita-Terashima \cite{KMT}. In fact, we treat a Galois element as if it were a mapping class of some pro-space and construct ``mapping torus'' and perform ``surgery'' as usual arguments in topology. Our main tool for performing it is homotopy theory of pro-spaces and their profinite completion which are introduced by Artin-Mazur \cite{AM} in their study of \'etale homotopy theory and after them further developed by Friedlander \cite{F} and Sullivan \cite{Su}. In \cite{KMT}, Ihara's work (\cite{I1}, \cite{I2}) is reinterpreted from the viewpoint of arithmetic topology. More concretely,  Jacobi sums and Soul\'e characters are shown to have expressions as linear combinations of arithmetic $\ell$-adic Milnor invariants of  Galois elements. Hence, our arithmetic Orr invariants may be thought of as a refinement of Jacobi sums and Soul\'e characters.

Second, we compute the rank of the third homotopy and homology groups of the arithmetic Orr space using a pro-$\ell$ analogue of methods in Igusa-Orr \cite{IO}. These computations are closely related to the dimension of vector space of $H$-colored tree Jacobi diagrams as shown by Massuyeau \cite{Ma} in his study of a relationship between total Johnson map and LMO homomorphism.

Finally,  we relate arithmetic Orr invariants to refinements of
Ellenberg's obstructions introduced in the context of Grothendieck's
section conjecture (\cite{E}, see also \cite{W1}, \cite{W2} and \cite{W3}).
In fact, we introduce a double indexed version of Ellenberg's obstructions with
a double indexed nilpotent tower not only  via the nilpotent quotient by Ellenberg but also via the Johnson filtration, and
show that the vanishing of the obstructions is equivalent to the vanishing of a 1-cocycle obtained from arithmetic Orr invariants (Theorem 9.2.5). In addition, we also give an explicit vanishing condition of arithmetic Orr invariants in terms of $\ell$-adic Milnor invariants for the case of the projective line $\mathbb{P}^1 \setminus \{0, 1, \infty\}$ minus three points in lower degree case (Theorem 9.3.2). For this, we apply a technique developed in Habegger-Masbaum \cite{HaMa} which relates tree Jacobi diagrams and Milnor invariants.  As a byproduct, we give a simple diagrammatic proof of Ihara's computation of vanishing of $\ell$-adic Milnor invariants with repetitive indices (Proposition 9.3.1).

\subsection*{Outline}
Here is an outline of this article. In Section 2, for the convenience of readers, we review background material, such as pro-objects, profinite completion of homotopy types, Galois actions on \'etale fundamental groups, pro-$\ell$ Magnus embeddings, and $\ell$-adic Milnor invariants. In Section 3, we define an arithmetic analogue of Orr space $K_k^{(\ell)}$ and Orr invariants $\theta_k^{(\ell)}$ for a based Galois element. In addition, we also consider the image $\tau_k^{(\ell)}$ of the invariants under Hurewicz homomorphism. In Section 4, we recall Igusa-Orr computation \cite{IO} of low dimensional homology groups for free nilpotent Lie algebra and see how it works in our situation. In Section 5, we compute low dimensional homology groups for torsion-free nilpotent pro-$\ell$ groups by considering a pro-$\ell$ version of  Igusa-Orr computation. In section 6, we compute the rank of the third homotopy group of a pro-$\ell$ (complete) Orr space $\widehat{K}_k^{(\ell)}$. In Section 7, we recall Massey products and their basic properties. In Section 8, we study basic (algebraic) properties of pro-$\ell$ Orr invariants $\theta_k^{(\ell)}$ and $\tau_k^{(\ell)}$. In Section 9, as an application of what we have done, we study pro-$\ell$ Orr invariants in the context of Grothendieck's section conjecture. In particular, we investigate a relation between Ellenberg obstructions and our arithmetic Orr invariants. Moreover, we give an explicit condition of vanishing of pro-$\ell$ Orr invariants in terms of $\ell$-adic Milnor invariants via tree Jacobi diagrammatic technique developed in \cite{HaMa} for lower degrees. Besides, we also give a simple diagrammatic proof of Ihara's computation of $\ell$-adic Milnor invariants with repetitive indices.

\subsection*{Notation}

We denote by $\mathbb{Z}$, $\mathbb{Q}$ and $\mathbb{C}$ the ring of (rational) integers, the field of rational numbers, and the field of complex numbers respectively. Throughout this paper, the symbol $\ell$ stands for a fixed (rational) prime number in $\mathbb{Z}$. We denote by $\mathbb{Z}_{\ell}$, $\mathbb{Q}_{\ell}$ and $\mathbb{F}_{\ell}$ the ring of $\ell$-adic integers, the field of $\ell$-adic numbers, and the finite field of order $\ell$ respectively. 

For subgroups $A$, $B$ of a (topological) group $G$, we denote by $[A, B]$ the closed subgroup of $G$ (topologically) generated by commutators $[a, b] := aba^{-1}b^{-1}$ for all $a \in A$ and $b \in B$. For a (topological) group $G$, we denote by $\Gamma_k(G)$ the $k$-th lower central subgroup recursively defined as $\Gamma_1(G):=G$ and $\Gamma_{k+1}(G):= [\Gamma_k(G), G]$ $(k \geq 1)$.

For a (topological) group $G$ and a positive integer $n$, the symbol $K(G, n)$ denotes an Eilenberg-MacLane space of type $(G, n)$. By abuse of notation, for a pro-group $G$, an Eilenberg-Maclane pro-space is denoted by the same symbol $K(G,1)$.

For a group $G$, we denote by $\widehat{G}$ or $G \ \widehat{\ }$ the profinite completion of $G$. Similarly, we denote by $\widehat{G}^{(\ell)}$ or ${G \ \widehat{\ }}^{(\ell)}$ the pro-$\ell$ completion of $G$. For the homotopy group $\pi_i(X)$ of a (pro-)space $X$, we sometimes denote its profinite completion by $\widehat{\pi}_i(X)$ and its pro-$\ell$ completion by $\widehat{\pi}_i^{(\ell)}(X)$.

For a number field $K$ and an algebraic closure $\overline{K}$ of $K$, we denoted by $\Gal(\overline{K}/K)$ or $G_K$ the absolute Galois group of $K$.

For topological spaces $X$ and $Y$, $X \sim Y$ means that $X$ and $Y$ are weak homotopy equivalent.

We denote by $\mathsf{Sets}, \mathsf{Grp}, \mathsf{Ab}$ the category of sets, groups, and Abelian groups respectively. For a category $\mathsf{C}$, we often write  $ X \in \mathsf{C}$ to mean that $X$ is an object of $\mathsf{C}$.

The symbol $\widehat{\otimes}$ stands for the complete tensor product or graded tensor product depending on context when there is no possibility of confusion.

\section{Background material} \label{sec:2}

For the convenience of readers, this section reviews background material such as pro-objects, profinite completion of homotopy types, Galois action on \'etale fundamental groups, pro-$\ell$ Magnus embeddings, and $\ell$-adic Milnor invariants of based Galois elements.
\subsection{Pro-objects}

In Section 2.1 to 2.3, we recall the basics of the idea of profinite completion of homotopy type of CW complex founded by Artin and Mazur in \cite{AM}. For details see \cite{AM} and \cite{Su}. In particular, \cite{Su} contains several concrete helpful examples.

To begin with, we recall the notion of pro-object in some category $\mathsf{C}$. A category $\mathsf{I}$ is said to be {\it cofiltering} if it satisfies the following properties (i) and (ii):
\begin{enumerate}[label=(\roman*)]
\item for any pair of objects $i,j \in \Ob(\mathsf{I})$,	there is an object $k \in \Ob(\mathsf{I})$ and morphisms with 
\begin{center}
	\begin{tikzcd}[row sep=tiny]
 						& i  \\
k \arrow[ur] \arrow[dr] & \\
& j
\end{tikzcd}
\end{center}	

\item for any pair of morphisms $f,g : i \rightarrow j$, their is a morphism $h: k \rightarrow i$ such that two composites are equal $f\circ h = g \circ h : k \rightarrow j$.
\end{enumerate}

\begin{remark}
By reversing arrows in the definition above, we get a definition of  {\it filtering} category.
\end{remark}

Let $\mathsf{C}$ be a category and $\mathsf{I}$ be a cofiltering index category. A {\it pro-object} in the category $\mathsf{C}$ is a covariant functor
\begin{equation} \label{eq:2.1.1}
	X : \mathsf{I} \rightarrow \mathsf{C}.
\end{equation}
For each index $i \in \mathsf{I}$, we denote its value $X(i)$ by $X_i$. By identifying a pro-object $X$ and its values, we often use the notation
\begin{equation} \label{eq:2.1.2}
	X = \{ X_i \}_{i \in \mathsf{I}}
\end{equation}

Since the index category $\mathsf{I}$ is cofiltering, one can regard a pro-object $\{X_i \}_{i \in I}$ as an inverse system of objects of $\mathsf{C}$.

Now we recall the definition of the category $\mathsf{pro}\mathchar`- \mathsf{C}$ of pro-objects valued in $\mathsf{C}$. The objects of $\mathsf{pro}\mathchar`- \mathcal{C}$ are pro-objects $X=\{X_i \}_{i \in \mathsf{I}}$ and the set of morphisms $\Hom(X, Y)$ is given by
\begin{equation} \label{eq:2.1.3}
	\Hom(X, Y) := \varprojlim_{j \in \mathsf{J}} \colim_{i \in \mathsf{I}} \Hom(X_i, Y_j).
\end{equation}
for any pro-objects $X=\{X_i\}_{i \in \mathsf{I}}$ and $Y=\{Y_j\}_{j \in \mathsf{J}}$.

\begin{remark}
We note that in \cite{AM} they define a pro-object as a contravariant functor on a filtering index category.
\end{remark}

\subsection{Pro-objects in the category of groups and homotopy category}
This section recalls pro-objects in the category of groups and that of the homotopy category to fix notations.

Let $\mathsf{Grp}$ the category of groups whose objects are groups and morphisms are homomorphism. Then, $\progrp$ is the category whose objects are pro-objects $G= \{G_i\}_{i \in I}$ of groups and morphisms are given as $\Hom(G,H)=\varprojlim \colim \Hom(G_i, H_j)$ for pro-objects $G=\{G_i\}_{i \in I}$ and $H=\{H_j\}_{j \in J}$. We call the category $\progrp$ the {\it category of pro-groups} and its objects {\it pro-groups}. One may regard the category $\progrp$ is obtained  from the category $\mathsf{Grp}$ by formally adding inverse systems. Note that any group $G \in \Ob(\mathsf{Grp})$ can be thought of as an object of $\progrp$ by considering constant inverse system $\{G\}_{i \in I}$. Hence, we have a functor 
\begin{equation} \label{eq:2.2.1}
	\mathsf{Grp} \rightarrow \progrp
\end{equation}

Next, we consider the homotopy category. Let $\mathsf{H}$ be the category whose objects are connected CW complexes and morphisms are given homotopy classes of continuous maps between them denoted by $[X, Y]$ for $X, Y \in \Ob(\mathsf{H})$. Similarly, we denote by $\mathsf{H}_{\ast}$ the category of pointed connected CW complexes. In this article,  we refer to $\mathsf{H}$ and $\mathsf{H}_{\ast}$ as {\it homotopy category} and {\it pointed homotopy category} respectively. Similar as the category of groups, we define pro-category $\proh$ and $\prohp$ by obvious manner and we have a functor
\begin{equation} \label{eq:2.2.2}
	\mathsf{H} \rightarrow \proh
\end{equation}
and
\begin{equation} \label{eq:2.2.3}
	\mathsf{H}_{\ast} \rightarrow \proh_{\ast}.
\end{equation}
In this article, we also refer to (pointed) CW complexes as {\it (pointed) spaces}. 

For pointed pro-spaces, we can consider their homotopy groups, homology groups and cohomology groups similar to the usual pointed spaces case as follows: For each non-negative integer $k \geq 0$, the homotopy group functor $\pi_k$ is defined as 
\begin{equation} \label{eq:2.2.4}
	\pi_k : \prohp \rightarrow \progrp, \quad X=\{X_i\}_{i \in I} \mapsto \pi_k(X):=\{\pi_k(X_i)\}_{i \in I},
\end{equation}
the homology group functor $H_k$ is defined as
\begin{equation}\label{eq:2.2.5}
	H_k : \prohp \rightarrow \proab, \quad X=\{X_i\}_{i \in I} \mapsto H_k(X) := H_k(X; \mathbb{Z}) := \{H_k(X_i; \mathbb{Z})\}_{i \in I},
\end{equation}
and the cohomology group functor $H^k$ is defined as
\begin{equation} \label{eq:2.2.6}
	H^k : \prohp \rightarrow \proab, \quad X=\{X_i\}_{i \in I} \mapsto H^k(X; \mathbb{Z}) := \{H^k(X_i; \mathbb{Z})\}_{i \in I}.
\end{equation}
Here, note that $\{H^k(X_i; \mathbb{Z})\}_{i \in I} = \colim_i H^k(X_i; \mathbb{Z})$.

We define homology and cohomology groups with coefficients in pro-module twisted by fundamental pro-group as follows: For a pro-space $X = \{X_i \}_{i \in I} \in \prohp$, take an Abelian group $A$ with representation $\tau :  \pi_1(X) \rightarrow  \Aut(A)$. Then, we define twisted homology group $H_k(X; A)$ and $H^k(X; A)$ as 
\begin{equation}
	H_k(X;A) = \{ H_k(X_i; A_{\phi})\}_{(i,\phi)}
\end{equation}
and
\begin{equation}
	H^k(X;A) = \colim_{(i, \phi)} H^k(X_i;  A_{\phi})
\end{equation}
where the index category consists of pairs $(i, \phi)$ such that $\phi : \pi_1(X_i) \rightarrow \Aut(A)$ is a homomorphism representing $\tau$.

\begin{remark}
The above definitions of homotopy groups and (co)homology groups are the same as in \cite[\S 2]{AM}	
\end{remark}

Next, we also define twisted homology group with coefficients in a pro-$\llbracket \widehat{\mathbb{Z}} \pi_1(X) \rrbracket$-module   $B$ as follows: For a pro space $X = \{X_i\}_{i \in \mathsf{I}}$ and a pro-module $B = \{B_i\}_{i \in \mathsf{I}}$ where each $B_i$ is a pro-$\llbracket \widehat{\mathbb{Z}} \pi_1(X_i) \rrbracket$-module. Then, we set
\begin{equation}
	H_n(X; B) = \{ H_n(X_i; B_i) \}_{i \in \mathsf{I}}
\end{equation}

Note that this definition corresponds to that of the profinite group homology group with coefficients in $B$ (cf. Lemma 2.2).

\subsection{Completion in the homotopy category}

With notations in Section 2.1 and 2.2, let us recall the notion of pro-$\mathcal{C}$ completion in the homotopy category. 

Let $\mathcal{C}$ be the class of finite groups, i.e., $\mathcal{C}$ is the full subcategory of $\mathsf{Grp}$ whose objects are finite groups (see \cite[Section 2.1]{RZ}). \footnote{In this article we restrict ourselves to the class of finite groups, although  the original article (\cite[\S 3]{AM})  treats more general class of groups.} Denote by $\proc$ the category of pro-objects in the class $\mathcal{C}$, called {\it pro-$\mathcal{C}$ groups}.  Then, in our situation, the pro-$\mathcal{C}$ completion of a group is restated as follows: Since each pro-object in $\proc$ can be viewed as pro-objects in $\progrp$, so we have an inclusion functor 
\begin{equation} \label{eq:2.3.1}
	\proc \rightarrow \progrp
\end{equation}
and it has a left adjoint, called the {\it pro-$\mathcal{C}$ completion}
\begin{equation}  \label{eq:2.3.2}
	\widehat{\ }\  : \progrp \rightarrow \proc
\end{equation}
By composing it with $\mathsf{Grp} \rightarrow \progrp$, we have a functor
\begin{equation}  \label{eq:2.3.3}
	\mathsf{Grp} \rightarrow \progrp \overset{\widehat{}}{\rightarrow} \proc
\end{equation}
and this functor is exactly the pro-$\mathcal{C}$ completion of a group in usual sense. The pro-$\mathcal{C}$ completion is equivalent to the following.

\begin{itemize}
\item 	for $G \in \Ob(\mathsf{Grp})$, the canonical map $G \rightarrow \widehat{G}$ is universal with respect to maps into pro-$\mathcal{C}$ groups,
\item for $G \in \Ob(\mathsf{Grp})$, the functor $\Hom(G, \cdot) : \mathcal{C} \rightarrow \mathsf{Sets}$ restricted to the category $\mathcal{C}$ is  pro-representable (see \cite[Appendix \S 2]{AM}) in $\mathcal{C}$. In particular, the pro-$\mathcal{C}$ completion $\widehat{G}$ represents the functor $\Hom(G, \cdot)$.
\end{itemize}

Similar to the group case, Artin and Mazur conceived the notion of pro-$\mathcal{C}$ completion of pointed CW complexes as follows. For a class of groups $\mathcal{C}$, we denote by $\mathcal{C}\mathsf{H}_{\ast}$ the full subcategory of $\mathsf{H}_{\ast}$ whose objects are pointed connected CW complexes with homotopy groups all in $\mathcal{C}$-groups. Then, we have the following analogue of the group completion for pointed spaces.
\begin{theorem}[{\cite[Theorem (3.4)]{AM}}]  \label{thm:2.3.1}
	The inclusion functor
	\begin{equation}  \label{eq:2.3.4}
		\prochp \rightarrow \prohp 
	\end{equation}
	has a left adjoint
	\begin{equation}  \label{eq:2.3.5}
		\widehat{} \  : \prohp \rightarrow \prochp.
	\end{equation}
Equivalently, we have the followings.
\begin{itemize}
	\item For $X \in \prohp$, there exists an object $\widehat{X} \in \prochp$, called {\it pro-$\mathcal{C}$ completion} of $\widehat{X}$ with a map $X \rightarrow \widehat{X}$ which is universal with respect to maps into objects of $\mathcal{C}\mathsf{H}_{\ast}$. 
	\item For $X \in \prohp$, the functor $[X, \cdot] : \mathcal{C}\mathsf{H} \rightarrow \mathsf{Sets}$ restricted to the category $\mathcal{C}\mathsf{H}$  is pro-representable in $\mathcal{C}\mathsf{H}$. In particular, the pro-$\mathcal{C}$ completion $\widehat{X}$ represents the functor $[X, \cdot]$.
\end{itemize}
\end{theorem}

In terms of Theorem \ref{thm:2.3.1}, we get a functor
\begin{equation}
	\mathsf{H}_{\ast}  \rightarrow \prohp \overset{\widehat{}}{\rightarrow} \prochp  
\end{equation}
which gives the pro-$\mathcal{C}$ completion of (pointed) spaces.

\begin{notation}
For two pro-spaces $X$ and $Y$, we say that $X$ and $Y$ are weak (homotopy) equivalent and denoted by $X \sim Y$ if they are $\natural$-isomorphic in \cite[Definition (4.2)]{AM}.	
\end{notation}

\begin{remark}
It is known that for two pro-spaces $X$ and $Y$, $X \sim Y$ if and only if $\pi_k(X) \simeq \pi_k(Y)$ for any $k$ (\cite[Corollary (4.4)]{AM}).
\end{remark}

\begin{example}[{\cite[p. 72]{Su}, \cite[Example 6.12]{AM}, and \cite[I.2.6]{Ser2}}]  \label{ex:2.3.1}
	Let $\mathcal{C}$ be the class of finite groups. 
	\begin{enumerate}[label=(\roman*)]  		\item The profinite completion of the circle $S^1$ is given as
		\begin{equation}  \label{eq:2.3.6}
			\widehat{S}^1 \sim  \widehat{K(\mathbb{Z}, 1)}  \sim K(\widehat{\mathbb{Z}}, 1).
		\end{equation}
		\item The profinite completion of the infinite dimensional complex projective space $\mathbb{CP}^{\infty} \sim K(\mathbb{Z}, 2)$ is given as
		\begin{equation}  \label{eq:2.3.7}
			\widehat{\mathbb{CP}^{\infty}}  \sim K(\widehat{\mathbb{Z}}, 2) \sim \widehat{K(\mathbb{Q} / \mathbb{Z}, 1)}
		\end{equation}
		\item  More generally, for finitely generated Abelian group $G$, we have
		\begin{equation}  \label{eq:2.3.8}
			\widehat{K(G, n)} \sim K(\widehat{G}, n) \sim K(G \otimes_{\mathbb{Z}} \widehat{\mathbb{Z}}, n).
		\end{equation}
		
		\item The profinite completion of an Eilenberg-MacLane space $K(\SL_2(\mathbb{Z}), 1)$ is given by
		\begin{equation}  \label{eq:2.3.9}
			K(\SL_2(\mathbb{Z}), 1) \  \widehat{}\  \sim K(\widehat{\SL_2(\mathbb{Z})}, 1),
		\end{equation}
		but for an integer $n \geq 3$, the profinite completion of an Eilenberg-MacLane space  $K(\SL_n(\mathbb{Z}), 1)$ is {\it not} weak equivalent to $K(\widehat{\SL_n(\mathbb{Z})}, 1)$;
		\begin{equation}  \label{eq:2.3.10}
			K(\SL_n(\mathbb{Z}), 1) \  \widehat{}\  \not\sim K(\widehat{\SL_n(\mathbb{Z})}, 1)
		\end{equation}
	\end{enumerate}
	\end{example}

	\begin{remark}
	As we see in Example \ref{ex:2.3.1} (iv), for a group $G$, the profinite completion of  Eilenberg-MacLane space $K(G,1)$ is not weak equivalent to $K(\widehat{G}, 1)$ in general. The weak equivalence $\widehat{K(G,1)} \sim K(\widehat{G}, 1)$ holds if and only if $G$ is ``$\mathcal{C}$-good'' in the sense of Serre (\cite[I.2.6]{Ser2}, see also \cite[Corollary (6.6)]{AM}). Here, for a pro-group $G$, $G$ is said to be {\it ``$\mathcal{C}$-good'' in the sense of Serre} if for every twisted abelian $\widehat{G}$-module $A \in \mathcal{C}$, there are isomorphisms
	\begin{equation}
		H^q(\widehat{G}; A) \simeq H^q(G; A) \quad \text{for any $q \geq 0$}.
	\end{equation}

	Finite groups, finitely generated Abelian groups, finitely generated free groups, and (pure) braid groups are examples of ``$\mathcal{C}$-good'' groups ([loc.cit.]).	
	\end{remark}

At the end of this section, we give a pro-space analogue of well-known isomorphism between the twisted (co)homology groups of $K(G,1)$ and the twisted (co)homology groups of $G$. In the following sections, we heavily use these isomorphisms.

\begin{lemma}  \label{lem:2.3.3}
Let $G$ be a profinite group and $K(G,1)$ be an Eilenberg-Maclane pro-space of type $(G,1)$. Then, the following statements hold.
\\
(1) Let $A$ be a discrete $G$-module. Then, for $n \geq 0$ we have
\begin{equation} \label{eq:2.3.13}
	H^n(K(G, 1); A) \simeq H^n(G; A)
\end{equation}
(2) Let $B$ be a profinite right $\llbracket \widehat{\mathbb{Z}} G \rrbracket$-module. Then, for each $n \geq 0$  we have
\begin{equation} \label{eq:2.3.14}
	H_n(K(G, 1); B) \simeq H_n(G; B)
\end{equation}
\end{lemma}

\begin{proof}
(2) follows from (1) by Pontryagin duality, so it suffices to show  (1).
Since $K(G,1)$ is unique up to weak equivalence (\cite[Corollary 4.14]{AM}), we may assume that $K(G,1)$ is given by $\{K(G/U,1)\}_{U \in \mathcal{U}}$ where $\mathcal{U}$ is the set of all open normal subgroup of $G$, and hence $A=\{A^U\}_{U \in \mathcal{U}}$ where $A^U$ is a $G/U$-module. Then, one gets
\begin{eqnarray*}
	H^n(K(G,1); A) & = &\colim_U H^n(K(G/U, 1); A^U) \\
	&\simeq & \colim_U H^n(G/U; A^U)\\
	&\simeq & H^n(G; A).
\end{eqnarray*}
Here, in the second isomorphism we use the fact $H^n(K(G/U,1); A^U)\simeq H^n(G/U; A^U)$ and in the final isomorphism is \cite[Proposition 6.5]{RZ}. 
\end{proof}

\subsection{\'Etale fundamental group}
In Section 2.4 to 2.7, to fix our notations, we recall basics of  \'etale fundamental groups, rational tangential points, Galois action on the \'etale fundamental groups in our situation following \cite{W3}. For more details, see \cite{SGA1}, \cite{D}, \cite{Na} and \cite{Wo}.

Let $K$ be a number field, i.e., a finite field extension of the rational number field $\mathbb{Q}$. Fix an embedding $K \hookrightarrow \mathbb{C}$. Let $\overline{K}$ be a fixed algebraic closure in $\mathbb{C}$. Let $n$ be a fixed integer with $n \geq 2$. Take $n$ distinct points $\{a_1, \ldots, a_n\}$ in $K$. Let us consider $X=\mathbb{P}^1_K \setminus \{\infty, a_1, \ldots, a_n \} \rightarrow \Spec(K)$ the projective line minus $\{\infty, a_1, \ldots, a_n\}$ defined over $K$. Then, by taking a geometric point $b$ of $X$ (that is,  a morphism $b : \Spec(\Omega) \rightarrow X$ where $\Omega$ is an algebraically closed field containing $K$), we obtain a functor called {\it fiber functor}
\begin{equation}
	F_{b} : \mathsf{Et}_X \rightarrow \mathsf{Sets}
\end{equation}
which takes a finite \'etale cover $Y \overset{\pi}{\rightarrow} X$ to $\pi^{-1}(b)$. Here $\pi^{-1}(b)$ is the set of geometric points of $Y$ over $b$, that is, the set of  morphisms $b_Y : \Spec(\Omega) \rightarrow Y$ such that $\pi \circ b_Y = b$.  Note that a fiber functor plays the role of the set of fiber over a chosen base point as in the usual covering space theory. 

In the above situation, a path between two base points are described in terms of natural transformations between functors. More precisely, given two geometric points $b_1, b_2$ of $X$, we denote by $\Path(b_1, b_2)$ the set of invertible natural transformation on the corresponding fiber functors $F_{b_1} \rightarrow F_{b_2}$. Note that $\Path(b_1, b_2)$ can be though of as a pro-object of finite sets $\{F_{b_1}(Y) \rightarrow F_{b_2}(Y)\}_{Y \in \mathsf{Et}_X}$ in $\prosets$ and so $\Path(b_1, b_2)$ is equipped with profinite topology. Note that the products of paths $\gamma_1 \in \Path(b_1, b_2)$ and $\gamma_2 \in \Path(b_2, b_3)$ are given by $\gamma_2 \cdot \gamma_1$. When $b_1 = b_2=b$, set
\begin{equation}
	\Piet(X, b) := \Path(b,b)=\Aut(F_b)
\end{equation}
and call it {\it \'etale fundamental group} of $X$ based at $b$. Note that $\Piet(X, b)$ is naturally equipped with profinite topology, i.e., $\Piet(X,b)$ is a profinite group.

\subsection{Galois action on \'etale fundamental groups}

Keep the notation as in Section 2.4.  Let  $X_{\overline{K}}:= X \times_{\Spec (K)} \Spec( \overline{K})$ the base change of $X$ to $\overline{K}$. Note that  each element $g$ of $G_K=\Gal(\overline{K}/K)$ acts on $X_{\overline{K}}$ via the functor $-\otimes_g \overline{K}$ given by the pullback of $g : \Spec(\overline{K}) \rightarrow \Spec(\overline{K})$.

A $K$-rational point $b: \Spec (K) \rightarrow X$ induces a geometric point $\bar{b} : \Spec (\overline{K}) \rightarrow X_{\overline{K}}$ of the base change $X_{\overline{K}}$. Therefore, the section  $\bar{b} : \Spec(\overline{K}) \rightarrow X_{\overline{K}}$ of the structure morphism induced by the rational point $\bar{b}$ gives rise to the following commutative diagram involving $g \in G_K$ action:

\begin{center}
\begin{tikzcd}
\Spec (\overline{K}) \ar[r, "g"] \ar[d, "\bar{b}"'] & \Spec (\overline{K}) \ar[d, "\bar{b}"] \\
X_{\overline{K}} \ar[r, "g"'] & X_{\overline{K}}
\end{tikzcd}
\end{center}

From the commutative diagram, we obtain the induced natural isomorphism $(-\otimes_g \overline{K})^{\ast}F_{\bar{b}} \rightarrow F_{\bar{b}}$ between fiber functors. Thus, we finally get a $G_K$-action on the set of profinite paths between two such geometric points.

\subsection{Rational tangential points}

Next, let us recall the notion of rational tangential points. For more details, see \cite[\S 15]{D} and \cite{Na}.

Let $\overline{X}$ be the smooth compactification of $X$. That is, $\overline{X}=\mathbb{P}^1_{K}$.  Take a rational point $x : \Spec(K) \rightarrow \overline{X}$ of $\overline{X}$. Then, the complete local ring of $\overline{X}$ at the image of $x$ is isomorphic to the ring of formal power series $K[[\epsilon]]$ in  $\epsilon$ over $K$. Such an isomorphism induces the map from the function field of $X$ into the field of Laurent power series $K(\!(\epsilon)\!)$ in $\epsilon$ over $K$.  This  gives rise to a map,  called a {\it rational tangential point},

\begin{equation}
	\vec{b} : \Spec (K(\!(\epsilon)\!)) \rightarrow X.
\end{equation}

As in the case of rational points, such a rational tangential point $\Spec(K(\!(\epsilon)\!)) \rightarrow X$ produces a map $\Spec(\overline{K}(\!(\epsilon)\!)) \rightarrow X_{\overline{K}}$  which factors through the generic point of $X$.

Now, let us consider the field of Puiseux power series $\cup_{n \in \mathbb{Z}_{>0}} \overline{K}(\!(\epsilon^{1/n})\!)$ in the symbol ``$\epsilon^{1/n}$'' with $(\epsilon^{1/mn})^m = \epsilon^{1/n}$ for $m, n \in \mathbb{Z}_{>0}$. Note that the field of Puiseux power series is algebraically closed since the field $\overline{K}$ is an algebraically closed field of characteristic 0 (\cite[Chap. IV \S 2 Proposition 8]{Ser1}). Then, through the canonical embedding $K(\!(\epsilon)\!) \rightarrow \cup_{n \in \mathbb{Z}_{>0}} \overline{K}(\!(\epsilon^{1/n})\!)$, one obtains a geometric point
\begin{equation}
	\vec{b}_{\Omega} : \Spec \Omega  \rightarrow X_{\overline{K}}
\end{equation}
where we set $\Spec \Omega := \Spec (\cup_{n \in \mathbb{Z}_{>0}} \overline{K}(\!(\epsilon^{1/n})\!))$. By the same manner as Section 2.5, the geometric point $\vec{b}_{\Omega}$ gives rise to  $G_K$-action between fiber functors $F_{\vec{b}_{\Omega}}$ where $G_K$-action on $\Spec(\Omega)$ is induced by $G_K$-action on coefficients of Puisueux series. In this way, we get $G_K$-action on profinite paths between two geometric points induced by rational tangential points. In the following, by a tangential base point, we mean a geometric point obtained by a rational tangential point.

\subsection{Galois action on $\Piet(\mathbb{P}^1_K \setminus \{\infty, a_1, \ldots, a_n \})$}

Finally, we recall a Galois action on the  \'etale fundamental groups of $\mathbb{P}^1_K \setminus \{\infty, a_1, \ldots, a_n \}$. Let $b_0, b_1, \ldots, b_n$ be tangential points based at $a_0:= \infty, a_1, \ldots, a_n$ respectively.

To begin with, we recall the notion of geometric generators of $\pi_1(X(\mathbb{C}), v)$ for a base point or a tangential point $v$ following \cite[\S2]{Wo}.

First, we consider the case of $v \in X(\mathbb{C})$. As in the Figure 1, we take a path $\gamma_i \in \Path(v, b_i)$ from $v$ to $b_i$ for $0 \leq i \leq n$ so that any two paths do not intersect and no path self-intersect. For each puncture $a_i$ $(0 \leq i \leq n)$, we take a small loop $l_i \in \Path(b_i, b_i)$ in the opposite clockwise direction. Then, we  define the loop based at $v$ by
\begin{equation} \label{eq:-2.0.0}
	x_i := \gamma_i^{-1} l_i \gamma_i \quad (0 \leq i \leq n)
\end{equation}
Moreover, we may assume that the indices are chosen so that, if we take a small loop in the opposite clockwise direction starting from $\gamma_1$, then we meet successively $\gamma_2,\ldots, \gamma_n, \gamma_0$. Then, $x_1, \ldots, x_{n}, x_0$ forms a generator of $\pi_1(X(\mathbb{C}, v)$ subject to the relation $x_{0} x_n \cdots x_1 = 1$. In particular, $\pi_1(X(\mathbb{C}), v)$ is isomorphic to free group $F_n$ generated by $x_1, \ldots, x_n$.
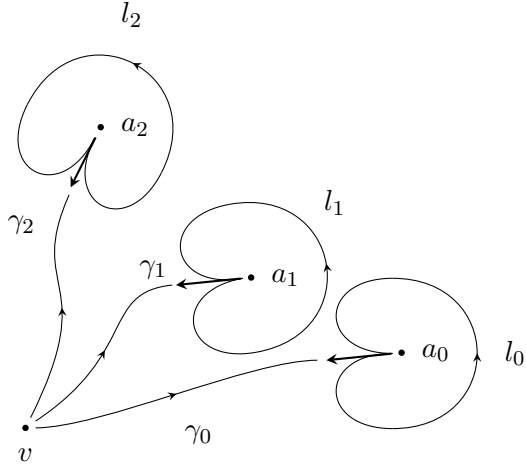
\begin{figure}[h] \label{fig:2.7.1}
\begin{center}
\begin{tikzpicture}
	\node[label=below:$v$] (A) at (0,0) {};
    \node[label=right:$a_0$] (B) at (5,1){};
    \node[label=right:$a_1$] (C) at (3,2){};
    \node[label=right:$a_2$] (D) at (1,4){};
	\begin{scope}[xshift=4.88cm, yshift=0.99cm]
		\draw (0,0) .. controls (-1,0) and (-1,1) .. (0,1);
		\draw (0,-1) .. controls (1.5,-1) and (1.5,1) .. (0,1) [arrow inside={end=stealth,opt={scale=1}}{0.5}];
		\draw (0,-1) .. controls (-1, -1) and (-1,0) .. (0,0);
	\end{scope}
	\begin{scope}[xshift=2.89cm, yshift=1.99cm,rotate=10]
		\draw (0,0) .. controls (-1,0) and (-1,1) .. (0,1);
		\draw (0,-1) .. controls (1.5,-1) and (1.5,1) .. (0,1) [arrow inside={end=stealth,opt={scale=1}}{0.5}];
		\draw (0,-1) .. controls (-1, -1) and (-1,0) .. (0,0);
	\end{scope}
	\begin{scope}[xshift=0.917cm, yshift=3.85cm, rotate=63]
		\draw (0,0) .. controls (-1,0) and (-1,1) .. (0,1);
		\draw (0,-1) .. controls (1.5,-1) and (1.5,1) .. (0,1) [arrow inside={end=stealth,opt={scale=1}}{0.5}];
		\draw (0,-1) .. controls (-1, -1) and (-1,0) .. (0,0);
	\end{scope}
	\draw (A) .. controls (1,0) and (3,0.9) .. (3.88,0.9) [arrow inside={end=stealth,opt={scale=1}}{0.5}];;
	\draw (A) .. controls (1.5, 1) and (1, 1.8) .. (1.95,1.9) [arrow inside={end=stealth,opt={scale=1}}{0.5}];
	\draw (A) .. controls (1, 2) and (0, 1.8) .. (0.58, 3.1) [arrow inside={end=stealth,opt={scale=1}}{0.5}];
    \draw [fill=black] (A) circle (1pt);
    \draw [fill=black] (B) circle (1pt);
    \draw [fill=black] (C) circle (1pt);
    \draw [fill=black] (D) circle (1pt);
    \node[label=right:$l_0$] at (6.1,1) {};
    \node[label=right:$l_1$] at (3.7,3) {};
    \node[label=right:$l_2$] at (1,5.5) {};
    \node[label=below:$\gamma_0$] at (2.3,0.3) {};
    \node[label=below:$\gamma_1$] at (1.7,2.5) {};
     \node[label=left:$\gamma_2$] at (0.4,2.7) {};
     \draw[->, >=stealth, thick] (B) -- (4,0.9);
     \draw[->, >=stealth, thick] (C) -- (2,1.9);
     \draw[->, >=stealth, thick] (D) -- (0.6,3.2);
\end{tikzpicture}
\end{center}
\caption{Geometric generators associated to a base point $v$}
\end{figure}

Similarly, we consider the case of tangential base point $v$. Without loss of generality, we may assume that $v$ is based at the point $a_0$. Let $v' \in X(\mathbb{C})$ be a point near $a_0$ in the direction $v$. Take a path $\gamma \in \Path(v, v')$ from $v$ to $v'$. Then, we take paths $\gamma_i' \in \Path(v', b_i)$ $(1 \leq i \leq n)$ with the same condition as the former case. Similarly, we may assume that if we take a small loop around $v'$ starting $\gamma$ then we meet successively $\gamma_1',\ldots, \gamma_n'$. We set $\gamma_i \in \Path(b_i, b_i)$ by $\gamma_i:= \gamma_i' \cdot \gamma$ for $i \geq 1$ and $\gamma_0 \in \Path(b_0, b_0)$ as the trivial path from $b_0$ to itself. Then, we define loops based at $v$, $x_0, x_1, \ldots, x_n$, by obvious manner. Finally, we get generator of $\pi_1(X(\mathbb{C}), v)$. We call a generating set chosen in such a way {\it geometric generator} of $X$. 

Let $\tau$ be a geometric generator of $X$, then we have an isomorphism
\begin{equation}
	\tau : F_n \overset{\sim}{\rightarrow} \pi_1(X(\mathbb{C}), v).
\end{equation}

Then, by \cite[XII. Corollarie 5.2]{SGA1}, we have $\Piet(X_{\overline{K}}, v) \simeq \widehat{\pi}_1(X(\mathbb{C}), v)$. By taking maximal pro-$\ell$ quotient of them, $\tau$ induces the isomorphism  
\begin{equation}
	\tau : \Fn \overset{\sim}{\rightarrow} \widehat{\pi}^{(\ell)}_1(X(\mathbb{C}), v) \simeq  \Piet(X_{\overline{K}}, v)^{(\ell)}.
\end{equation}

Since each choice of geometric generator determines such isomorphism $\tau$, we identify a geometric generator and associated isomorphism $\tau : \Fn \overset{\sim}{\rightarrow} \Piet(X_{\overline{K}}, v)^{(\ell)}$. Note that $\tau$ gives a parametrization of $\Piet(X_{\overline{K}}, v)^{(\ell)}$  by (non-commutative) coordinate system $x_1,\ldots, x_n$ of $\Fn$.
\begin{figure}[h]
\begin{center}
\begin{tikzpicture}
	\node[label=below:$v'$] (A) at (0,0) {};
    \node[label=left:$a_0$] (B) at (-2,0){};
    \node[label=right:$a_1$] (C) at (4,0.5){};
    \node[label=right:$a_2$] (D) at (2,3){};
    \draw [fill=black] (A) circle (1pt);
    \draw [fill=black] (B) circle (1pt);
    \draw [fill=black] (C) circle (1pt);
    \draw [fill=black] (D) circle (1pt);
    \draw (A) .. controls (1,0) and (3, 0.2) .. (2.84, 0.2) [arrow inside={end=stealth,opt={scale=1}}{0.5}];
    \draw (A) .. controls (1,0.4) and (1, 2) .. (1.35,2.355) [arrow inside={end=stealth,opt={scale=1}}{0.5}] ;
    \draw (B) .. controls (-0.5, -0.2) and (-0.5, -0.1) .. (A) [arrow inside={end=stealth,opt={scale=1}}{0.5}];
    
    \begin{scope}[xshift=-1.9cm, yshift=-0.02cm, rotate=180]
		\draw (0,0) .. controls (-1,0) and (-1,1) .. (0,1);
		\draw (0,-1) .. controls (1.5,-1) and (1.5,1) .. (0,1) [arrow inside={end=stealth,opt={scale=1}}{0.5}];
		\draw (0,-1) .. controls (-1, -1) and (-1,0) .. (0,0);
	\end{scope}
	\begin{scope}[xshift=3.88cm, yshift=0.46cm, rotate=10]
		\draw (0,0) .. controls (-1,0) and (-1,1) .. (0,1);
		\draw (0,-1) .. controls (1.5,-1) and (1.5,1) .. (0,1) [arrow inside={end=stealth,opt={scale=1}}{0.5}];
		\draw (0,-1) .. controls (-1, -1) and (-1,0) .. (0,0);
	\end{scope}
	\begin{scope}[xshift=1.864cm, yshift=2.864cm, rotate=43]
		\draw (0,0) .. controls (-1,0) and (-1,1) .. (0,1);
		\draw (0,-1) .. controls (1.5,-1) and (1.5,1) .. (0,1) [arrow inside={end=stealth,opt={scale=1}}{0.5}];
		\draw (0,-1) .. controls (-1, -1) and (-1,0) .. (0,0);
	\end{scope}
	\draw[->,>=stealth, thick] (B) -- (-0.7,0);
	\draw[->,>=stealth,thick] (C) -- (2.9,0.2);
	\draw[->,>=stealth, thick] (D) -- (1.4,2.4);
	\node[label=right:$l_1$] at (5,0.9) {};
    \node[label=right:$l_0$] at (-2.4,1.2) {};
     \node[label=right:$l_2$] at (0.5,4) {};
    \node[label=below:$\gamma_1$] at (2,0.2) {};
    \node[label=below:$\gamma_2$] at (1.4,2) {};
     \node[label=left:$\gamma$] at (-0.5,-0.4) {};
      \node[label=above:$v$] at (-0.6,0.1) {};
\end{tikzpicture}	
\end{center}
\caption{Geometric generators associated to a tangential base point $v$}
\end{figure}

Then, let us compute the $G_K$-action on generator $x_1,\ldots, x_n$ for a fixed geometric generator $\tau$ with respect to a tangential base point $b_{\infty}:=b_{0}$  at $a_0 = \infty$. By definition of $x_i$ and $G_K$-action on profinite paths, for any $\sigma \in G_K$ and $1 \leq i \leq n$, we have
\begin{align*}
	\sigma(x_i) & = (\sigma(\gamma_i))^{-1} \sigma(l_i) \sigma(\gamma_i)\\
	& = (\gamma_i^{-1} (\sigma(\gamma_i)))^{-1} (\gamma_i^{-1} \sigma(l_i) \gamma_i) (\gamma_i^{-1} (\sigma(\gamma_i))) \\
	& = y_i(\sigma)^{-1} (\gamma_i^{-1} \sigma(l_i) \gamma_i) y_i(\sigma)
\end{align*}
where we set $y_i(\sigma):= \gamma_i^{-1} \cdot (\sigma(\gamma_i))$. By \cite[XIII. Corollarie 2.12]{SGA1}, we know that $l_i$ generates the inertia group at $a_i$. Therefore, $\sigma(l_i) = l_i^{\chi(\sigma)}$ where $\chi : G_K \rightarrow \mathbb{Z}_{\ell}^{\times}$ is a continuous homomorphism. Moreover, it turns out that this $\chi$ is actually the $\ell$-cyclotomic character $\chi : G_K \rightarrow \mathbb{Z}_{\ell}^{\times}$ (cf.\cite[Observation 3.8]{W3}). Thus, we conclude that the $G_K$-action on $\Fn \overset{\tau}{\simeq} \pi_1(X_{\overline{K}}, b_{\infty})^{(\ell)}$ is given by the form
\begin{equation}
	\sigma(x_i) = y_i(\sigma)^{-1} x_i^{\chi(\sigma)} y_i(\sigma), \quad (1 \leq i \leq n)
\end{equation}
and
\begin{equation} \label{eq:2.8.-1}
	\sigma(x_0) = x_0^{\chi(\sigma)}
\end{equation}
where $y_i(\sigma) = \gamma_i^{-1} \cdot (\sigma(\gamma_i)) \in \Fn$, and $\chi : G_K \rightarrow \mathbb{Z}_{\ell}^{\times}$ is the $\ell$-cyclotomic character defined as $\sigma(\zeta_{\ell^k}) = \zeta_{\ell^k}^{\chi(\sigma)}$ for all $k \geq 1$. Here, $(\zeta_{\ell^k})_{k \geq 1}$ is a fixed system of primitive $\ell^k$-th roots of unity $\zeta_{\ell^k}$ in $\overline{K}$.

\subsection{Magnus embedding}

Here, we recall the notion of Magnus embedding of pro-$\ell$ free groups. For more details, see \cite[Section 6.3]{Mo2}, \cite{MKS} and \cite[Section 5.9]{RZ}.
 
 Let $\mathbb{Z}_{\ell} \langle \langle X_1,\ldots, X_n \rangle \rangle$ be the ring of non-commutative power series, sometimes called a {\it Magnus algebra},  in indeterminates $X_1, \ldots, X_n$ with coefficients in $\mathbb{Z}_{\ell}$. Its multiplicative group of units is denoted by $\mathbb{Z}_{\ell} \langle \langle X_1,\ldots, X_n \rangle \rangle^{\times}$. Let $\llbracket \mathbb{Z}_{\ell} \Fn \rrbracket$ be the complete group algebra of $\Fn$ over $\mathbb{Z}_{\ell}$. We denote by $\epsilon : \llbracket \mathbb{Z}_{\ell} \Fn \rrbracket \rightarrow \mathbb{Z}_{\ell}$ the augmentation map with augmentation ideal  $(\!(I\Fn)\!):=(\!( I_{\mathbb{Z}_{\ell}} \Fn)\!):=\Ker(\epsilon)$.

Then, the {\it pro-$\ell$ Magnus embedding}
\begin{equation} \label{eq:2.8.2}
	\Theta: \Fn \rightarrow \mathbb{Z}_{\ell} \langle \langle X_1, \ldots, X_n \rangle \rangle^{\times} 
\end{equation}
is defined by $x_i \mapsto 1 +X_i$ and $x_i^{-1} \mapsto 1 - X_i + X_i^2 - \cdots$ for $1 \leq i \leq n$. It is known that $\Theta$ actually  gives an injective (continuous) homomorphism. In addition, $\Theta$ extends to an isomorphism of (topological) $\mathbb{Z}_{\ell}$-algebra
\begin{equation} \label{eq:2.8.3}
	\Theta : \llbracket \mathbb{Z}_{\ell} \Fn \rrbracket \overset{\sim}{\rightarrow} \mathbb{Z}_{\ell} \langle \langle X_1, \ldots, X_n \rangle \rangle.
\end{equation}

Next, we recall pro-$\ell$ Magnus coefficients. For $\alpha \in \llbracket \mathbb{Z}_{\ell} \Fn \rrbracket$, the image $\Theta(\alpha)$, called {\it pro-$\ell$ Magnus expansion} of $\alpha$, is given by
\begin{equation} \label{eq:2.8.4}
	\Theta(\alpha) = \epsilon(\alpha) + \sum_{k=1}^{\infty} \sum_{\substack{I=(i_1 \cdots i_k) \\ 1 \leq i_1, \ldots, i_k \leq n}} \mu(I; \alpha) X_{i_1}\cdots X_{i_k}
\end{equation}
where $I=(i_1\cdots i_k)$ denotes multi-index of length $k$ with $1 \leq i_1, \ldots, i_k \leq n$. For a multi-index $I$, the coefficient $\mu(I;\alpha) \in \mathbb{Z}_{\ell}$ is called the {\it pro-$\ell$ Magnus coefficient} of $\alpha$ with respect to $I$.

Note that the pro-$\ell$ Magnus coefficient may be thought of as a map from $\Fn$ to $\mathbb{Z}_{\ell}$ as follows: For a multi-index $I=(i_1 \cdots i_k)$, the map
\begin{equation} \label{eq:2.8.5}
	\mu(I; -) : \Fn \rightarrow \mathbb{Z}_{\ell}
\end{equation}
sends each element $x$ in $\Fn$ to  the coefficient $\mu(I; x)$ of $X_{i_1} \cdots X_{i_k}$ in $\Theta(x)$.

\begin{remark} \label{rem:2.8.6}
(1) For an integer $k \geq 1$, let $\mathbb{Z}_{\ell}\langle \langle X_1, \ldots, X_n \rangle \rangle_{\deg \geq k}$ denote the subring of $\mathbb{Z}_{\ell}\langle \langle X_1, \ldots, X_n \rangle \rangle$ consisting of elements with degree $\geq k$. Then, for $k \geq 1$, $\Theta$ gives identification 
\begin{equation}
	(\!(I \Fn)\!)^k \overset{\sim}{\longrightarrow} \mathbb{Z}_{\ell}\langle \langle X_1, \ldots, X_n \rangle \rangle_{\deg \geq k}.
\end{equation}
(2) For $k \geq 1$, $x \in \Gamma_k(\Fn)$ if and only if the pro-$\ell$ Magnus coefficients $\mu(I; x)$ vanish for all multi-indices $I$ of length $< k$.
\end{remark}

\subsection{$\ell$-adic Milnor invariant}

In \cite{KMT}, we introduced the notion of $\ell$-adic Milnor invariants for Galois element following the analogies between absolute Galois groups and pure braid groups. This section recalls these invariants in our situation.

As in section 2.7, we consider the representation
\begin{equation}
	G_K \rightarrow \Aut(\Piet(X_{\overline{K}}, b_{\infty})^{(\ell)}) \overset{\tau}{\simeq} \Aut(\Fn)
\end{equation} 
which is completely determined by $n$-tuple of pro-$\ell$ words $y_1(\sigma), \ldots, y_n(\sigma)$. As in \cite[Lemma 3.2.1]{KMT}, we can choose $y_i(\sigma)$ such that the coefficient of $[x_i]$ in $[y_i(\sigma)] \in H_1(\Fn; \mathbb{Z}_{\ell})$ is zero ($1 \leq i \leq n)$. In the following, we assume that $y_1(\sigma), \ldots, y_n(\sigma)$ satisfy such condition.

Then, for multi-index $(i_1\cdots i_k i)$, the {\it $\ell$-adic Milnor $\mu$-invariant} $\mu(\sigma; (i_1 \cdots i_ki))$ is defined as
\begin{equation}
	\mu(\sigma; (i_1 \cdots i_ki)):= \mu((i_1 \cdots i_k); y_i(\sigma))
\end{equation}
the pro-$\ell$ Magnus coefficient of $y_i(\sigma)$ with respect to the multi-index $(i_1\cdots i_k i)$.

\begin{remark}
(1) Our definition is different from that of $\cite{KMT}$ where $\ell$-adic Milnor invariants are defined as Magnus coefficients for $y_i(\sigma)^{-1}$ $(1 \leq i \leq n)$ in our notation.
(2) The $\ell$-adic Milnor $\mu$-invariants are considered as invariants for triple $(\sigma, \tau, \{a_1, \ldots, a_n \})$ where  $\sigma$ is an element  in $G_K$, $\tau$ is a geometric generator associated with a tangential base point at $\infty$, $\{a_1, \ldots, a_n \}$ are $K$-rational points in $X$. (3) In \cite{HiMo}, Hirano-Morishita study mod $\ell$ version of $\ell$-adic Milnor invariants. (4) Our $\ell$-adic Milnor invariants are also computable from $\ell$-adic iterated integral by Wojtkowiak (\cite{Wo}).
\end{remark}

\section{Pro-$\ell$ Orr invariants for absolute Galois groups} \label{sec:3}

In \cite{O1}, K. Orr gives homotopy invariants, called {\it Orr invariants}, for a based link in $S^3$ . The invariants take values in the third homotopy group of a space $K_k$, we call {\it Orr space} following \cite{C}, constructed from a free nilpotent group of nilpotency class $k$. The Orr invariant for a based link has all information of Milnor invariants of length $l$ with $k \leq l \leq 2k-1$. In this section, we define a pro-$\ell$ analogue of Orr invariants for some automorphism groups of the pro-$\ell$ free group of rank $n$.

\subsection{Profinite group $G$-action on $\Fn$ and Johnson filtration}

Here, we consider the profinite group $G$-action on $\Fn$ which is a model of the Galois action on \'etale fundamental group of the punctured projective line as in 2.7. In addition, we recall the notion of Johnson filtration of $G$.

Let $G$ be a profinite group which acts on $\Fn$ continuously by\begin{equation}
	g(x_i) = y_i(g)^{-1} x_i^{\chi(g)} y_i(g) \quad (1 \leq i \leq n)
\end{equation}
with
\begin{equation}
	g(x_0) = x_0^{\chi(g)}
\end{equation}
where $\chi: G \rightarrow \mathbb{Z}_{\ell}^{\times}$ be a (continuous) homomorphism, $y_i(g)$ is an element  in $\Fn$ for $1 \leq i \leq n$ and $x_0:=(x_n \cdots x_1)^{-1}$. Then, we can choose $y_i(g)$ uniquely so that the coefficient of $[x_i]$ is 0 in its abelianization $[y_i(g)] \in \Fn/\Gamma_2(\Fn)$ (cf.\cite[Lemma 3.2.1]{KMT}).

Note that by considering the action of the symmetric group of degree $n+1$ on indices of generators $x_1, \ldots, x_n, x_{0}$, we can freely choose an index $i$ such that $G$ acts on $x_i$ by $x_i^{\chi(g)}$. For simplicity, throughout this section, we fix the above choice of generators and $G$-action on them.

Next, we recall the notion of Johnson filtration of $G$ associated with lower central series of $\Fn$. For a non-negative integer $k \geq 0$, we denote by $G[k]$ the subgroup of $G$ consisting of elements that act on $\Fn/\Gamma_{k+1}(\Fn)$ trivially and call it $k$-th Johnson subgroup of $G$. Then, we have the following descending filtration, called the {\it Johnson filtration}, of $G$
\begin{equation}
	G = G[0] \supset G[1] \supset G[2] \supset \cdots \supset G[k] \supset \cdots.
\end{equation}

\begin{remark}
	(1) $G[1]=\ker(\chi) \subset G$. (2) Let $k \geq 1$ be an integer. For $g \in G$, $g \in G[k]$ if and only if $y_i(g) \in \Gamma_k(\Fn)$ for any $1 \leq i \leq n$.
\end{remark}

\subsection{The space $E_{g}$ associated with $G[1]$-action on $\Fn$}

In this section, we define a space $E_{g}$ associated with a profinite group action on $\Fn$ as a pro-space realisation of ``pro-$\ell$ link group'' of $g \in G[1]$ defined in \cite[Section 3.3]{KMT}.

Let $D_n$ denote the $n$-punctured 2-disc. Take a base point $b \in \partial D_n$ and fix once and for all. We denote by $\widehat{D}_{n}^{(\ell)}$ the pro-$\ell$ completion of $D_n$. Note that $\widehat{D}_{n}^{(\ell)}$ is an Eilenberg-MacLane space $K(\Fn, 1)$ of type $(\Fn, 1)$ since $D_n=K(F_n, 1)$ and $F_n$ is ``$\ell$-good'' in the sense of Serre. We note that, up to weak  equivalence, there is no difference between $n$-punctured disk and $n+1$-punctured 2-dimensional sphere. 

To identify $\pi_1(\widehat{D}_{n}^{(\ell)}, b)$ with $\Fn$, we prepare the notion of basing as in \cite{O1}. 

\begin{definition}[Basing]
Let $\mathcal{T}$ denote a set of isomorphisms $\tau : \Fn \rightarrow \pi_1(\widehat{D}_{n}^{(\ell)}, b)$ satisfy the following property:  For any $\tau, \tau' \in \mathcal{T}$ there is the unique isomorphism $ \psi :  \Fn \rightarrow \Fn$ such that
\begin{equation}
	\psi(x_i) = w_i^{-1} x_{i} w_i \quad (1 \leq i \leq n), \quad \psi(x_n\cdots x_1) = x_n\cdots x_1,
\end{equation}
and 
\begin{equation}
	 \tau \circ \psi = \tau'
\end{equation}
for some $w_1, \ldots, w_n \in \Fn$. We call an element $\tau \in \mathcal{T}$ a {\it basing}.
\end{definition}

\begin{example}
Geometric generators associated with a tangential base point at $\infty$ in Section 2.7 are examples of basing.	
\end{example}

By choosing a basing $\tau$, we identify $\pi_1(\widehat{D}_{n}^{(\ell)}, b)$ with $\Fn$, and so $\widehat{D}_{n}^{(\ell)}$ with $K(\Fn, 1)$.  Note that, for any group $G$, the group of homotopy classes of homotopy equivalences of $K(G, 1)$ is canonically identified with automorphism group of $G$. Thus, corresponding to the $G[1]$-action on $\Fn$, there is the  homotopy equivalence $\varphi_g : K(\Fn, 1) \rightarrow K(\Fn ,1)$ for any $g \in G[1]$. In particular, we obtain the homotopy equivalence
\begin{equation}
	\varphi_g : \widehat{D}_{n}^{(\ell)} \rightarrow \widehat{D}_{n}^{(\ell)}.
\end{equation} 

More concretely, $\varphi_g$ is given as an inverse limit $\{(\varphi_g)_N\}_{N \in \mathcal{N}}$ corresponding to $\widehat{D}_{n}^{(\ell)}=\{(\widehat{D}_{n}^{(\ell)})_N\}_{N \in \mathcal{N}}$. Here, $(\varphi_g)_N : (\widehat{D}_{n}^{(\ell)})_N \rightarrow (\widehat{D}_{n}^{(\ell)})_N$ is a class of homotopy equivalence and $\mathcal{N}$ denotes the set of open normal subgroups of $\Fn$. Then, for each $N \in \mathcal{N}$, we  consider the mapping torus associated to $(\varphi_{g})_N$. The resulting pro-space of mapping tori is denoted by $M_{g}$. Note that $\pi_1(M_{g})$ is ``good'' in the sense of Serre by homotopy exact sequence of fiber space and \cite[(d) page 16]{Ser2}. Thus, its pro-$\ell$ completion denoted by $\widehat{M}_{g}^{(\ell)}$ is an Eilenberg-MacLane space of type $(\pi, 1)$. 

Similarly, we glue $D^2 \times \widehat{S}^1{}^{(\ell)}$ with $M_{g}$ as follows. There is an inclusion $\iota : S^1 \times \widehat{S}^1{}^{(\ell)}  \hookrightarrow D^2 \times \widehat{S}^1{}^{(\ell)}$ induced by the canonical inclusion $S^1 \rightarrow \partial D^2$ and the identity map of $\widehat{S}^1{}^{(\ell)}$. Let $\eta : \widehat{S}^1{}^{(\ell)} \rightarrow \widehat{D}_{n}^{(\ell)}$ be the continuous map induced from the inclusion $S^1 \hookrightarrow \partial D_n$. By using $\eta$, we consider a continuous map $\phi : [0,1] \times \widehat{S}^1{}^{(\ell)} \rightarrow \widehat{D}_{n}^{(\ell)}  \times [0,1]$ which sends $(x,y)$ to $(\eta(y), x)$. Composition of $\phi$ and the canonical projection $\widehat{D}_{n}^{(\ell)}  \times [0,1] \rightarrow M_{g}$ induces a continuous map $\phi: S^1 \times \widehat{S}^1{}^{(\ell)} \rightarrow M_{g}$ denoted by the same $\phi$ by abuse of notation. Then, we define $E_{g}$ as the (homotopy) pushout of $\phi: S^1 \times \widehat{S}^1{}^{(\ell)}  \rightarrow M_{g}$ and  $\iota : S^1 \times \widehat{S}^1{}^{(\ell)} \hookrightarrow  D^2 \times \widehat{S}^1{}^{(\ell)}$:

\begin{center}
		\begin{tikzcd}
		 S^1 \times \widehat{S}^1{}^{(\ell)} \arrow[d, "\phi"] \arrow[r, "\iota"] 
		& D^2 \times \widehat{S}^1{}^{(\ell)} \arrow[d]
  \\
		M_{g}  \arrow[r]
		& M_{g} \cup_{\phi} (D^2 \times \widehat{S}^1{}^{(\ell)}) =:E_{g}
		\end{tikzcd}
	\end{center} 
Here, note that the pro-spaces $E_g$ that appear in the above commutative diagram is homotopy equivalent to the pro-space obtained as the result of the (homotopy) pushout.	
	
By construction, we see that the fundamental group of  $E_{g}$ has the following presentation
\begin{equation} \label{eq:4.1.9}
		\pi_1(E_{g}) \overset{\tau}{\simeq} \langle x_1, \ldots, x_n \ | \ [x_1, y_1(g)] = \cdots [x_n, y_n(g)]=1 \rangle
\end{equation}
In particular, for each integer $k \geq 1$, the $k$-th nilpotent quotient of $\pi_1(E_{g})$ has the following presentation 
\begin{equation}
	\pi_1(E_{g})/ \Gamma_{k}(\pi_1(E_{g})) \overset{\tau}{\simeq} \langle x_1, \ldots, x_n \ | \ [x_1, y_1(g)] = \cdots [x_n, y_n(g)]=1, \Gamma_{k}(\Fn) \rangle
\end{equation}

Denote by $E$ the set of (homotopy equivalence classes of) $E_{g}$ for any $g \in G[1]$ constructed as above. To make its dependency on the choice of the basing $\tau$ explicit, we also denote it by $(E_g, \tau)$. Note that in this expression $(E_g, \tau)$, $g$ should be understood as an element in $\Aut(\Fn)$ acting on generators of $\Fn$ determined by basing $\tau$.  Then, we define a product of two (homotopy equivalence classes of) pro-spaces $(E_{g_1}, \tau_1), (E_{g_2}, \tau_2) \in E$ for $g_1, g_2 \in G[1] \subset \Aut(\Fn)$. First observe that we may define a product of $(E_{g_1}, \tau_1), (E_{g_2}, \tau_2)$ as $(E_{g_1 g_2}, \tau_1)$ when $\tau_2 = \tau_1$ by gluing $\tau_2$ with $\tau_1 \circ g_1$ as in the case of braid groups. For general case, we set
\begin{equation}
	(E_{g_1}, \tau_1) \circ (E_{g_2}, \tau_2) := (E_{g_1}, \tau_1) \circ_{\tau_1 \circ g_1 \sim \tau_2} (E_{g_2}, \tau_2) := (E_{g_1 \cdot({}^{\psi_{12}}g_2) }, \tau_1).
\end{equation}
Here, we glue two basing $\tau_1 \circ g_1$ and $\tau_2$ in terms of the automorphism $\psi_{12}$ of $\Fn$ such that $\tau_2 =  \tau_1 \circ \psi_{12}$ and we set ${}^{\psi_{12}}g_2:=\psi_{12} \circ  g_2 \circ \psi_{12}^{-1}$. We first show the associativity of the product $\circ$.

\begin{lemma}
The above product $\circ$ is associative.	
\end{lemma}
\begin{proof}
By direct computation, we get
\begin{equation} \label{eq:0.0.0}
	((E_{g_1}, \tau_1) \circ (E_{g_2}, \tau_2)) \circ (E_{g_3}, \tau_3) = (E_{g_1\cdot ({}^{\psi_{12}}g_2) \cdot ({}^{\psi_{13}}g_3)}, \tau_1)
\end{equation}
and 
\begin{equation} \label{eq:0.0.1}
	(E_{g_1}, \tau_1) \circ ((E_{g_2}, \tau_2) \circ (E_{g_3}, \tau_3)) = (E_{g_1 \cdot {}^{\psi_{12}}(g_2 \cdot {}^{\psi_{23}}g_3)}, \tau_1).
\end{equation}
By definition of basing, one sees that $ \psi_{12} \cdot \psi_{23} =\psi_{13}$. In terms of this equation, we get $\eqref{eq:0.0.0}=\eqref{eq:0.0.1}$. Thus, the product $\circ$ is associative.
\end{proof}

If we vary basing, the above product $\circ$ only satisfy associativity. However, if we fix a basing $\tau$, the above product $\circ$ endow the set $E(\tau)$ consisting of  $(E_g, \tau)$ for all $g \in G$ the structure of group as follows.
\begin{lemma}
	Let $\tau$ be a fixed basing. The above $\circ$ gives a group structure on $E(\tau)$. In particular, the assignment
	\begin{equation}
		G[1] \rightarrow E(\tau), \quad g \mapsto (E_{g}, \tau)
	\end{equation}
	is a group homomorphism.
\end{lemma}

\begin{proof}
	For the unit element $e \in G[1]$ and for any $g \in G[1]$, we have
	\begin{equation}
		(E_{g}, \tau) \circ (E_{e}, \tau) = (E_{ge}, \tau) = (E_{g}, \tau) = (E_{eg}, \tau) = (E_{e},\tau) \circ (E_{g}, \tau).
	\end{equation}
	Hence, $(E_{e}, \tau)$ is the unit in $E(\tau)$.
	For any $g \in G[1]$, we have
	\begin{equation}
		(E_{g}, \tau) \circ (E_{g^{-1}},\tau) = (E_{e}, \tau) = (E_{g^{-1}}, \tau) \circ (E_{g}, \tau)
	\end{equation}
	so $(E_{g^{-1}}, \tau)$ is the inverse element of $(E_{g}, \tau)$. We know that $\circ$ is associative. Thus, $(E(\tau), \circ)$ forms a group.
\end{proof}

\begin{remark}
	The above definition of the product of $E(\tau)$ corresponds to the product structure of the braid groups (with canonical basing) if we consider that $G[1]$ is a profinite analogue of the pure braid group as in \cite{KMT}.
\end{remark}

We complete this section by computing the homology group of $E_{g}$ with coefficients in $\mathbb{Z}_{\ell}$.

\begin{proposition}
	The third homology group $H_3(E_g; \mathbb{Z}_{\ell})$ is isomorphic to $\mathbb{Z}_{\ell}$ and the higher homology group $H_i(E_g;\mathbb{Z}_{\ell})=0$ for $i > 3$.
\end{proposition}
\begin{proof}
To begin with, we compute the homology groups of $M_g$. By considering Mayer-Vietoris sequence for each index $N$
\begin{equation*}
\cdots \rightarrow H_{\ast}((\widehat{D}_n^{(\ell)})_N; \mathbb{Z}_{\ell})  \rightarrow H_{\ast}((\widehat{D}_n^{(\ell)})_N; \mathbb{Z}_{\ell}) \oplus H_{\ast}((\widehat{D}_n^{(\ell)})_N; \mathbb{Z}_{\ell}) \rightarrow H_{\ast}((M_g)_N; \mathbb{Z}_{\ell}) \rightarrow \cdots	
\end{equation*}
and taking inverse limit $\varprojlim$, we see that $H_k(M_g; \mathbb{Z}_{\ell})=0$ for $k \geq 3$ since $H_k(\widehat{D}_n^{(\ell)}; \mathbb{Z}_{\ell})=H_k(\Fn;\mathbb{Z}_{\ell})=0$ for $k \geq 2$. Here, note that $\varprojlim$ is an exact functor on the category of compact $\mathbb{Z}_{\ell}$-modules. Since $M_g$ is $K(\pi_1(M_g),1)$ space, we know that $H_1(M_g; \mathbb{Z}_{\ell})\simeq \mathbb{Z}_{\ell}^{\oplus n+1}$ and $H_2(M_g; \mathbb{Z}_{\ell})\simeq \mathbb{Z}_{\ell}^{\oplus n}$ by group presentation of $\pi_1(M_g)$ and Hopf isomorphism. 

Next, we show that $H_3(M_g, S^1 \times \widehat{S}^1{}^{(\ell)}; \mathbb{Z}_{\ell}) \simeq \mathbb{Z}_{\ell}$. Note that $H_i(\Fn; \mathbb{Z}_{\ell})=0$ for $i \geq 2$. Thus, $H_3(M_g, S^1 \times \widehat{S}^1{}^{(\ell)}; \mathbb{Z}_{\ell}) \simeq H_3(M_g, S^1 \times \widehat{S}^1{}^{(\ell)} \cup \vee_{i=1}^n  \widehat{S}^1{}^{(\ell)}; \mathbb{Z}_{\ell}) \simeq H_3(D^2 \times S^1, S^1 \times S^1; \mathbb{Z}_{\ell})\simeq \mathbb{Z}_{\ell}$. Here, we use the canonical inclusion $\vee_{i=1}^n \widehat{S}^1{}^{(\ell)} \rightarrow \widehat{D}_n^{(\ell)} \subset M_g$ given by a chosen basing. Therefore, we see that $H_3(M_g, S^1 \times \widehat{S}^1{}^{(\ell)}; \mathbb{Z}_{\ell}) \simeq H_2(S^1 \times \widehat{S}^1{}^{(\ell)}; \mathbb{Z}_{\ell}) \simeq \mathbb{Z}_{\ell}$.

 We turn to computation of  $H_i(E_g)$. For $i \geq 3$, consider Mayer-Vietoris sequence for each index $N$ of pro-spaces
 \begin{equation*}
 	\cdots \rightarrow H_{\ast}((S^1 \times \widehat{S}^1{}^{(\ell)})_N; \mathbb{Z}_{\ell}) \rightarrow H_{\ast}((M_g)_N; \mathbb{Z}_{\ell}) \oplus H_{\ast}((D^2 \times \widehat{S}^1{}^{(\ell)})_N; \mathbb{Z}_{\ell}) \rightarrow H_{\ast}((E_g)_N; \mathbb{Z}_{\ell}) \rightarrow \cdots
 \end{equation*}
and take inverse limit $\varprojlim$. As mentioned above, note that $\varprojlim$ is an exact functor in this case. Since $H_2(S^1 \times \widehat{S}^1{}^{(\ell)}; \mathbb{Z}_{\ell}) \simeq \mathbb{Z}_{\ell} \rightarrow H_2(M_g; \mathbb{Z}_{\ell})$ has zero image from above computation, we have $H_3(E_g; \mathbb{Z}_{\ell})\simeq \mathbb{Z}_{\ell}$. Since $H_i(\Fn;\mathbb{Z}_{\ell})=0$ for $i \geq 2$ and $H_i(M_g; \mathbb{Z}_{\ell})=0$ for $i \geq 3$, we have $H_i(E_g; \mathbb{Z}_{\ell})=0$ for $i > 3$.
\end{proof}

\subsection{Pro-$\ell$ Orr space}

In this section, we define a pro-$\ell$ analogue of Orr space $K_k$ which is constructed from a finite rank free nilpotent group of nilpotency class $k$.

Let $\Fn$ be the pro-$\ell$ free group of rank $n$ freely (topologically) generated by $x_1, \ldots, x_n$ and $\Gamma_k(\Fn)$ be the $k$-th lower central subgroup of $\Fn$. We set \begin{equation} \label{eq:4.2.1}
 	K^{(\ell)}(k):= K(\Fn/\Gamma_k(\Fn), 1)
 \end{equation}
 an Eilenberg-MacLane pro-space of type $(\Fn/\Gamma_k(\Fn), 1)$. Let us denote by $\widehat{S}^1{}^{(\ell)}$ the pro-$\ell$ completion of $S^1$. Note that  $\widehat{S}^1{}^{(\ell)} \sim \widehat{K(\mathbb{Z}, 1)}^{(\ell)} \sim K(\mathbb{Z}_{\ell}, 1)$ an Eilenberg-MacLane pro-space of type $(\mathbb{Z}_{\ell}, 1)$ (cf. Example \ref{ex:2.3.1} (i)). 

Then, noting that the natural surjection $\Fn \rightarrow \Fn / \Gamma_{k}(\Fn)$ induces the inclusion $\vee \widehat{S}^1{}^{(\ell)} \rightarrow K^{(\ell)}(k)$, we give the following definition.

\begin{definition}[pro-$\ell$ Orr space] \label{def:orr}
	Let $K^{(\ell)}_k$ be the mapping cone associated to the inclusion $\vee \widehat{S}^1{}^{(\ell)} \rightarrow K^{(\ell)}(k)$. Then, we call $K^{(\ell)}_k$ the {\it pro-$\ell$ Orr space}. Let $\widehat{K}_k^{(\ell)}$ denote the pro-$\ell$ completion of $K_k^{(\ell)}$ and we call it the {\it pro-$\ell$ complete Orr space}.
	\end{definition}

Then, by definition, we immediately see the following.
\begin{lemma} \label{lem:4.2.1} 	 The pro-$\ell$ Orr space $K_k^{(\ell)}$ is simply connected pro-space. In particular, $\pi_1(K_k^{(\ell)})=0$.\\
\end{lemma}

In terms of result of Section 6.1, we get the following proposition.
\begin{proposition} \label{lem:4.2.2}
	Let $K_k^{(\ell)}$ be a pro-$\ell$ Orr space and $\widehat{K}_k^{(\ell)}$ be a pro-$\ell$ complete Orr space. Then, following statements hold:
	\begin{enumerate}[label=$(\arabic{enumi})$]
		\item $\pi_1(K_k^{(\ell)})=\pi_1(\widehat{K}_k^{(\ell)})=0$,
		\item $\pi_2(K_k^{(\ell)}) \simeq  \pi_2(\widehat{K}_k^{(\ell)}) \simeq \mathbb{Z}_{\ell}^{\oplus  N_k}$,
		\item $\widehat{\pi}^{(\ell)}_3(K_k^{(\ell)}) \simeq \pi_3(\widehat{K}_k^{(\ell)})$.
	 \end{enumerate}
	 Here, $N_k$ is an integer defined in \eqref{eq:5.1.4}.
\end{proposition}
\begin{proof}
(1) is a consequence of Lemma \ref{lem:4.2.1}. \\
(2) By the proof of Lemma \ref{eq:7.1.8}, we know that $\pi_2(K_k^{(\ell)}) \simeq \mathbb{Z}_{\ell}^{\oplus  N_k}$, so $\widehat{\pi}^{(\ell)}_2(K_k^{(\ell)}) \simeq \mathbb{Z}_{\ell}^{\oplus N_k} \simeq \pi_2(K_k^{(\ell)})$. Since $K_k^{(\ell)}$ is simply connected by Lemma \ref{lem:4.2.1}, we get $\widehat{\pi}_2^{(\ell)}(K_k^{(\ell)}) \simeq \pi_2(\widehat{K}_k^{(\ell)})$ by  \cite[Corollary (6.2)]{AM}. Thus, there is a desired isomorphism. \\
(3) Since $K_k^{(\ell)}$ is simply connected and $\pi_2(K_k^{(\ell)})$ is ``good'' in the sense of Serre, we get $\widehat{\pi}^{(\ell)}_3(K_k^{(\ell)}) \simeq \pi_3(\widehat{K}_k^{(\ell)})$ by applying \cite[Theorem (6.7)]{AM}.
\end{proof}

\subsection{Construction of pro-$\ell$ Orr invariants}

Now, we are in the position to define a pro-$\ell$ Orr invariant for $g \in G[k]$. From now on, we take an element of $k$-th Johnson subgroup $G[k]$ for some integer $k \geq 1$.  Then, $y_1(g), \ldots, y_n(g)$ lie on  $\Gamma_k(\Fn)$. For each $1 \leq i \leq n$, we denote by $\widehat{L}_i^{(\ell)}(g)$ a closed path in $\widehat{D}_{n}^{(\ell)}$  which represents the pro-$\ell$ word $y_i(g)$ in $\pi_1(\widehat{D}_{n}^{(\ell)}, b) \overset{\tau}{\simeq} \Fn$. Since $y_i(g) \in \Gamma_k(\Fn)$ for all $i$, there is a composition map
\begin{equation} \label{eq:4.3.1}
	\bigsqcup_{i=1}^n \widehat{L}^{(\ell)}_i(g) \rightarrow E_{g} \rightarrow K(\pi_1(E_{g})/\Gamma_k(\pi_1(E_{g}), 1)
\end{equation}
whose image is nullhomotopic. By basing $\tau$, we have an induced isomorphism
\begin{equation} \label{eq:4.3.2}
	\tau : \Fn / \Gamma_k(\Fn) \overset{\sim}{\longrightarrow} \pi_1(E_{g}) / \Gamma_k(\pi_1(E_{g})).
\end{equation}
Thus, we get a map
\begin{equation} \label{eq:4.3.3}
	\bigvee^n \widehat{S}^1{}^{(\ell)} \overset{\tau}{\rightarrow} K(\pi_1(E_{g}) / \Gamma_k(\pi_1(E_{g}), 1) \simeq K^{(\ell)}(k).
\end{equation}

By homotopy extension theorem, we have a following commutative diagram.
\begin{center}
		\begin{tikzcd}
		 \bigsqcup_{i=1}^n (\widehat{L}^{(\ell)}_i(g) \times \widehat{S}^1{}^{(\ell)}) \arrow[d] \arrow[r, "proj"] 
		& \bigvee^n \widehat{S}^1{}^{(\ell)} \arrow[d]
  \\
		E_{g}  \arrow[r]
		& K^{(\ell)}(k)
		\end{tikzcd}
	\end{center} 
where the vertical map is an inclusion corresponding to relations $[x_i, y_i(g)]$ $(1 \leq i \leq n)$ of $\pi_1(E_{g})$ and the $proj$ sends each $\widehat{S}^1{}^{(\ell)}$ of $\widehat{L}_i(g) \times \widehat{S}^1{}^{(\ell)}$ to the $i$-th summand of $\bigvee^n \widehat{S}^1{}^{(\ell)}$. 

\begin{lemma} \label{lem:4.3.6}
Consider  the composition 
\begin{equation}
	E_{g} \rightarrow K^{(\ell)}(k) \rightarrow K_k^{(\ell)}.
\end{equation}
Then, the above map induces a canonical map
\begin{equation} \label{eq:4.3.4}
	\rho^{(\ell)} : S^3 \longrightarrow K_k^{(\ell)}(\rightarrow \widehat{K}_k^{(\ell)}).
\end{equation}
\end{lemma}

\begin{proof}
By considering the completion map $S^1 \rightarrow \widehat{S}^1{}^{(\ell)}$, we obtain the following commutative diagram of homotopy classes of continuous maps inside of the pro-$\ell$ Orr space $K_k^{(\ell)}$:

\begin{center}
\begin{tikzcd}[row sep=scriptsize, column sep=scriptsize]
&  S^1 \times S^1  \arrow[dr] \arrow[rr] \arrow[dd] & & D^2 \times S^1 \arrow[dr] \arrow[dd] \\
 & & D^2 \times S^1 \arrow[rr, crossing over]& & (D^2 \times S^1) \cup_{\phi} (D^2 \times S^1) \simeq  S^3 \arrow[dd] \\
& S^1 \times \widehat{S}^1{}^{(\ell)} \arrow[dr] \arrow[rr] & & D^2 \times \widehat{S}^1{}^{(\ell)} \arrow[dr] \\
 & & M_{g}  \arrow[from=uu, crossing over] \arrow[rr]&  & E_{g}
\end{tikzcd}	
\end{center}
Here, the bottom commutative diagram is the fiber coproduct diagram defining $E_{g}$ and the top commutative diagram is the corresponding fiber coproduct diagram which gives $S^3$. Note that $\phi$ and $\iota$ induces corresponding maps in the top diagram. In fact,  $\phi$ and $\iota$ are defined by the combination of the identity map and induced map from $S^1 \rightarrow \partial D^2$, so it induces the map in the top diagram.  Also note that the above diagram is valid if it is  considered under the map $E_{g} \rightarrow K_k^{(\ell)}$. Thus, by homotopy,  we get a desired canonical map $S^3 \rightarrow K_k^{(\ell)} (\rightarrow \widehat{K}_k^{(\ell)})$.
\end{proof}

Then, we set
\begin{equation} \label{eq:4.3.5}
	\theta_k^{(\ell)}(g, \tau) := [\rho^{(\ell)}] \in \widehat{\pi}^{(\ell)}_3(K_k^{(\ell)}) \simeq \pi_3(K_k^{(\ell)})\otimes_{\mathbb{Z}} \mathbb{Z}_{\ell} \simeq \pi_3(\widehat{K}_k^{(\ell)})
\end{equation}

Since, up to homotopy equivalence, there is no difference in choice of $E_{g}$, the homotopy class $\theta_k^{(\ell)}(g, \tau)$ is well-defined. This leads to the following definition.
\begin{definition}[Pro-$\ell$ Orr invariant] \label{def:oinv}
	Let $k \geq 1$ be a fixed integer. For a pair of $g \in G[k]$ and  basing $\tau \in \mathcal{T}$, the element  $\theta_k^{(\ell)}(g, \tau)\in \pi_3(\widehat{K}_k^{(\ell)})$ is called the {\it pro-$\ell$ Orr invariant} for $(g, \tau)$. In particular, when $G=G_K$ and $\tau$ is a geometric generator associated with a tangential base point at $\infty$ as in Section 2.7, we call $\theta_k^{(\ell)}(\sigma, \tau)$ the {\it pro-$\ell$ Orr invariants for a based Galois element} $(\sigma, \tau)$.
\end{definition}

\begin{remark}
(1) The pro-$\ell$ Orr invariants for a based Galois element also depends on the choices of the $K$-rational points $\{a_1, \ldots, a_n\}$ of $X$.\\
(2) Note that similarly, we can define profinite and pro-$\Sigma$ analogues of Orr invariants for based Galois elements. Here, $\Sigma$ is a set of (rational) prime numbers.
\end{remark}

\begin{theorem}
The map 
\begin{equation}
	\theta_k^{(\ell)} : G[k] \times \mathcal{T} \rightarrow \pi_3(\widehat{K}_k^{(\ell)})
\end{equation}	
is additive under the product of $(E_g, \tau)$ as the following manner:
\begin{equation}
	\theta_k^{(\ell)}((g_1, \tau_1) \circ (g_2, \tau_2)) = \theta_k^{(\ell)}(g_1, \tau_1) + \theta_k^{(\ell)}(g_2, \tau_2)
\end{equation}
for $(g_1, \tau_1)$ and $(g_2, \tau_2)$ in $G[k] \times \mathcal{T}$. Here, $\theta_k^{(\ell)}((g_1, \tau_1) \circ (g_2, \tau_2))$ means the pro-$\ell$ Orr invariant obtained from $(E_{g_1}, \tau_1) \circ (E_{g_2}, \tau_2)$.
\end{theorem}

\begin{proof}
In \cite[Theorem 8]{O1}, Orr gives geometric proof of additivity of his invariants under the connected sum of based links. Here, we give algebraic proof via amalgamated free products of fundamental groups and property of Johnson subgroup $G[k]$.

To begin with, recall that the $\theta_k^{(\ell)}$ invariant of $(E_{g_1}, \tau_1) \circ (E_{g_2}, \tau_2)$ is induced from the following commutative diagram:
\begin{center}
		\begin{tikzcd}
		 \bigsqcup_{i=1}^n (\widehat{L}^{(\ell)}_i(g_1 \cdot ({}^{\psi_{12}}g_2)) \times \widehat{S}^1{}^{(\ell)}) \arrow[d] \arrow[r, "proj"] 
		& \bigvee^n \widehat{S}^1{}^{(\ell)} \arrow[d]
  \\
		(E_{g_1}, \tau_1) \circ (E_{g_2}, \tau_2)  \arrow[r]
		& K^{(\ell)}(k)	
	\end{tikzcd}
\end{center}
where $\psi_{12} \in \Aut(\Fn)$ is an automorphism such that $\tau_2 = \tau_1 \circ \psi_{12}$. Note that, modulo $k$-th lower central subgroups, we have an isomorphism
\begin{equation}
	\pi_1((E_{g_1}, \tau_1)) \ast_{\pi_1(\vee^n \widehat{S}^1{}^{(\ell)})} \pi_1(E_{g_2}, \tau_2) \simeq \pi_1((E_{g_1}, \tau_1) \circ (E_{g_2}, \tau_2)).
\end{equation}
Here, the amalgamated product is induced by homomorphisms, for $j=1,2$,
\begin{equation}
	I_j : \pi_1(\vee^n \widehat{S}^1{}^{(\ell)}) \rightarrow \pi_1((E_{g_j}, \tau_j))
\end{equation}
so that $I_1 = \tau_1 \circ g_1$ and $I_2 = \tau_2$. Also note that
\begin{equation}
	K^{(\ell)}(k) \simeq K(\pi_1((E_{g_1}, \tau_1) \circ (E_{g_2}, \tau_2))/\Gamma_k(\pi_1((E_{g_1}, \tau_1) \circ (E_{g_2}, \tau_2))), 1).
\end{equation}
Since $K_k^{(\ell)}$ is roughly a quotient space $K^{(\ell)}(k)/\vee^n \widehat{S}^1{}^{(\ell)}$ and $\vee^n \widehat{S}^1{}^{(\ell)}$ shrink to one point in $K_k^{(\ell)}$, the above commutative diagram factor through
\begin{equation}
	K(\pi_1(E_{g_1}) \ast \pi_1(E_{g_2}), 1) = K(\pi_1(E_{g_1}), 1) \vee K(\pi_1(E_{g_2}), 1) \rightarrow K^{(\ell)}(k) \vee K^{(\ell)}(k) \rightarrow K^{(\ell)}(k) 
\end{equation}
where the last map is folding map. Therefore, corresponding to this factorisation, the above commutative diagram also factorise into
\begin{center}
		\begin{tikzcd}
		 \bigsqcup_{i=1}^n (\widehat{L}^{(\ell)}_i(g_1) \times \widehat{S}^1{}^{(\ell)}) \vee \bigsqcup_{i=1}^n (\widehat{L}^{(\ell)}_i({}^{\psi_{12}}g_2) \times \widehat{S}^1{}^{(\ell)})\arrow[d] \arrow[r, "proj"] 
		& (\bigvee^n \widehat{S}^1{}^{(\ell)}) \vee (\bigvee^n \widehat{S}^1{}^{(\ell)}) \arrow[d]
  \\
		(E_{g_1}, \tau_1) \vee  (E_{{}^{\psi_{12}}g_2}, \tau_1 \circ g_1)  \arrow[r]
		& K^{(\ell)}(k) \vee K^{(\ell)}(k) \rightarrow K^{(\ell)}(k).	
	\end{tikzcd}
\end{center} 
Thus, we conclude that 
\begin{equation}
	\theta_k^{(\ell)}((g_1, \tau_1) \circ  (g_2, \tau_2)) = \theta_k^{(\ell)}(g_1, \tau_1) + \theta_k^{(\ell)}({}^{\psi_{12}}g_2, \tau_1 \circ g_1)).
\end{equation}
To complete the proof, we need to show that 
\begin{equation}
	\theta_k^{(\ell)}({}^{\psi_{12}}g_2, \tau_1 \circ g_1)) = \theta_k^{(\ell)}(g_2, \tau_2).
\end{equation}
Since $\pi_3(K_k^{(\ell)})$ is an abelian group, as automorphism group of $\Fn$, $G$ factor through its maximal abelian quotient. It suffices to consider the part of $\pi_1(E_{g})/[\Gamma_{k+1}(\pi_1(E_{g})), \Gamma_{k+1}(\pi_1(E_{g}))]$ in our definition of $\theta_k^{(\ell)}$, Since $G[k]$ acts on $\Fn/\Gamma_{k+1}(\Fn)$ trivially. Then, $g_1$ action does not affect on  the words of $y_1({}^{\psi_{12}}g_2), \ldots, y_n({}^{\psi_{12}}g_2)$ in $\pi_1(E_{{}^{\psi_{12}}g_2})/[\Gamma_{k+1}(\pi_1(E_{{}^{\psi_{12}}g_2})), \Gamma_{k+1}(\pi_1(E_{{}^{\psi_{12}}g_2}))])$ since $g_1 \in G[k]$. This means that $\theta_k^{(\ell)}({}^{\psi_{12}}g_2, \tau_1 \circ g_1) = \theta_k^{(\ell)}({}^{\psi_{12}}g_2, \tau_1)$. Finally, we note that
\begin{equation}
	\theta_k^{(\ell)}({}^{\psi_{12}}g_2, \tau_1) = \theta_k^{(\ell)}(g_2, \tau_2).
\end{equation}
Therefore, we have shown the desired additivity.
\end{proof}

By applying the proof of the above theorem, we get the following corollary.
\begin{corollary}
	Let $\tau$ be a fixed basing. Then, the map
	\begin{equation}
		\theta_k^{(\ell)}(-, \tau) : G[k] \rightarrow \pi_3(\widehat{K}_k^{(\ell)})
	\end{equation}
	is additive under the product of $(E_g, \tau)$ as the following manner:
	\begin{equation}
		\theta_k^{(\ell)}(g_1 g_2, \tau) = \theta_k^{(\ell)}(g_1, \tau) + \theta_k^{(\ell)}(g_2, \tau)
	\end{equation}
	for $g_1, g_2 \in G[k]$.
\end{corollary}

\subsection{Image of $\theta_k^{(\ell)}(g,\tau)$ under the Hurewicz homomorphism}

This section considers the image of pro-$\ell$ Orr invariant $\theta_k^{(\ell)}(g,\tau)$ under the Hurewicz homomorphism. Note that, as in usual classical algebraic topology, the pro-space analogue of the Hurewicz theorem is established by Artin-Mazur (\cite[Corollary (4.5)]{AM}).

Recall that the Hurewicz homomorphism  $h_{\ast} : \pi_3(K_k^{(\ell)}) \rightarrow H_3(K_k^{(\ell)})$ sends $\rho \in  \pi_3(K_k^{(\ell)})$ to $\rho_{\ast}([S^3]) \in H_3(K_k^{(\ell)})$ for the fundamental class $[S^3] \in H_3(S^3;\mathbb{Z})$.

\begin{lemma} \label{lem:4.4.2}
Let $h_{\ast} : \pi_3(K_k^{(\ell)}) \rightarrow H_3(K_k^{(\ell)})$ be the Hurewicz homomorphism. Then, the followings hold:
\begin{enumerate}[label=$(\arabic{enumi})$]
	\item $h_{\ast}$ is surjective homomorphism.
	\item $H_3(K_k^{(\ell)};\mathbb{Z}_{\ell}) \simeq H_3(\Fn/\Gamma_k(\Fn); \mathbb{Z}_{\ell})$.
\end{enumerate}
\end{lemma}

\begin{proof}
	(1) Recall that for any path-connected 1-connected space $X$, the Hurewicz homomorphism gives a surjection $\pi_3(X) \rightarrow H_3(X)$. Since $K_k^{(\ell)}$ is simply connected, the assertion follows by applying the proof of \cite[Corollary (4.5)]{AM}.\\
	(2) Note that $\mathrm{Tor}_1^{\mathbb{Z}}(-, \mathbb{Z}_{\ell})=0$ since $\mathbb{Z}_{\ell}$ is flat $\mathbb{Z}$-module. Hence, by applying universal coefficient theorem for each index of pro-spaces, we get $H_3(K_k^{(\ell)})\otimes_{\mathbb{Z}} \mathbb{Z}_{\ell} \simeq H_3(K_k^{(\ell)}; \mathbb{Z}_{\ell})$. By homology exact sequence for pair $(K(\Fn/\Gamma_{k}(\Fn),1), \vee \widehat{S}^1{}^{(\ell)})$, we conclude that $H_3(K_k^{(\ell)}; \mathbb{Z}_{\ell}) \simeq H_3(\Fn/\Gamma_k(\Fn); \mathbb{Z}_{\ell})$ since $H_i(\Fn; \mathbb{Z}_{\ell})=0$ for any $i \geq 2$.
\end{proof}

By Lemma \ref{lem:4.4.2}, we denote  the Hurewicz homomorphism as $h_{\ast}:=h_{\ast} \otimes \Id : \pi_3(K_k^{(\ell)})\otimes \mathbb{Z}_{\ell}=\pi_3(\widehat{K}_k^{(\ell)}) \rightarrow H_3(\Fn/\Gamma_k(\Fn);\mathbb{Z}_{\ell})$.

\begin{definition} \label{def:4.4.3}
Let $h_{\ast} : \pi_3(\widehat{K}_k^{(\ell)}) \rightarrow H_3(\Fn/\Gamma_k(\Fn);\mathbb{Z}_{\ell})$ be the Hurewicz homomorphism. Then, we set 
\begin{equation}
	\tau_k^{(\ell)}(g, \tau):=h_{\ast}(\theta_k^{(\ell)}(g,\tau)) \in H_3(\Fn/\Gamma_{k}(\Fn); \mathbb{Z}_{\ell})
\end{equation}	
\end{definition}

\begin{proposition} \label{prop:4.4.4}
Let $[E_{g}]$ denote a topological generator of $H_3(E_{g}; \mathbb{Z}_{\ell})\simeq \mathbb{Z}_{\ell}$. Then, the image of $[E_{g}]$ under the composition homomorphisms
\begin{equation}
	H_3(E_{g}; \mathbb{Z}_{\ell}) \rightarrow H_3(K(\pi_1(E_{g})/\Gamma_k(\pi_1(E_{g})), 1);\mathbb{Z}_{\ell}) \overset{\sim}{\rightarrow} H_3(\Fn/\Gamma_k(\Fn); \mathbb{Z}_{\ell})
\end{equation}
coincides with $\tau_k^{(\ell)}(g,\tau)$.
\end{proposition}

\begin{proof}
	By Lemma \ref{lem:4.3.6}, the composition maps factor through $H_3(S^3; \mathbb{Z}_{\ell})$:
	\begin{center}
	\begin{tikzcd}
H_3(E_{g}; \mathbb{Z}_{\ell}) \arrow[d] \arrow[r] & H_3(K_k^{(\ell)}; \mathbb{Z}_{\ell}) \simeq H_3(\Fn/\Gamma_k(\Fn); \mathbb{Z}_{\ell}) \\
H_3(S^3; \mathbb{Z}_{\ell}) \arrow[ur]& 
\end{tikzcd}
\end{center}
By definition of Hurewicz homomorphism,  map $H_3(S^3; \mathbb{Z}_{\ell}) \rightarrow H_3(\Fn/\Gamma_k(\Fn); \mathbb{Z}_{\ell})$ is exactly given by $\tau_k^{(\ell)}(g, \tau)$. This completes the proof.
\end{proof}

\section{Low dimensional homology groups for free-nilpotent Lie algebras over $\mathbb{Z}_{\ell}$} \label{sec:4}

In this section, we recall the rank of homology groups of free-nilpotent Lie algebras with dimension  $\leq 3$. At the level of Lie algebra, there is no difference in computation of Igusa-Orr (\cite{IO}) between over $\mathbb{Z}$ and over $\mathbb{Z}_{\ell}$.

\subsection{Associated graded free Lie algebra by lower central series filtration of pro-$\ell$ free groups}

Let $n$ be a fixed positive integer and $\Fn$ be the free pro-$\ell$ group of rank $n$ (topologically) generated by letters $x_1, \ldots, x_n$. For each positive integer $k$, we denote by $\Gamma_k(\Fn)$ the $k$-th lower central subgroup of $\Fn$ which is recursively defined as
\begin{equation} \label{eq:5.1.1}
	\Gamma_1(\Fn):= \Fn, \quad \Gamma_k(\Fn) := [\Gamma_{k-1}(\Fn), \Fn] \quad (k \geq 2).
\end{equation}
Here, we note that each commutator subgroup $\Gamma_k(\Fn)$ is closed in $\Fn$, since $\Fn$ is finitely generated pro-$\ell$ group (cf. \cite[Proposition 1.1.9]{DDMS}).

The lower central series of $ \Fn$ induces a filtration on $\Fn$ called {\it lower central filtration}
\begin{equation} \label{eq:5.1.2}
	\Fn = \Gamma_1(\Fn) \supset \Gamma_2(\Fn) \supset \cdots \supset \Gamma_k(\Fn) \supset \cdots 
\end{equation}
and we have the associated graded Lie algebra over $\mathbb{Z}_{\ell}$
\begin{equation} \label{eq:5.1.3}
	\mathcal{L} := \bigoplus_{m \geq 1} \gr_{m}(\Fn)
\end{equation}
where $\gr_k(\Fn) := \Gamma_k(\Fn) / \Gamma_{k+1}(\Fn)$ for $k \geq 1$ with Lie bracket induced by commutator of $\Fn$. Then, $\gr_k(\Fn)$ is a free $\mathbb{Z}_{\ell}$-module of  rank $\mathbb{Z}_{\ell}^{\oplus N_k}$ since we have an isomorphism $\gr_k(\Fn) \simeq \mathcal{L}_k$ given by the pro-$\ell$ Magnus homomorphism $\Theta$ which sends $x_i \mapsto 1 + X_i$ and $x_i^{-1} \mapsto 1 - X_i + X_i^2 - X_i^3 + \cdots$,  $(1 \leq i \leq n)$ (cf. Section 2.8). Here,  $\mathcal{L}_k$ is degree $k$ part of free Lie algebra  generated by $X_1, \ldots, X_n$ over $\mathbb{Z}_{\ell}$ and  Witt's formula state that $N_k$ is explicitly given as
\begin{equation}\label{eq:5.1.4}
N_k := \frac{1}{k} \sum_{d | k} \phi(d)(n^{k/d})
\end{equation}
	where $\phi(d)$ is the M\"obius function evaluated at $d$ (\cite[Theorem 5.11]{MKS}). 
	Note that $\mathcal{L}$ obviously depends on the rank $n$ of $\Fn$ (and also on $\ell$) but we suppress it for simplicity of notations.
	
	We set
	\begin{equation} \label{eq:5.1.5}
		\mathcal{L}_{\geq k} := \bigoplus_{m  \geq k} \gr_{m}(\Fn)
	\end{equation}
	then $\mathcal{L}_{\geq k}$ becomes an ideal of the Lie algebra $\mathcal{L}$. Then, its quotient Lie algebra
	\begin{equation} \label{eq:5.1.6}
		\mathcal{L}_{\leq  k -1}  := \mathcal{L} / \mathcal{L}_{\geq k}   =  \bigoplus_{m \geq 1}^{k-1} \gr_{m}(\Fn)
	\end{equation}
	is a free object in the category of nilpotent Lie algebra of nilpotency class $\leq k-1$ with exponents in $\mathbb{Z}_{\ell}$.
	
	For simplicity, under the isomorphism $\Theta : 1 + X_i \leftrightarrow x_i$ we write as
		\begin{equation} \label{eq:5.1.7}
			\mathcal{L}_k = \gr_k(\Fn)
		\end{equation}

\subsection{Koszul complex of free nilpotent Lie algebras and its low dimensional homology groups}

	Let $(\Lambda_{\ast}(\mathcal{L} / \mathcal{L}_{\geq k}), \partial_{\ast})$ be the {\it Koszul complex} of the free nilpotent Lie algebra $\mathcal{L}/\mathcal{L}_{\geq k}$ (with trivial coefficients in $\mathbb{Z}_{\ell}$), that is, for each $m \in \mathbb{Z}$, $\Lambda_m(\mathcal{L} / \mathcal{L}_{\geq k})$ is given by
	\begin{equation} \label{eq:5.2.1}
		\Lambda_m(\mathcal{L} / \mathcal{L}_{\geq k}) = \begin{cases}
 \mathrm{span}_{\mathbb{Z}_{\ell}}\{ g_1 \wedge \cdots \wedge g_m \mid g_i \in \mathcal{L}/\mathcal{L}_{\geq k} \} & (m \geq 1)\\
 \mathbb{Z}_{\ell}	& (m=0) \\
 0 & (\text{otherwise}),
 \end{cases}
\end{equation}
where $\wedge$ denotes the wedge product, and the differential $\partial_m : \Lambda_m(\mathcal{L} / \mathcal{L}_{\geq k}) \rightarrow \Lambda_{m-1}(\mathcal{L} / \mathcal{L}_{\geq k})$ is given by
\begin{equation}\label{eq:5.2.2}
	\partial_m(g_1 \wedge \cdots \wedge g_m) = \sum_{i <j} (-1)^{i+j+1}[g_i, g_j] \wedge g_1 \wedge \cdots \hat{g}_i \cdots \hat{g}_j \cdots \wedge g_m. 
\end{equation}
Here, $\hat{g}_i$ and $\hat{g}_j$ mean that we remove the terms $g_i$ and $g_j$ from the monomials.
	
Next, we recall the notion of  weight of elements of $\mathcal{L}_k$. For $g \in \mathcal{L}_k$, we define the {\it weight} of $g$ as the integer $k$.  We denote it by $wt(g) = k$. For homogeneous element of $m$-th exterior algebra $\Lambda_m(\mathcal{L} / \mathcal{L}_{\geq k})$
	\begin{equation} \label{eq:5.2.3}
		g_1 \wedge \cdots \wedge g_m \in \Lambda_m(\mathcal{L} / \mathcal{L}_{\geq k}),
	\end{equation}
	we define its weight by sum of weights of each component: 
	\begin{equation} \label{eq:5.2.4}
		wt(g_1 \wedge \cdots \wedge g_m) := \sum_{i=1}^m wt(g_i)
	\end{equation}
		
	Let $\Lambda^{(w)}_m(\mathcal{L}/ \mathcal{L}_{ \geq k})$ be the submodule of $\Lambda_m(\mathcal{L}/ \mathcal{L}_{ \geq k})$ consisting of weight $w$ elements. Then, we may see that the boundary operator $\partial_m$ preserve weights, so $\Lambda^{(w)}_{\ast}(\mathcal{L}/ \mathcal{L}_{\geq k}) \subset \Lambda_{\ast}(\mathcal{L} /\mathcal{L}_{\geq k})$ forms a subcomplex. Thus, we have a decomposition 
	\begin{equation} \label{eq:5.2.5}
		\Lambda_{\ast}(\mathcal{L} / \mathcal{L}_{\geq k}) \simeq \bigoplus_{w} \Lambda^{(w)}_{\ast}(\mathcal{L} / \mathcal{L}_{\geq k})
	\end{equation}
	as a chain complex and we have
	\begin{equation} \label{eq:5.2.6}
		H_{\ast}(\mathcal{L} / \mathcal{L}_{\geq k}; \mathbb{Z}_{\ell}) = H_{\ast}(\Lambda_{\ast}(\mathcal{L} / \mathcal{L}_{\geq k})) \simeq H_{\ast}(\bigoplus_{w} \Lambda^{(w)}_{\ast}(\mathcal{L} / \mathcal{L}_{\geq k})) = \bigoplus_{w} H_{\ast}^{(w)} (\Lambda_{\ast}(\mathcal{L} / \mathcal{L}_{\geq k}))
	\end{equation}
	where we set $H_{\ast}^{(w)} (\Lambda_{\ast}(\mathcal{L} / \mathcal{L}_{\geq k}) := H_{\ast}(\Lambda^{(w)}_{\ast}(\mathcal{L} / \mathcal{L}_{\geq k}))$.
	
	Note that the reduction morphism of Lie algebra $\phi : \mathcal{L} / \mathcal{L}_{\geq k+1} \rightarrow \mathcal{L} / \mathcal{L}_{\geq k}$ preserves weights. Therefore, it induces weight preserving morphism of chain complex and homology groups
	\begin{equation} \label{eq:5.2.7}
		\phi_{\ast} : H_{\ast}^{(w)}(\mathcal{L} / \mathcal{L}_{\geq k+1}; \mathbb{Z}_{\ell}) \rightarrow H_{\ast}^{(w)} (\mathcal{L} / \mathcal{L}_{\geq k}; \mathbb{Z}_{\ell}).
	\end{equation}
	
	Then, for the first homology group, the reduction homomorphism induces isomorphism for $k\geq 2$,
	\begin{equation} \label{eq:5.5.1}
		H_1(\mathcal{L}; \mathbb{Z}_{\ell})  \overset{\sim}{\rightarrow} H_1(\mathcal{L}/\mathcal{L}_{\geq k}; \mathbb{Z}_{\ell}),
	\end{equation}
	 and they are concentrated at weight one. Noting that $H_1(\mathcal{L}; \mathbb{Z}_{\ell}) \simeq \mathbb{Z}_{\ell}^{\oplus n}$, we get 
\begin{equation} \label{eq:5.5.2}
	H_1(\mathcal{L}/\mathcal{L}_{\geq k}; \mathbb{Z}_{\ell}) \simeq \mathbb{Z}_{\ell}^{\oplus n}
\end{equation}

Next, by considering the Hochschild-Serre spectral sequence for
\begin{equation} \label{eq:5.3.12}
			0 \longrightarrow \mathcal{L}_{\geq k} \longrightarrow \mathcal{L} \longrightarrow \mathcal{L}/\mathcal{L}_{\geq k} \longrightarrow 0,
\end{equation}
we get a weight-preserving five term exact sequence (see, for example \cite[Corollary A.27]{RZ})
		\begin{equation} \label{eq:5.3.13}
			H_2(\mathcal{L}; \mathbb{Z}_{\ell}) \rightarrow H_2(\mathcal{L} / \mathcal{L}_{\geq k}; \mathbb{Z}_{\ell}) \overset{d_{2,0}^2}{\rightarrow} H_1(\mathcal{L}_{\geq k}; \mathbb{Z}_{\ell})_{\mathcal{L}}\left(=\frac{\mathcal{L}_{\geq k}}{[\mathcal{L}, \mathcal{L}_{\geq k}]}\right) \rightarrow H_1(\mathcal{L}; \mathbb{Z}_{\ell}) \rightarrow H_1(\mathcal{L}/\mathcal{L}_{\geq k}; \mathbb{Z}_{\ell}) \rightarrow 0.
\end{equation}
Since $H_2(\mathcal{L};\mathbb{Z}_{\ell})=0$ and \eqref{eq:5.5.1}, we get
\begin{equation} \label{eq:5.5.3}
	H_2(\mathcal{L}/\mathcal{L}_{\geq k}; \mathbb{Z}_{\ell}) \simeq \mathcal{L}_k \simeq \mathbb{Z}_{\ell}^{\oplus N_k}
\end{equation}
and conclude that $H_2(\mathcal{L}/\mathcal{L}_{\geq k})$ is concentrated at weight $k$ part.

Finally, we recall Igusa-Orr computation of  third homology group $H_3(\mathcal{L}/\mathcal{L}_{\geq k};\mathbb{Z}_{\ell})$ in our situation. They consider the weighted version of Hochschild-Serre spectral sequence $E=\{E_{p,q}^r\}$ associated with the central extension
\begin{equation} \label{eq:5.5.4}
	0 \rightarrow \mathcal{L}_k \rightarrow \mathcal{L}/\mathcal{L}_{\geq k+1} \rightarrow \mathcal{L}/\mathcal{L}_{\geq k} \rightarrow 0
\end{equation}
with 
\begin{equation} \label{eq:5.3.3}
			E_{p,q}^2 \simeq H_p^{(w - qk)}(\mathcal{L} / \mathcal{L}_{\geq k}; \mathbb{Z}_{\ell}) \otimes_{\mathbb{Z}_{\ell}} \Lambda_q(\mathcal{L}_k)
		\end{equation}
		and 
		\begin{equation} \label{eq:5.3.4}
			E_{p,q}^r \Rightarrow H_{p+q}^{(w)}(\mathcal{L} / \mathcal{L}_{\geq k+1}); \mathbb{Z}_{\ell}).
		\end{equation}	
	Then, further computation gives the following theorem.
	
	\begin{theorem}[{\cite[Theorem 5.9]{IO}}] \label{thm:5.4.2}
		For each non-negative integer $k$, the following assertions hold:
		\begin{enumerate}[label=$(\arabic{enumi})$]
			\item \begin{equation} \label{eq:5.4.12}
				H_3^{(w)}(\mathcal{L} / \mathcal{L}_{\geq k}; \mathbb{Z}_{\ell}) \simeq \begin{cases}
 \mathbb{Z}_{\ell}^{\oplus(nN_{w - 1} - N_{w})} & (k+1 \leq w \leq 2k -1), \\
 	0 & (\text{otherwise}),
 \end{cases}
			\end{equation}
	\item The homomorphism $\delta_{k+1,k} : H_3^{(w)}(\mathcal{L} / \mathcal{L}_{ \geq k+1}; \mathbb{Z}_{\ell}) \rightarrow H_3^{(w)}(\mathcal{L} / \mathcal{L}_{\geq k}; \mathbb{Z}_{\ell})$ induced from the canonical map $\mathcal{L}/\mathcal{L}_{\geq k +1} \rightarrow \mathcal{L} / \mathcal{L}_{\geq k}$ is an isomorphism for $w \notin \{k+1, 2k, 2k+1\}$ and the zero homomorphism otherwise.
		\end{enumerate}
	In particular, by combining (1) and (2), we have
	\begin{equation} \label{eq:5.4.13}
		H_3(\mathcal{L} / \mathcal{L}_{\geq k}; \mathbb{Z}_{\ell}) \simeq \bigoplus_{w=k+1}^{2k-1} H_3^{(w)}(\mathcal{L} / \mathcal{L}_{\geq k}; \mathbb{Z}_{\ell}) \simeq \bigoplus_{i=k}^{2k-2} \mathbb{Z}_{\ell}^{\oplus (nN_{i} - N_{i+1})}.
	\end{equation}
	\end{theorem}
	
	\begin{remark}
	(1) In \cite{Ma}, Massuyeau relates $H_3(\mathcal{L} / \mathcal{L}_{\geq k}; \mathbb{Q})$ and a direct sum of the vector space $\mathcal{T}_d(H_{\mathbb{Q}})$ of tree, connected $H_{\mathbb{Q}}=H_1(F_n; \mathbb{Q})$-colored Jacobi diagrams graded by internal degree $d \geq 1$ subject to AS, IHX and multilinearity relations (for details, see [loc.cit.]). His computation is also applicable in our situation when we consider $H_3(\mathcal{L} / \mathcal{L}_{\geq k}; \mathbb{Q}_{\ell})$ and $H_{\mathbb{Q}_{\ell}}=H_1(\Fn;\mathbb{Q}_{\ell})$-colored Jacobi diagrams. More precisely, we have an isomorphism
	\begin{equation}
		\Phi : \bigoplus_{d=k-1}^{2k-3} \mathcal{T}_d(H_{\mathbb{Q}_{\ell}}) \overset{\sim}{\rightarrow} H_3(\mathcal{L}/\mathcal{L}_{\geq k}; \mathbb{Q}_{\ell}).
	\end{equation}
	(2) However, as mentioned in \cite[Remark 1.7]{Ma}, when considering $H_3(\mathcal{L} / \mathcal{L}_{\geq k}; \mathbb{Z}_{\ell})$ and $H_{\mathbb{Z}_{\ell}}=H_1(\Fn;\mathbb{Z}_{\ell})$-colored Jacobi diagrams, we still have a map
	\begin{equation}
		\Phi : \bigoplus_{d=k-1}^{2k-3} \mathcal{T}_d(H_{\mathbb{Z}_{\ell}}) \rightarrow H_3(\mathcal{L}/\mathcal{L}_{\geq k}; \mathbb{Z}_{\ell}),
	\end{equation}
	but, in general, $\Phi$ is not a bijection between two such spaces.		\end{remark}

\section{Low dimensional homology groups for torsion-free nilpotent pro-$\ell$ groups} \label{sec:5}

In this section, we compute the low dimensional homology groups for torsion-free nilpotent pro-$\ell$ groups with coefficients in $\mathbb{Z}_{\ell}$ by considering pro-$\ell$ group analogue of Igusa-Orr computation \cite[\S 6]{IO}. Namely, we compute the rank of homology groups of free nilpotent pro-$\ell$ groups through isomorphisms to that of its associated Lie algebras.

\subsection{Presentation of a torsion-free nilpotent pro-$\ell$ group in terms of basis $\mathcal{B}$}

Let $G$ be a finitely generated nilpotent pro-$\ell$ group. We denote by $\Gamma_k(G)$ the $k$-th lower central subgroup of $G$ which is recursively defined as
\begin{equation} \label{eq:6.1.1}
	\Gamma_1(G):= G, \quad \Gamma_k(G) := [\Gamma_{k-1}(G), G] \quad (k \geq 2).
\end{equation}
Here, we note that each commutator subgroup $\Gamma_k(G)$ is closed in $G$, since $G$ is finitely generated pro-$\ell$ group (cf. \cite[Proposition 1.1.9]{DDMS})

We assume that, for all $k \geq 1$, the associated abelian groups $\Gamma_k(G) / \Gamma_{k+1}(G)$ are torsion free and finitely generated pro-$\ell$ abelian groups. Set
\begin{equation}\label{eq:6.1.2}
	\mathrm{gr}(G) = \bigoplus_{k \geq 1} \Gamma_k(G) / \Gamma_{k+1}(G)
\end{equation}
the associated graded Lie algebra whose Lie bracket is given by commutator of $G$. In this case, note that for each $k \geq 1$, we can choose a (topological) basis $\mathcal{B}_k$ for the free abelian pro-$\ell$ group $\Gamma_{k}(G) / \Gamma_{k+1}(G)$ in terms of basis of the free Lie algebra $\mathcal{L}_k$ of degree $k$ as follows: It is known that there several choices of (ordered) basis sets such as, {\it Hall basis} in \cite[Theorem 5.8]{MKS} and {\it Lyndon basis} in \cite{CFL}. We fix $\mathcal{B}_k(\mathcal{L})$ as such a basis for $\mathcal{L}_k$. Then, the pro-$\ell$ Magnus embedding $\Theta$ induces isomorphism
\begin{equation}
	\Theta : \Gamma_{k}(G) / \Gamma_{k+1}(G) \overset{\sim}{\longrightarrow} \mathcal{L}_k
\end{equation}
as in \eqref{eq:5.1.7}, since $\Gamma_{k}(G) / \Gamma_{k+1}(G)$ is torsion free $\mathbb{Z}_{\ell}$-module. Therefore, we define the corresponding basis set $\mathcal{B}_k$ for $\Gamma_{k}(G) / \Gamma_{k+1}(G)$ by
\begin{equation}
	\mathcal{B}_k = \{ \Theta^{-1}(e) \in \Gamma_{k}(G) / \Gamma_{k+1}(G) \ |\  e \in \mathcal{B}_k(\mathcal{L}) \}.
\end{equation}
Here, note that $\mathcal{B}_k$ inherits the ordering of $\mathcal{B}_k(\mathcal{L})$.

Let $\mathcal{B}$ be the union $\bigcup_{k \geq 1}\mathcal{B}_k$ of such basis set for degree $k$ elements. We endow $\mathcal{B}$ with total ordering such that for any $e_j \in \mathcal{B}_j$ and $e_k \in \mathcal{B}_k$ we have $e_j < e_k$ if $j < k$. 

Then, we define the notion of weight for $\mathcal{B}_k$ as in the case of $\mathcal{L}_k$.
\begin{definition}[weight] \label{def:6.1.3}
	Let $k$ be a positive integer. For $a \in \mathcal{B}_k \sqcup \mathcal{B}_k^{-1}$, we define its weight $wt(a)$ as $k$, i.e., $wt(a)=k$.
\end{definition}

\begin{definition}[reduced word] \label{eq:6.1.4}
	For each $a < b \in \mathcal{B}$, define a reduced word $\mathfrak{c}(a,b)$ in the letters of $\mathcal{B}$ such that
	\begin{enumerate}[label=$(\arabic{enumi})$]
		\item $\mathfrak{c}(a,b) = b^{-1}a^{-1} b a$ as an element of $G$,
		\item  $\mathfrak{c}(a,b)$  has weight $wt(\mathfrak{c}(a, b)) \geq wt(a) + wt(b)$,
		\item $\mathfrak{c}(a,b) = e $ if $b^{-1}a^{-1} b a$ is the identity element in $G$.
	\end{enumerate}
	In addition, we set $\mathfrak{c}(b,a):= \mathfrak{c}(a,b)^{-1}$ as a word in the letters of $\mathcal{B}$.
\end{definition}

\begin{lemma}  \label{lem:6.1.5}
Notations being as above, the nilpotent pro-$\ell$ group $G$ has the following pro-$\ell$ presentation in terms of the chosen basis $\mathcal{B}$ with total ordering as above:
\begin{equation} \label{eq:6.1.6}
	G \simeq \langle \mathcal{B} \mid ab(\mathfrak{c}(a,b))a^{-1}b^{-1} \  \text{for all} \ a < b \in \mathcal{B} \rangle =: \mathcal{H}.
\end{equation}
In particular, each element of $G$ can be uniquely written as a finite product of basis elements with exponents in $\mathbb{Z}_{\ell}$ according to the total ordering of $\mathcal{B}$.
\end{lemma}

\begin{proof}
Let $F^{(\ell)}(\mathcal{B})$ be the free pro-$\ell$ group on the set $\mathcal{B}$ and put
\begin{equation}  \label{eq:6.1.7}
	R:=\overline{\langle \langle ab(\mathfrak{c}(a,b))a^{-1}b^{-1}  \mid  \ a < b \in \mathcal{B} \rangle \rangle}
\end{equation}
the closure of the minimal normal subgroup of $F(\mathcal{B})$ containing  $ab(\mathfrak{c}(a,b))a^{-1}b^{-1}$ $(a < b \in \mathcal{B})$. Here, $F(\mathcal{B})$ denotes the free group generated by $\mathcal{B}$.

By definition of presentation of pro-$\ell$ groups (see \cite[Appendix C]{RZ}), we have to show that there is a short exact sequence of pro-$\ell$ groups
\begin{equation} \label{eq:6.1.8}
	1 \rightarrow R \rightarrow F^{(\ell)}(\mathcal{B}) \rightarrow G \rightarrow 1.
\end{equation}
First, notice that we have the obvious continuous homomorphism $\phi : F^{(\ell)}(\mathcal{B}) \rightarrow G$ which sends a letter of $\mathcal{B}$ to the word it represents. Since each lower central subgroup $\Gamma_k(G)$ is automatically closed, any element of $G$ can be written as combination of a letter of $\mathcal{B}$ as in the usual abstract group theory. Hence, $\phi$ is epimorphism and induces an continuous epimorphism
\begin{equation}  \label{eq:6.1.9}
	\bar{\phi} : \mathcal{H}:=F^{(\ell)}(\mathcal{B})/R \longrightarrow G.
\end{equation}
To complete the proof, we have to show that $\bar{\phi}$ is injective, that is, any word $w$ in $\mathcal{B} \sqcup \mathcal{B}^{-1}$ which is trivial in $G$ is trivial in $\mathcal{H}$. For this, we show that for each $k \geq 1$ there is the unique isomorphism
\begin{equation}  \label{eq:6.1.10}
	\mathcal{H}/ \mathcal{H}^{(k,\ell)}=F(\mathcal{B})/ R_k F^{(k,\ell)}(\mathcal{B}) \simeq G / G^{(k,\ell)} 
\end{equation}
where $\mathcal{H}^{(k,l)}, G^{(k,\ell)}$, and  $F^{(k,\ell)}(\mathcal{B})$ are $k$-th $\ell$-lower central subgroup of $\mathcal{H}, G, F(\mathcal{B})$ respectively, and $R_k = R/F^{(k,\ell)}(\mathcal{B})$. Here, for a group $H$, $\ell$-lower central series $H^{(k, \ell)}$ is inductively defined as follows: $H^{(1, \ell)}:=H$ and $H^{(k+1,\ell)}:= (H^{(k, \ell)})^{\ell}[H^{(k, \ell)}, H]$ $(k \geq 1)$.

Then, we have an epimorphism $\bar{\phi}_k : \mathcal{H}/ \mathcal{H}^{(k,\ell)} \rightarrow G/G^{(k, \ell)}$. To show that it is injective, we use the same argument as in \cite[Lemma 6.1]{IO}, namely, by double induction on the ordered pair $(- m(w), n(m(w)))$. Here, $m(w)$ is the minimal weight of any letter in a word $w$ in $\mathcal{B} \sqcup \mathcal{B}^{-1}$ and $n(m(w))$ is the number of letters in the word $w$ of weight $m(w)$. 

Take  a word $w$ in $\mathcal{B} \sqcup \mathcal{B}^{-1}$ which is trivial in $G/G^{(k, \ell)}$. Note that the minimal weight $m(w)$ is well-defined when considered as element in $F(\mathcal{B}) / F^{(k,\ell)}(\mathcal{B})$. Assume $m(w)=k$. Since $G$ is nilpotent and $\Gamma_k(G) / \Gamma_{k+1}(G)$ is finitely generated, the basis set $\mathcal{B}$ is finite set, and so the set $\{m(w) \mid w \in \mathcal{B} \}$ is bounded. Hence, to proceed induction step, it is enough to increase the weight of $w$ by exchanging letters in $w$ of weight $m(w)$ with those of higher weight.

Since $w$ is trivial in $\Gamma_k(G)/\Gamma_{k+1}(G)G^{(k,\ell)}$, a word with cancelling pair of letters of weight $m(w)$ represents $w$ in $\mathcal{H}/ \mathcal{H}^{(k,\ell)} $ denoted by the same $w$ by abuse of notation.

We have two cases to consider: When $w\equiv g_1 b^{-1} g_2 b g_3 \in \mathcal{H}/ \mathcal{H}^{(k,\ell)}$ for some words $g_1, g_2, g_3 \in  \mathcal{H}$ and a letter $b$ of weight $m(w)$, in terms of the relations $b^{-1}ab = a \mathfrak{c}(a,b)$ and $b^{-1} a^{-1} b = \mathfrak{c}(a, b) a^{-1}$, one can eliminate $b$ and $b^{-1}$. When $w\equiv g_1 b g_2 b^{-1} g_3 \in \mathcal{H}/ \mathcal{H}^{(k,\ell)}$ for some words $g_1, g_2, g_3 \in  \mathcal{H}$ and a letter $b$ of weight $m(w)$, then $w \equiv bb^{-1} g_1 b g_2 b^{-1} g_3 b b^{-1}=b g_1' g_2 g_3' b^{-1}$ where $g_1' = b^{-1} g_1 b$ and $g_2' = b^{-1} g_2 b$. Note that by the first case, the letter $b$ can be eliminated in $g_1'$ and $g_3'$. Hence by double induction on $(-m(w), l(m(w)))$ we conclude that $g_1' g_2 g_3'$ is trivial, and so is $w$. 

Therefore, the homomorphism $\bar{\phi}_k : \mathcal{H}/ \mathcal{H}^{(k,\ell)} \rightarrow  G / G^{(k,\ell)} $ is an isomorphism. Note that $\ell$-lower central series is cofinal in inverse system of pro-$\ell$ groups. Passing to the inverse limit $\varprojlim_{k}$, the assertion follows.
\end{proof}

\subsection{Free chain resolution of $\mathbb{Z}_{\ell}$ generated by basis $\mathcal{B}$}

To relate group homology of $G$ with Lie algebra homology of its associated Lie algebra, we construct a chain complex of $G$ which is isomorphic to the Koszul complex in terms of basis set $\mathcal{B}$.

Let $\llbracket \mathbb{Z}_{\ell} G \rrbracket$ be the completed group algebra of $G$ over $\mathbb{Z}_{\ell}$ (see \cite[Section 5.3]{RZ}). For $n \geq 1$, we set $C_n(G)$ the pro-$\ell$ left $\llbracket \mathbb{Z}_{\ell} G \rrbracket$ module generated by terms expressed as the form
\begin{equation}  \label{eq:6.2.1}
	\langle b_1, b_2, \ldots, b_n \rangle 
\end{equation}

where $b_i \in \mathcal{B}$ and subject to the relation:
\begin{equation*}\label{con:1} 
	\langle b_{\sigma(1)}, \ldots, b_{\sigma(n)}  \rangle = (-1)^{\sign(\sigma)}\langle b_1, \ldots, b_n \rangle \quad \text{for}\ \sigma \in S_n.
\end{equation*}

Here, $S_n$ denotes the $n$-th symmetric group. For $n=0$, we set $C_0(G)$ as free pro-$\ell$ left  $\llbracket \mathbb{Z}_{\ell} G \rrbracket$ module generated by $\langle \ \rangle$, thus $C_0(G) \simeq \llbracket \mathbb{Z}_{\ell} G \rrbracket$. For $\langle b_1, \ldots, b_n \rangle \in C_n(G)$, we define its weight by $wt(\langle b_1, \ldots, b_n \rangle):= \sum_{i=1}^n wt(b_i)$ and for $\langle \ \rangle \in C_0(G)$ its weight is defined by $wt(\langle \ \rangle):=0$.

Then, we see that $C_n(G)$ is, in fact, the free pro-$\ell$ left $\llbracket \mathbb{Z}_{\ell} G \rrbracket$ module with basis given by the expressions
\begin{equation*}
	\langle b_1, b_2, \ldots, b_n \rangle\quad \text{with}\quad b_1 < b_2 < \cdots < b_n
\end{equation*}
canonically determined by the ordering of $\mathcal{B}$. Following \cite[\S 6]{IO}, we will use the convention that the first entry $b_1$ in $\langle b_1, b_2, \ldots, b_n \rangle$ is allowed to be a word in $\mathcal{B} \sqcup \mathcal{B}^{-1}$ with the additional relations:
\begin{enumerate}[label=$(\arabic{enumi})$]
	\item $\langle  b^{-1}, b_2,\ldots, b_n  \rangle = -b^{-1} \langle b, b_2, \ldots, b_n \rangle,  \quad b \in \mathcal{B}$\label{con:2}
	\item $\langle  ab, b_2,\ldots, b_n  \rangle =  \langle a, b_2, \ldots, b_n \rangle + a \langle b, b_2, \ldots, b_n \rangle, \quad$ $a$ and $b$ are words in $\mathcal{B} \sqcup \mathcal{B}^{-1}$\label{con:3}
\end{enumerate}

When $a=b^{-1}$, (2) implies (1) and hence these relations are consistent.

Consider the chain complex of free profinite left $\llbracket \mathbb{Z}_{\ell}G \rrbracket$-modules
\begin{equation} \label{eq:6.2.2}
	C_2(G) \overset{\partial_2}{\rightarrow} C_1(G) \overset{\partial_1}{\rightarrow} C_0(G) \overset{\epsilon}{\rightarrow} \mathbb{Z}_{\ell} \rightarrow 0
\end{equation}
whose differential operators are defined for each basis as
\begin{eqnarray*}
	&&\partial_1 (\langle b \rangle) = b - 1, \\
	&&\partial_2 (\langle a,b \rangle) = (a - 1)\langle b \rangle - (b -1)\langle a \rangle + ab\langle \mathfrak{c}(a,b)\rangle,\\
	&&\epsilon(g)=1, \quad (g \in G).
\end{eqnarray*}
One may check that the above sequence actually forms chain complex, that is, $\partial_{k+1} \circ \partial_k =0$ by direct computation.

Our next task is to extend the above chain complex \eqref{eq:6.2.2} to a chain resolution
\begin{equation}
	\cdots \overset{\partial_{n+1}}{\rightarrow} C_{n+1}(G) \overset{\partial_n}{\rightarrow}C_n(G) \overset{\partial_n}{\rightarrow} \cdots \overset{\partial_2}{\rightarrow} C_1(G) \overset{\partial_1}{\rightarrow} C_0(G) \overset{\epsilon}{\rightarrow} \mathbb{Z}_{\ell} \rightarrow 0
\end{equation}
whose differentials are similar to that of Koszul complex of free nilpotent Lie algebras in \eqref{eq:5.2.2}. For this, we first prove Pseudoweight Lemma (Lemma \ref{lem:6.2.11}), and then, in terms of this pseudoweight lemma we obtain a chain resolution with desired property (Theorem \ref{thm:6.2.13}). 

To begin with, we prepare some notations which plays important role in  Pseudoweight Lemma (Lemma \ref{lem:6.2.11}). Let $z \in \mathcal{B}$ be a maximal weight element. Set $\overline{G}= G/\langle\langle z \rangle\rangle$ where $\langle\langle z \rangle\rangle$ denotes the closed normal subgroup of $G$ topologically generated by $z$. We suppose $wt(z)=k$. We define a (continuous) homomorphism $q : C_n(G) \rightarrow C_n(\overline{G})$ by
\begin{equation}\label{eq:6.2.3}
	q(\langle b_1, \ldots, b_n \rangle) = \begin{cases}
 	\langle b_1, \ldots, b_n \rangle & (\text{if} \ b_i \neq z\ \text{for all} \ i),\\
 	0 & (\text{otherwise}).
 \end{cases}
\end{equation}

Since the kernel of the canonical surjection $\llbracket \mathbb{Z}_{\ell} G \rrbracket \rightarrow \llbracket \mathbb{Z}_{\ell} \overline{G} \rrbracket$ is generated by $z-1$ as pro-$\ell$ $\llbracket \mathbb{Z}_{\ell} G \rrbracket$-module, the kernel $\Ker(q)$ is generated by the terms $\langle b_1, \ldots, b_{n-1}, z \rangle$ and $(z-1)\langle b_1, \ldots, b_n \rangle$ with $b_i \neq z$ for $i=1, \ldots, n$. Note that the (continuous) homomorphism $q : C_n(G) \rightarrow C_n(\overline{G})$ is a chain homomorphism for $n \leq 2$.

\begin{lemma} \label{eq:6.2.4}
Let $z$ be a maximal weight element of $\mathcal{B}$. Then, $z-1$ is not a zero divisor of $\llbracket \mathbb{Z}_{\ell}G \rrbracket$
\end{lemma}

\begin{proof}
Let $J := (IG) + \ell [\mathbb{Z}_{\ell}G]$ be an ideal of (abstract) group algebra $[\mathbb{Z}_{\ell}G]$. Here, $(IG)$ is the augmentation ideal of $[\mathbb{Z}_{\ell}G]$. Recall that the family of ideals $\{J^m\}_{m\geq 1}$ is cofinal with the inverse system of $\mathbb{Z}_{\ell}$-algebra $\llbracket \mathbb{Z}_{\ell}G \rrbracket$ (cf. \cite[Lemma 7.1]{DDMS}), and so  we have $\llbracket \mathbb{Z}_{\ell}G \rrbracket= \varprojlim_{m} ([\mathbb{Z}_{\ell}G]/J^m)$. Thus, an element $\lambda \in \llbracket \mathbb{Z}_{\ell}G \rrbracket$ can be regraded as the coherent sequence $\lambda = (\lambda_m)$ with $\lambda_m \in [\mathbb{Z}_{\ell}G]/J^m$. For each $m$, $\lambda_m$ is uniquely decomposed into a finite sum $\sum_g \alpha_g \cdot g \in [\mathbb{Z}_{\ell}G]/J^m$. Note that each element of $G$ can be written uniquely as a product of the elements of $\mathcal{B}$ with exponents in $\mathbb{Z}_{\ell}$. Therefore, for large $m$, there exist $n_m \in \mathbb{Z} / \ell^r\mathbb{Z}$ and $a_m \in \mathbb{Z} / \ell^s\mathbb{Z}$ for some $r, s$ uniquely such that $\lambda_m = n_m\cdot z^{a_m} + \sum_{b \neq z} \alpha_b \cdot b$. By considering for all $m$, we see that there exist unique $n \in \mathbb{Z}_{\ell}$ and $a \in \mathbb{Z}_{\ell}$ such that $\lambda = n\cdot z^a + (\text{terms not containing}\  z)
$. Hence, if $(z-1)\lambda = 0$ then $\lambda = 0$. This completes the proof.
\end{proof}

\begin{definition} \label{def:6.2.5}
	For $l \geq 1$, we set $\lpbracket I_lG \rpbracket:= \ker(\llbracket \mathbb{Z}_{\ell}G \rrbracket \rightarrow \llbracket \mathbb{Z}_{\ell}G/\Gamma_l(G) \rrbracket)$ and for $l=0$ $\lpbracket I_0G\rpbracket:= \llbracket \mathbb{Z}_{\ell}G \rrbracket$. An element $x \in C_n(G)$ is said to be of {\it pseudoweight} $\geq h$ 
	if 
	\begin{equation} \label{eq:6.2.6}
		x \in \sum_{j=0}^h \lpbracket I_jG\rpbracket\cdot C_n^{(h-j)}(G)
	\end{equation}
where $C_n^{(w)}$ denotes the submodule of $C_n(G)$ generated by $\beta = \langle b_1,\ldots, b_n \rangle \in \mathcal{B}^{n}$ with weight $wt(\beta) = w$. For $x \in C_n(G)$ with pseudoweight $\geq h$, we denote it by $\widetilde{wt}(x) \geq h$.
\end{definition}

\begin{definition} \label{def:6.2.7}
	For $\beta = \langle b_1, \ldots, b_n \rangle \in C_n(G)$, we define its {\it molecule} $M_{\beta}$ by
	\begin{eqnarray*}
		&&M_{\beta} = \sum_{i=1}^n (-1)^{i+1}(b_i -1) \langle b_1, \ldots, \hat{b}_i, \ldots, b_n \rangle\\
		&&+ \sum_{i<j} (-1)^{i+j+1} b_i b_j \langle \mathfrak{c}(b_i, b_j), b_1, \ldots, \hat{b}_i, \ldots, \hat{b}_j, \ldots, b_n \rangle
	\end{eqnarray*}
\end{definition}

\begin{definition}[pseudoweight condition]  \label{def:6.2.8}
let $(C_{\ast}(G), \partial_{\ast})$ be a $\llbracket \mathbb{Z}_{\ell}G \rrbracket$ partial free resolution of $\mathbb{Z}_{\ell}$ up to  degree $\leq n+1$:
\begin{equation} \label{eq:6.2.9}
	C_{n+1}(G) \overset{\partial_{n+1}}{\rightarrow} C_{n}(G) \overset{\partial_{n}}{\rightarrow} \cdots \overset{\partial_2}{\rightarrow} C_1(G) \overset{\partial_1}{\rightarrow} C_0(G) \overset{\epsilon}{\rightarrow} \mathbb{Z}_{\ell} \rightarrow 0.
\end{equation}
The partial free resolution $(C_{\ast}(G), \partial_{\ast})$ is said to satisfy the {\it pseudoweight condition in dimension $n$} if following condition holds: Any $n$-cycle $c \in \ker(\partial_n) \subset C_n(G)$ has pseudoweight $\widetilde{wt}(c) \geq h$ if and only if $c$ is boundary of some element with weight $\geq h$, i.e., there is an element $d \in C_{n+1}(G)$ such that $wt(d) \geq h$ and $\partial_{n+1}(d)=c$.
\end{definition}

\begin{lemma} \label{lem:6.2.10}
	Suppose that a $\llbracket \mathbb{Z}_{\ell}G \rrbracket$ partial free resolution  $(C_{\ast}(G), \partial_{\ast})$ of $\mathbb{Z}_{\ell}$ up to degree $\leq n$ are given, and,  for each generator $\beta$ of  $C_n(G)$, its differential satisfies the following condition:
	\begin{equation} \label{eq:6.2.11}
		\partial_n(\beta) = M_{\beta} + (\text{terms of weight}\ > wt(\beta)).
	\end{equation}
	Then, for any generator $\alpha$ of $C_{n+1}(G)$ with weight $wt(\alpha) = h$, the differential $\partial_n(M_{\alpha})$ of the molecule $M_{\alpha}$ of $\alpha$ has pseudoweight $>h$, i.e., $\widetilde{wt}(\partial_n(M_{\alpha})) > h$.
\end{lemma}

\begin{proof}
	We set $\Delta(\beta):=M_{\beta} - \partial_n(\beta)$. Equivalently, we  have $\partial_n(\beta) = M_{\beta} - \Delta(\beta)$. Then, by assumption \eqref{eq:6.2.11}, we know $wt(\Delta(\beta)) > wt(\beta)$. Noting that for any $a. b \in \mathcal{B}$,  $\langle \mathfrak{c}(a,b) \rangle$ can be written as linear combination 
	\begin{equation} \label{eq:6.4.1}
		\langle \mathfrak{c}(a,b) \rangle = \sum_{m} \gamma_m(a,b) \langle e_m(a,b) \rangle \in C_1(G)
	\end{equation}
	by iterate use of the  relations of the module $C_1(G)$. Here, $| \gamma_m(a, b) | \in G$ and $e_m(a,b) \in \mathcal{B}$. We define a symbol $\epsilon^{ijk}$ as follows: 
	\begin{equation} \label{eq:6.4.2}
		\epsilon^{ijk} = \begin{cases}
  (-1)^{i+j+k} & (i<k<j),\\
  (-1)^{i+j+k+1} & (k < i < j\ \text{or}\ i < j <k).	
 \end{cases}
	\end{equation}
	Then, using \eqref{eq:6.4.1} and \eqref{eq:6.4.2}, through long computation, the differential $\partial_{n} M_{\langle a_0, \ldots, a_n \rangle}$ of molecule of $\langle a_0, \ldots, a_n \rangle \in C_{n+1}$ can be finally computed as follows:
	\begin{alignat}{3}
	\partial_{n}M_{\langle a_0, \ldots, a_n \rangle} &=& &\sum_{i=0}^n (-1)^i (a_i - 1) \partial_n \langle a_0, \ldots, \hat{a}_i, \ldots, a_n \rangle  \label{eq:6.4.3} \\ \nonumber
	& & &+ \sum_{i <j} (-1)^{i+j +1}  \sum_{m} a_i a_j \gamma_m(a_i, a_j) \partial_n \langle e_m(a_i, a_j), \ldots, \hat{a}_i, \ldots, \hat{a}_j, \ldots, a_n \rangle \\ 
	&=& &- \sum_{i=0}^n (a_i - 1) \Delta(\langle a_0, \ldots, \hat{a}_i, \ldots, a_n \rangle) \label{eq:6.4.4} \\ 
	& & & - \sum_{i <j} \sum_{m} (-1)^{i+j+1} a_i a_j \gamma_m(a_i, a_j) \Delta(\langle e_m(a_i, a_j), a_0, \ldots, \hat{a}_i, \ldots, \hat{a}_j, \ldots, a_n \rangle) \label{eq:6.4.5}\\ 
	& & & + \sum_{i <j}\sum_m \epsilon^{ijk} a_k a_i a_j \gamma_m(a_i, a_j)(1- [a_k, a_ia_j |\gamma_m(a_i, a_j)|] \langle e_m(a_i, a_j), \ldots, \hat{a}_i, \hat{a}_j,\hat{a}_k, \ldots \rangle \label{eq:6.4.6}\\ \nonumber
	& & & + \sum_{i < j, i<r<s}\sum_{m,l} (-1)^i \epsilon^{rjs}(a_ia_j \gamma_m(a_i, a_j) a_r a_s \gamma_l(a_r, a_s) - a_r a_s \gamma_l(a_r, a_s) a_i a_j \gamma_m(a_i, a_j))) \nonumber  \\ 
	& & & \cdot \langle e_m(a_i, a_j), e_l(a_r, a_s), \ldots, \hat{a}_i, \hat{a}_j, \hat{a}_r, \hat{a}_s, \dots \rangle  \label{eq:6.4.7} \\ 
	&+ & & \sum_{i <j}\sum_{m,k} \epsilon^{ikj} a_i a_j \gamma_m(a_i, a_j) e_m(a_i, a_j) a_k \langle [e_m(a_i, a_j), a_k], \ldots, \hat{a}_i, \hat{a}_j, \hat{a}_k, \ldots \rangle. \label{eq:6.4.8}
\end{alignat}
By following (1), (2), (3) and (4), we conclude that $\partial_n M_{\langle a_0,\ldots, a_n \rangle}$ has pseudoweight $> h = wt(\langle a_0, \ldots, a_n \rangle)$:
\begin{enumerate}[label=$(\arabic{enumi})$]
	\item Note that $wt(e_m(a_i, a_j)) =wt(\mathfrak{c}(a_i, a_j)) \geq wt(a_i) + wt(a_j)$ by Definition \ref{eq:6.1.4} (2) and \eqref{eq:6.4.1}. Then, by assumption \eqref{eq:6.2.11}, we have
	\begin{equation}
		wt(\langle \Delta(e_m(a_i, a_j), a_0, \ldots, \hat{a}_i, \ldots, \hat{a}_j, \ldots, a_n \rangle)) > h
	\end{equation}
	and
	\begin{equation}
		wt(\Delta(\langle a_0, \ldots, \hat{a}_i, \ldots, a_n \rangle)) > wt(\langle a_0, \ldots, \hat{a}_i, \ldots, a_n \rangle) = h - wt(a_i).
	\end{equation}
	Hence, the sums \eqref{eq:6.4.4} and \eqref{eq:6.4.5} have pseudoweight $>h$ since $(a_i -1) \in (\!(I_{wt(a_i)} G )\!)$.
	\item Since we have
	\begin{equation}
		wt([a_k, a_i a_j|\gamma_m(a_i,a_j)|]) = wt(a_k) +  wt(a_i a_j|\gamma_m(a_i,a_j)|) > wt(a_k)
	\end{equation}
	and
	\begin{alignat}{2}
		& & & wt(\langle e_m(a_i, a_j), \ldots, \hat{a}_i, \hat{a}_j, \hat{a}_k, \ldots \rangle)\\
		 =& & &   wt(e_m(a_i, a_j)) + h - wt(a_i) - wt(a_j) - wt(a_k) \\
		 \geq & & & h - wt(a_k),
	\end{alignat}
	we have $(1- [a_k, a_i a_j|\gamma_m(a_i,a_j)|]) \in (\!(I_{wt(a_k)+1}G )\!)$, and so the sum \eqref{eq:6.4.6} has pseudoweight $>h$.
	\item Note that $a_ia_j \gamma_m(a_i, a_j) a_r a_s \gamma_l(a_r, a_s) - a_r a_s \gamma_l(a_r, a_s) a_i a_j \gamma_m(a_i, a_j)$ is contained in $\Ker(\epsilon)$. Thus, $a_ia_j \gamma_m(a_i, a_j) a_r a_s \gamma_l(a_r, a_s) - a_r a_s \gamma_l(a_r, a_s) a_i a_j \gamma_m(a_i, a_j) \in (\!(I_1G)\!)$. On the other hand, the term $\langle e_m(a_i, a_j), e_l(a_r, a_s), \ldots, \hat{a}_i, \hat{a}_j, \hat{a}_r, \hat{a}_s, \dots \rangle$ has weight $\geq h$. Therefore, the sum \eqref{eq:6.4.7} has pseudoweight $\geq h+1$.
	
	\item The term $ \langle [e_m(a_i, a_j), a_k], \ldots, \hat{a}_i, \hat{a}_j, \hat{a}_k, \ldots \rangle$ has weight $\geq h$ since $wt(e_m(a_i, a_j))\geq wt(a_i) + wt(a_j)$. 
		Therefore, the only problem we need to consider is the remaining sign of $\sgn(\gamma_m(a_i,a_j))$ after taking augmentation to their coefficients. By considering mod $\Gamma_{w+1}(\Fn)$ with $w = wt(a_i) + wt(a_j) +wt(a_k)$, we have
	\begin{equation}
		\sum_{m} \sgn(\gamma_m(a_i,a_j))[e_m,(a_i, a_j), a_k] = [[a_i, a_j], a_k].
	\end{equation}
	In addition, by taking summation on indices $i,j,k$ and noting the effect of $\epsilon^{ikj}$, we get  
	\begin{equation} \label{eq:6.11.1}
		[[a_i, a_j],a_k] - [[a_i, a_k],a_j] + [[a_j, a_k],a_i].
	\end{equation}
	By  Jacobi identity, we know that \eqref{eq:6.11.1} lies in $\Gamma_{w+1}(\Fn)$. Thus, finally we conclude that the sum \eqref{eq:6.4.8} has weight $\geq h+1$, and thus, has pseudoweight $\geq h+1$.
\end{enumerate}
\end{proof}

\begin{lemma}[Pseudoweight Lemma]\label{lem:6.2.11}
Suppose  following condition (1), (2) and (3):

\begin{enumerate}[label=$(\arabic{enumi})$]
	\item Given a following commutative diagram with exact rows:
	\begin{center}
\begin{tikzcd}
C_{n+1}(G) \arrow[r, "\partial_{n+1}"] \arrow[d, "q"] & C_n(G) \arrow[r, "\partial_{n}"] \arrow[d, "q"] & C_{n-1}(G) \arrow[d, "q"] \\
C_{n+1}(\bar{G}) \arrow[r, "\partial_{n+1}"]
& C_{n}(\bar{G}) \arrow[r, "\partial_n"] & C_{n-1}(\bar{G}).
\end{tikzcd}	
\end{center}

\item The differentials  $\partial_{n+1} : C_{n+1}(G) \rightarrow C_{n}(G)$ is given as follows: \begin{eqnarray*}
	&&\partial_n(\langle b_1, \ldots, b_{n+1} \rangle) =  \sum_{i=1}^{n+1} (-1)^{i+1}(b_i - 1) \langle b_1, \ldots, \hat{b}_i, \ldots, b_{n+1} \rangle  \\
	&&+\sum_{i <j }(-1)^{i+j+1} b_i b_j \langle \mathfrak{c}(b_i, b_j), b_1, \ldots, \hat{b}_i, \ldots, \hat{b}_j, \ldots, b_{n+1} \rangle + (\text{terms of higher weight}).
\end{eqnarray*}
\item For $j \in \{n-1, n\}$, the differentials $\partial_{j+1}$ satisfies
\begin{equation} \label{lem:6.2.12}
	\partial_{j+1} \langle \beta, z \rangle = (-1)^j (z - 1) \beta + \langle \partial_j \beta, z \rangle
\end{equation}
for each generator $\beta = \langle b_1, \ldots, b_j \rangle$ of $C_j(G)$.
\end{enumerate}

Then, the following statements hold: If the bottom row satisfies the pseudoweight condition in dimension $n$, then so does the top row. Moreover, the condition is satisfied compatible manner with the above commutative diagram, that is, for any $n$-cycle $c \in C_n(G)$ with pseudoweight $\widetilde{wt}(c) \geq h$ and any $x \in C_{n+1}(\bar{G})$ with weight $wt(x) \geq h$ such that $q(c) = \partial_{n+1}(x)$, there exists a $y \in C_{n+1}(G)$ of weight $wt(y) \geq h$ such that $q(y) = x$ and $\partial_{n+1}(y) = c$.
\end{lemma}

\begin{proof}
	We suppose that $c \in C_n(G)$ satisfies  the conditions $\widetilde{wt}(c) \geq h$ and $\partial_n(c) =0$. By definition of $q$, we have $\widetilde{wt}(q(c)) \geq h$ and $\partial_n(q(c))=q(\partial_n(c))=0$. Then, by assumption, there exists $x \in C_{n+1}(G)$ with $\partial_{n+1}(q(x))=q(c)$ and $wt(q(x)) \geq h$. Indeed, since the bottom row in the commutative diagram (1) satisfies the pseudoweight condition in dimension $n$ and  $q(c) \in \Ker(\partial_n)$, there exists $l \in C_{n+1}(\bar{G})$ such that $wt(l) \geq h$ and $\partial_{n+1}(l) = q(c)$. By surjectivity of $q$, there exists $x \in C_{n+1}(G)$ such that $q(x)=l$. Thus, we have $\partial_{n+1}(q(x))=q(c)$. Moreover, by definition of $q$, we may assume $wt(x) \geq h$.
	
Set $\delta := c -\partial_{n+1}(x)$. Then, we have $q(\delta) = 0$, so $q \in \Ker(q)$. Therefore, as we mention in the below of \eqref{eq:6.2.3}, $\delta$ can be written as the form:
\begin{equation} \label{eq:6.5.2}
	\delta = \sum_{i} \gamma_i \langle \alpha_i, z \rangle + \sum_{j} \omega_j (z - 1) \langle \beta_j \rangle
\end{equation}
where $\alpha_i = \langle a_1, \ldots, a_{n-1} \rangle$ and $\beta_j \langle b_1, \ldots, b_n \rangle$ are some generators, and $\gamma_i, \omega_j \in \llbracket \mathbb{Z}_{\ell} G \rrbracket$.

Then, to complete the proof, it suffices to show 
\begin{equation} \label{eq:6.5.1}
	\partial_{n+1}\left( \sum_{j} (-1)^n \omega_j \langle \beta_j, z \rangle \right)=\delta.
\end{equation}

In fact, if \eqref{eq:6.5.1} is true, then $y:= x + \sum_j \omega_j \langle \beta_j, z \rangle \in C_{n+1}(G)$ satisfies the desired property. Since $\sum_j \omega_j \langle \beta_j, z \rangle$ has weight $\geq h$ and lies in $\Ker(q)$, $y$ has also weight $\geq h$ and $q(y)=q(x)$. By definition of $\delta$, we have $\partial_{n+1}(y)=c$. Therefore, the pseudoweight condition in dimension $n$ holds in the top row in compatible way.

We give a proof of \eqref{eq:6.5.1}. By \eqref{lem:6.2.12}, we get
\begin{alignat}{2}
\partial_{n+1} \left( \sum_{j} \omega_j \langle \beta_j, z \rangle \right)&=  & \sum_{j} (-1)^n \omega_j (z-1) \langle \beta_j \rangle + \sum_{j} \omega_j \langle \partial_n \beta_j, z \rangle\\
&=& \sum_{j} (-1)^n \omega_j (z-1) \langle \beta_j \rangle + \sum_{j} \omega_j \langle \partial_n' \beta_j, z \rangle \label{eq:6.5.3}
\end{alignat}
where $\partial_n' \beta_j = \partial_n \beta_j - \zeta$ and  $\zeta$ denotes  terms in $\partial_n \beta_j$ which contain $\langle c_1, \ldots, c_n \rangle$ with $c_i = z$ for some $i$. Note that $\langle c_1, \ldots, c_n \rangle = 0$ if $c_i=c_j$ for some $i \neq j$. Note that $\partial_{n}(\delta) = \partial_n(c - \partial_{n+1}(x))=0$ since $c$ is a cycle. Thus, using \eqref{eq:6.5.2} and \eqref{eq:6.5.3}, 
\begin{alignat}{3}
	0 &=& &\partial_n(\delta)\\
	 &=& & \partial_n(\sum_i  \gamma_i \langle \alpha_i, z \rangle + \sum_{j} \omega_j (z - 1) \langle \beta_j \rangle) \\
	 &=& & \partial_n(\sum_i  \gamma_i \langle \alpha_i, z \rangle + (-1)^n \partial_{n+1}(\sum_j \omega_j \langle \beta_j, z \rangle) - (-1)^n \sum_j \omega_j \langle \partial_n' \beta_j, z\rangle)\\
	 &=& & \partial_n (\sum_i  \gamma_i \langle \alpha_i, z \rangle   - (-1)^n \sum_j \omega_j \langle \partial_n' \beta_j, z\rangle) \\
	 &=& & (\sum_i \gamma_i \langle \partial_{n-1}\alpha_i, z \rangle  + \sum_j (-1)^n \omega_j \langle \partial_{n-1} \zeta, z \rangle) \\
	 & & &+ (z-1) (\sum_{i} (-1)^n \gamma_i \langle \alpha_i \rangle - \sum_{j} \omega_j \langle \partial_n' \beta_j \rangle).
\end{alignat}
Then, $\sum_i \gamma_i \langle \partial_{n-1}\alpha_i, z \rangle  + \sum_j (-1)^n \omega_j \langle \partial_{n-1} \zeta, z \rangle$ and $(z-1) (\sum_{i} (-1)^n \gamma_i \langle \alpha_i \rangle - \sum_{j} \omega_j \langle \partial_n' \beta_j \rangle$ are linearly independent because one sum  contain $z$ in $\langle c_1, \ldots, c_n \rangle$ but the other does not. Therefore, we have
\begin{equation}
	\sum_i \gamma_i \langle \partial_{n-1}\alpha_i, z \rangle  + \sum_j (-1)^n \omega_j \langle \partial_{n-1} \zeta, z \rangle = 0
\end{equation}
and
\begin{equation} \label{eq:6.5.5}
	(z-1) (\sum_{i} (-1)^n \gamma_i \langle \alpha_i \rangle - \sum_{j} \omega_j \langle \partial_n' \beta_j \rangle) = 0.
\end{equation}
Then, by Lemma \ref{eq:6.4.2}, $(z-1)$ is not a zero divisor, so we get from \eqref{eq:6.5.5},
\begin{equation} \label{eq:6.5.6}
	\sum_{i} (-1)^n \gamma_i \langle \alpha_i \rangle - \sum_{j} \omega_j \langle \partial_n' \beta_j \rangle = 0.
\end{equation}
Thus, by putting \eqref{eq:6.5.6} into \eqref{eq:6.5.3}, we get \eqref{eq:6.5.1}. Henceforth, the assertion has been proved.
\end{proof}

\begin{theorem}\label{thm:6.2.13}
	Notations being as above, there are continuous $\llbracket \mathbb{Z}_{\ell} G \rrbracket$-module homomorphism
	\begin{equation} \label{eq:6.2.14}
		\partial_n : C_n(G) \rightarrow C_{n-1}(G)
	\end{equation}
	such that
	\begin{enumerate}[label=$(\arabic{enumi})$]
		\item  $(C_{\ast}(G), \partial_{\ast})$ is a free chain resolution of $\mathbb{Z}_{\ell}$ with $\epsilon : C_0(G) \rightarrow \mathbb{Z}_{\ell}$ the augmentation map.
		\item The homomorphism $\partial_n$ is given as follows:\begin{eqnarray*}
			&&\partial_n(\langle b_1, \ldots, b_n \rangle) = \sum_{i=1}^n(-1)^{i+1} (b_i - 1)\langle b_1, \ldots, \hat{b}_i, \ldots, b_n \rangle \\
			 &+& \sum_{i <j }(-1)^{i+j+1} b_i b_j \langle \mathfrak{c}(b_i, b_j), b_1, \ldots, \hat{b}_i, \ldots, \hat{b}_j, \ldots, b_n \rangle 
			+ (\text{terms of higher weight}).
		\end{eqnarray*}
	\end{enumerate}
\end{theorem}

\begin{proof}
We prove the assertion by induction. For the first step of induction, we can use the explicitly constructed chain complex in \eqref{eq:6.2.2}.  

We assume that the homomorphisms $\partial_i : C_i(G) \rightarrow C_{i-1}(G)$ has been defined for all $1 \leq i \leq n$ which acts on each generator $\langle b_1, \ldots, b_n \rangle$ as (2). Note that the differential satisfies the assumption of Lemma  \ref{lem:6.2.10}. For each generator $\alpha \in C_{n+1}(G)$, let $M_{\alpha} \in C_n(G)$ be its molecule. Then, using Lemma \ref{lem:6.2.10}, we see $\widetilde{wt}(\partial_n M_{\alpha}) > wt(\alpha)$. By applying the Pseudoweight Lemma to the cycle $\partial_n(M_{\alpha}) \in C_{n-1}(G)$, there exists $\Delta(\alpha) \in C_n(G)$ with $\partial_n(\Delta(\alpha))=\partial_n(M_{\alpha})$ and $wt(\Delta(\alpha)) > wt(\alpha)$. In terms of $M_{\alpha}$ and $\Delta(\alpha)$, we define $\partial_{n+1} : C_{n+1}(G) \rightarrow C_n(G)$ by
\begin{equation}
	\partial_{n+1}(\alpha):= M_{\alpha} - \Delta(\alpha)
\end{equation}
Then, by construction we get $\partial_{n+1}\circ \partial_n = 0$ and $\partial_{n+1} \circ q = q \circ \partial_{n+1}$ by compatibility condition. Since a chain homomorphism $q : C_n(G) \rightarrow C_n(\overline{G})$ is defined compatible way, $C_{\bullet}(G)$ satisfies the Pseudoweight condition for any $n$. This  implies $C_{\bullet}(G)$ is exact a fortirori and the assertion has been proved.
\end{proof}

\subsection{Computation of dimension of $H_3(\Fn / \Gamma_k(\Fn); \mathbb{Z}_{\ell})$}

In this section, we compute the rank of the third homology group $H_3(\Fn/\Gamma_k(\Fn); \mathbb{Z}_{\ell})$ of $\Fn/\Gamma_k(\Fn)$ with (trivial) coefficients in $\mathbb{Z}_{\ell}$ in term of the free chain resolution of $\mathbb{Z}_{\ell}$ constructed in the previous section. Since this free chain resolution is constructed so as to be related to the Koszul complex of free nilpotent Lie algebra, one may show isomorphism between  $H_3(\Fn/\Gamma_k(\Fn); \mathbb{Z}_{\ell})$ and $H_3(\mathcal{L}/\mathcal{L}_{\geq k}; \mathbb{Z}_{\ell})$ using spectral sequence associated with  weight filtrations.

By thinking $\mathbb{Z}_{\ell}$ as a $\llbracket \mathbb{Z}_{\ell} G \rrbracket$-module with trivial $G$-action, we set
\begin{equation}\label{eq:6.3.1}
	D_{\ast}(G):= \mathbb{Z}_{\ell} \hat{\otimes}_{\llbracket \mathbb{Z}_{\ell} G \rrbracket} C_{\ast}(G).
\end{equation}
Here $\hat{\otimes}_{\llbracket \mathbb{Z}_{\ell} G \rrbracket}$ means a complete tensor product over $\llbracket \mathbb{Z}_{\ell} G \rrbracket$ (see \cite[Section 5.5.]{RZ}). Note that by theorem \ref{thm:6.2.13}(2), the induced homomorphism $\partial_{n}$ on $D_{n}(G)$ is given by
\begin{equation}\label{eq:6.3.2}
	\partial_n(\langle b_1, \ldots, b_n \rangle) = \sum_{i <j }(-1)^{i+j+1}  \langle \mathfrak{c}(b_i, b_j), b_1, \ldots, \hat{b}_i, \ldots, \hat{b}_j, \ldots, b_n \rangle 
			+ (\text{terms of higher weight}).
\end{equation}

Filter the chain complex $(D_{\ast}(G), \partial_n)$ by setting $\mathcal{F}^lD_n(G)$ be the submodule of $D_{\ast}(G)$ generated by terms $\langle g_1, \ldots, g_n \rangle$ whose weight $wt(\langle g_1, \ldots, g_n \rangle) \geq l$. Setting the associated graded chain complex as
\begin{equation} \label{eq:6.3.3}
	\mathrm{gr}(D_{\ast}(G)) :=\bigoplus_{l} \mathcal{F}^l D_{\ast}(G) / \mathcal{F}^{l+1} D_{\ast}(G),
\end{equation}
we see that $\mathrm{gr}(D_{\ast}(G))$ is isomorphic to the Koszul complex  $\Lambda_{\ast}(\mathrm{gr}(G))$ of the associated graded Lie algebra $\mathrm{gr}(G)$ of $G$, i.e., $\mathrm{gr}(D_{\ast}(G)) \simeq \Lambda_{\ast}(\mathrm{gr}(G))$. Note that the filtration of $D_n(G)$ is bounded. In fact, for a nilpotent group $G$ of nilpotency class $(k-1)$, $\mathcal{F}^l D_n(G) = D_n(G)$ if $l\leq n$ and $\mathcal{F}^l D_n(G) = 0$ if $l > n(k-1)$. 

Thus, by considering spectral sequence associated to a filtered chain complex (see \cite[Theorem A.3.1]{RZ}), we get a May spectral sequence $E=\{E^r_{p,q} \}$ such that
\begin{equation}\label{cor:6.3.4}
		E^1_{p,q} = H^{(p)}_{p+q}(\mathrm{gr}(D_{\ast}(G)); \mathbb{Z}_{\ell}), \quad E^1_{p,q} \Rightarrow H_{p+q}(G; \mathbb{Z}_{\ell})
	\end{equation}
	where $H^{(p)}_{\ast}(\mathrm{gr} D_{\ast}(G))$ means the weight $p$ part of the homology group of the Lie algebra $\mathrm{gr}(D_{\ast}(G))$ and $H_{\ast}(G;\mathbb{Z}_{\ell}):=H_{\ast}((D_{\ast}(G), \partial_{\ast}))$ is filtered by
	\begin{equation}\label{eq:6.3.7}
		\mathcal{F}^l H_{\ast}(G; \mathbb{Z}_{\ell}) := \Im(H_{\ast}(\mathcal{F}^l D_{\ast}(G)) \rightarrow H_{\ast}(D_{\ast}(G))).
	\end{equation}
	In particular, we have
	\begin{equation}\label{eq:6.3.8}
		E^{\infty}_{l, n - l} \simeq \frac{\mathcal{F}^l H_n(G; \mathbb{Z}_{\ell})}{\mathcal{F}^{l+1} H_n(G; \mathbb{Z}_{\ell})}.
	\end{equation}

\begin{corollary}\label{eq:6.3.9}
	For $i \leq 3$, the homology groups $H_i(\mathcal{L}/ \mathcal{L}_{\geq k}; \mathbb{Z}_{\ell})$ of Lie algebra $\mathcal{L} / \mathcal{L}_{\geq k}$ and those $H_{i}(\Fn / \Gamma_k(\Fn); \mathbb{Z}_{\ell})$ of free nilpotent pro-$\ell$ group $\Fn/\Gamma_k(\Fn)$ are isomorphic, that is, 
	\begin{equation}\label{eq:6.3.10}
		H_i(\mathcal{L} / \mathcal{L}_{\geq k}; \mathbb{Z}_{\ell}) \simeq H_i(\Fn / \Gamma_k(\Fn);\mathbb{Z}_{\ell}) \quad (i \leq 3).
	\end{equation}
	Moreover, associated graded abelian groups are given by
	\begin{equation}\label{eq:6.3.11}
		\frac{\mathcal{F}^l H_3(\Fn / \Gamma_k(\Fn); \mathbb{Z}_{\ell})}{\mathcal{F}^{l+1} H_3(\Fn / \Gamma_k(\Fn); \mathbb{Z}_{\ell})} \simeq H_3^{(l)}(\mathcal{L} / \mathcal{L}_{\geq k}; \mathbb{Z}_{\ell}) \simeq \begin{cases}
			\mathbb{Z}_{\ell}^{\oplus (n N_{l-1} - N_{l})} & (k+1 \leq l \leq 2k-1),\\
			0 & \text{otherwise}.
		\end{cases}
	\end{equation}
	and 
	\begin{equation}\label{eq:6.3.12}
		\mathcal{F}^{2k}H_3(\Fn / \Gamma_k(\Fn); \mathbb{Z}_{\ell}) \simeq 0.
	\end{equation}
	In particular, we have
	\begin{equation}\label{eq:6.3.13}
		H_3(\Fn / \Gamma_k(\Fn); \mathbb{Z}_{\ell}) \simeq \bigoplus_{i=k}^{2k-2} \mathbb{Z}_{\ell}^{\oplus(nN_i - N_{i+1})}
	\end{equation}
\end{corollary}

\begin{proof}
After passing to Malcev completion of $\Fn / \Gamma_{k}(\Fn)$ over $\mathbb{Q}_{\ell}$, by applying Pickel's isomorphism $P$ (\cite{P}), we get 
\begin{equation}\label{eq:6.3.14}
	P : H_i(\Fn / \Gamma_{k}(\Fn); \mathbb{Q}_{\ell}) \simeq H_i(\mathcal{L} / \mathcal{L}_{\geq k}; \mathbb{Q}_{\ell}) \quad (\text{for all} \ i).
\end{equation}
Note that, since the above isomorphism is defined over $\mathbb{Q}_{\ell}$, each relevant derived step from $E^r_{p,q}$ to $E^{r+1}_{p,q}$ in the May spectral sequences in \eqref{cor:6.3.4} cannot reduce the rank of $E^r_{p,q}$. Also note that by Theorem \ref{thm:5.4.2}, \eqref{eq:5.5.3} and \eqref{eq:5.5.1} \eqref{eq:5.5.2}, we know that $H_i(\mathcal{L} / \mathcal{L}_{\geq k}; \mathbb{Z}_{\ell})$ is torsion free for $0 \leq i \leq 3$ and so $E^1_{p,q} = H^{(p)}_{p+q}(\gr(D_{\ast}(\Fn/\Gamma_k(\Fn))); \mathbb{Z}_{\ell})$ is also torsion free for $p+q \leq 3$. Then, we conclude that $E^1_{p,q} \simeq E^{\infty}_{p,q}$ for $p+q \leq 3$. In fact, $E^{2}_{p,q}$ is a subquotient of torsion free module $E^{1}_{p,q}$ and must have the same rank with $E^1_{p,q}$, so  $E^1_{p,q} \simeq E^2_{p,q}$. By iterating this procedure, we obtain  $E^1_{p,q} \simeq E^{\infty}_{p,q}$.

Thus, by Theorem \ref{thm:5.4.2}, we deduce that 
	\begin{equation}\label{eq:6.3.15}
		\frac{\mathcal{F}^l H_3(\Fn / \Gamma_k(\Fn); \mathbb{Z}_{\ell})}{\mathcal{F}^{l+1} H_3(\Fn / \Gamma_k(\Fn); \mathbb{Z}_{\ell})} \simeq E^{\infty}_{l, 3-l} \simeq E^1_{l,3-l} \simeq  H_3^{(l)}(\mathcal{L} / \mathcal{L}_{\geq k}; \mathbb{Z}_{\ell}) 
	\end{equation}
	is isomorphic to the free $\mathbb{Z}_{\ell}$-module $\mathbb{Z}_{\ell}^{\oplus (nN_{l-1} - N_l)}$ if $k+1 \leq l \leq 2k-1$ and 0 otherwise. Since every relevant terms $E^1_{p,q} \simeq E^{\infty}_{p,q}$ converging to $H_{p+q}(\Fn/\Gamma_k(\Fn); \mathbb{Z}_{\ell})$ are torsion free, the extension problem can be solved by just taking the direct sum. Combination of this extension with Pickel's isomorphism leads to $\mathcal{F}^{2k} H_3(\Fn/ \Gamma_k(\Fn); \mathbb{Z}_{\ell})=0$, and therefore, the desired equation \eqref{eq:6.3.11} and \eqref{eq:6.3.12}. 
\end{proof}

Consider the reduction map, for $l \geq k$,
\begin{equation} \label{eq:6.3.16}
	\phi_{l,k} : H_3(\Fn/\Gamma_l(\Fn); \mathbb{Z}_{\ell}) \rightarrow H_3(\Fn/ \Gamma_k(\Fn); \mathbb{Z}_{\ell}).
\end{equation}
In the following, we consider the chain complex $D_{\ast}(\Fn/\Gamma_k(\Fn))$ by setting $G=\Fn/\Gamma_k(\Fn)$ in \eqref{eq:6.3.1}.
\begin{lemma}\label{lem:6.3.17}
For	each $n \leq 3$ and integer $s$, we have
\begin{equation}\label{eq:6.3.18}
	\mathcal{F}^sH_n(D_{\ast}(\Fn/\Gamma_l(\Fn)); \mathbb{Z}_{\ell}) \simeq H_n(\mathcal{F}^s D_{\ast}(\Fn/\Gamma_l(\Fn)); \mathbb{Z}_{\ell}).
\end{equation}
For $l \geq k$, the following reduction map is surjective;
\begin{equation}\label{eq:6.3.19}
	H_3(\mathcal{F}^{s} D_{\ast}(\Fn/\Gamma_l(\Fn)); \mathbb{Z}_{\ell}) \rightarrow H_3(\mathcal{F}^s D_{\ast}(\Fn/\Gamma_k(\Fn)); \mathbb{Z}_{\ell}).
\end{equation}

\end{lemma}

\begin{proof}
	Let us consider a spectral sequence associated with the weight filtration of $\mathcal{F}^lD_{\ast}(\Fn/\Gamma_k(\Fn))$. Then, one gets a May-like spectral sequence $\{E_{p,q}^r\}$ satisfying 
	\begin{equation}
		E^1_{p,q} = H_{p+q}^{(p)}(\mathrm{gr}(\mathcal{F}^s D_{\ast}(\Fn/\Gamma_k(\Fn)));\mathbb{Z}_{\ell}), \quad E_{p,q}^1 \Rightarrow H_{p+q}(\mathcal{F}^s D_{\ast}(\Fn/\Gamma_k(\Fn)); \mathbb{Z}_{\ell}).
	\end{equation}
	Here, $\mathrm{gr}(\mathcal{F}^s D_{\ast}(\Fn/\Gamma_k(\Fn)))$ denotes the associated graded chain complex with respect to the weight filtration of $\mathcal{F}^lD_{\ast}(\Fn/\Gamma_k(\Fn))$. Then, by  Corollary \ref{eq:6.3.9}, we conclude
	\begin{align} \label{eq:6.6.1}
		H_{p+q}^{(p)}(\mathrm{gr}(\mathcal{F}^s D_{\ast}(\Fn/\Gamma_k(\Fn)); \mathbb{Z}_{\ell})  & \simeq  \bigoplus_{l \geq s} H_{p+q}^{(p)}(\Lambda^l_{\ast}(\mathrm{gr}(\Fn/\Gamma_k(\Fn))); \mathbb{Z}_{\ell})\\
		 & \simeq  \begin{cases}
				H_{p+q}^{(p)}(\mathcal{L}/\mathcal{L}_{\geq k}; \mathbb{Z}_{\ell}) & (p \geq s)\\
				0 & (\text{otherwise}) \nonumber
 		\end{cases}
	\end{align}

	Since the differentials preserve the weights, the inclusion $\mathcal{F}^s D_{\ast}(\Fn/\Gamma_k(\Fn)) \subset D_{\ast}(\Fn/\Gamma_k(\Fn))$ preserves the filtration, and hence,  it induces a homomorphism of their spectral sequences. By Corollary \ref{eq:6.3.9} and \eqref{eq:6.6.1}, the induced homomorphisms give isomorphism on $E_{p,q}^1$ terms for $p \geq s$. As we see in the proof of Corollary \ref{eq:6.3.9}, in the spectral sequence for $D_{\ast}(\Fn/\Gamma_k(\Fn))$ with  range $p+q \leq 3$,  all differentials involving terms $E^{\ast}_{p,q}$ must be trivial. Thus, the comparison of spectral sequences implies that the differentials must also be zero in the spectral sequence for $\mathcal{F}^s D_{\ast}(\Fn/\Gamma_k(\Fn))$ with $p+q \leq 3$. Therefore, in both spectral sequences, $E_{p,q}^1 \simeq E_{p,q}^{\infty}$ in the range of $p+q \leq 3$. As in the proof of Corollary \ref{eq:6.3.9}, when $p+q \leq 3$ all therms $E^{\infty}_{p,q}$ are torsion free, and therefore, $H_n(\mathcal{F}^s D_{\ast}(\Fn/\Gamma_k(\Fn)); \mathbb{Z}_{\ell}) \rightarrow \mathcal{F}^s H_n(D_{\ast}(\Fn/\Gamma_k(\Fn)); \mathbb{Z}_{\ell})$ is an isomorphism for $0 \leq n \leq 3$ and any $s$. 
	
	Next, by comparing the spectral sequence for $\mathcal{F}^sD_{\ast}(\Fn/\Gamma_{k+1}(\Fn))$ and $\mathcal{F}^s D_{\ast}(\Fn/\Gamma_k(\Fn))$ in terms of the reduction map $\Fn/\Gamma_{k+1}(\Fn) \rightarrow \Fn / \Gamma_k(\Fn)$, we obtain surjective homomorphism, for $s \geq k+2$,
	\begin{equation}
		H_3(\mathcal{F}^s D_{\ast}(\Fn/\Gamma_{k+1}(\Fn)); \mathbb{Z}_{\ell}) \rightarrow H_3(\mathcal{F}^s D_{\ast}(\Fn/\Gamma_k(\Fn)); \mathbb{Z}_{\ell}).
	\end{equation}
	In fact, the reduction map induces surjection on $E_{p,q}^1$ terms for $p+q \leq 3$ and $s \geq k+2$, and differentials involving these terms are trivial in both spectral sequences. This completes the proof.
	\end{proof}

We introduce another filtration to $H_3(\Fn / \Gamma_k(\Fn) ; \mathbb{Z}_{\ell})$ induced by the reduction homomorphisms $\phi_{l,k}$ as follows:
\begin{equation} \label{eq:6.3.20}
	\mathcal{G}^l H_3(\Fn/\Gamma_k(\Fn); \mathbb{Z}_{\ell}) := \Im \phi_{l-1,k} := \Im\left( H_3(\Fn/\Gamma_{l-1}(\Fn); \mathbb{Z}_{\ell}) \rightarrow H_3(\Fn/\Gamma_k(\Fn); \mathbb{Z}_{\ell}) \right)
\end{equation}

\begin{theorem}\label{thm:6.3.21}
Two filtrations $\mathcal{G}^l H_3(\Fn / \Gamma_{k}(\Fn); \mathbb{Z}_{\ell})$ and $\mathcal{F}^l H_3(\Fn / \Gamma_{k}(\Fn); \mathbb{Z}_{\ell})$ coincide:
	
\begin{equation} \label{eq:6.3.22}
	\mathcal{G}^l H_3(\Fn / \Gamma_{k}(\Fn); \mathbb{Z}_{\ell}) = \mathcal{F}^l H_3(\Fn / \Gamma_{k}(\Fn); \mathbb{Z}_{\ell})	
\end{equation}

In particular, 
\begin{equation} \label{eq:6.3.23}
	0 = \mathcal{G}^{2k}H_3(\Fn / \Gamma_{k}(\Fn); \mathbb{Z}_{\ell}) \subset \cdots \subset \mathcal{G}^{k+1} H_3(\Fn / \Gamma_{k}(\Fn); \mathbb{Z}_{\ell}) = H_3(\Fn / \Gamma_{k}(\Fn); \mathbb{Z}_{\ell})
\end{equation}
with associated graded terms 
\begin{equation} \label{eq:6.3.24}
	\frac{\mathcal{G}^l H_3(\Fn / \Gamma_{k}(\Fn); \mathbb{Z}_{\ell})}{\mathcal{G}^{l+1}H_3(\Fn / \Gamma_{k}(\Fn); \mathbb{Z}_{\ell})} \simeq \mathbb{Z}_{\ell}^{\oplus(nN_{l-1} - N_l)}
\end{equation}
and $x \in \ker(\phi_{l,k})$  if and only if $x \in \Im(\phi_{2k-1,l})$ for $k \leq l \leq 2k-1$.
\end{theorem}

\begin{proof}
We consider the following commutative diagram:
\begin{center}
\begin{tikzcd}
H_3(\mathcal{F}^{l+1} D_{\ast}(\Fn/\Gamma_l(\Fn)); \mathbb{Z}_{\ell}) \ar[r] \ar[d] \ar[dr, "\alpha"] & H_3(\mathcal{F}^{l+1} D_{\ast}(\Fn/\Gamma_k(\Fn)); \mathbb{Z}_{\ell})	\ar[d, "\beta"] \\
\mathcal{F}^{l+1} H_3(\Fn/\Gamma_{l}(\Fn); \mathbb{Z}_{\ell}) \ar[r] & \mathcal{F}^{l+1} H_3(\Fn/\Gamma_k(\Fn) ; \mathbb{Z}_{\ell})
\end{tikzcd}
\end{center}

We immediately see that the tow vertical homomorphisms are isomorphisms by definition of filtration $\mathcal{F}$ on $H_3(\Fn/ \Gamma_{\ast}(\Fn); \mathbb{Z}_{\ell})$ and the top horizontal homomorphism is onto by Lemma \ref{lem:6.3.17}. Note that $\mathcal{F}^{l+1}H_3(\Fn / \Gamma_l(\Fn);\mathbb{Z}_{\ell}) = H_3(\Fn/\Gamma_l(\Fn); \mathbb{Z}_{\ell})$ by Corollary \ref{eq:6.3.9}. Since the reduction homomorphism $\phi_{l,k}$ preserves the filtration $\mathcal{F}$ on $H_3(\Fn/\Gamma_k(\Fn); \mathbb{Z}_{\ell})$, one sees that the image $\Im(\alpha)$ is exactly equal to $\Im(\phi_{l,k})=\mathcal{G}^{l+1}H_3(\Fn/\Gamma_k(\Fn); \mathbb{Z}_{\ell})$. Therefore, we conclude that
\begin{equation}\label{eq:6.3.25}
	\mathcal{G}^{l+1}H_3(\Fn/\Gamma_k(\Fn); \mathbb{Z}_{\ell}) = \Im(\alpha) = \Im(\beta) = \mathcal{F}^{l+1}H_3(\Fn/\Gamma_k(\Fn); \mathbb{Z}_{\ell}).
\end{equation}
\end{proof}

As a byproduct of the above arguments, we get the following pro-$\ell$ version of Stallings' theorem.

\begin{corollary} \label{cor:6.3.26} The homomorphism 
\begin{equation}
	H_2(\Fn / \Gamma_{k+1}(\Fn); \mathbb{Z}_{\ell}) \rightarrow H_2(\Fn / \Gamma_{k}(\Fn); \mathbb{Z}_{\ell})
\end{equation}
is zero map.
	
\end{corollary}

\begin{proof}
By \eqref{eq:5.5.3}, we know that the homomorphism $H_2(\mathcal{L} / \mathcal{L}_{\geq k+1}; \mathbb{Z}_{\ell}) \rightarrow H_2(\mathcal{L} / \mathcal{L}_{\geq k}; \mathbb{Z}_{\ell})$ is zero map since they are concentrated at different weight to each other. By applying the isomorphism $H_2(\mathcal{L} / \mathcal{L}_{\geq k}; \mathbb{Z}_{\ell}) \simeq H_2(\Fn / \Gamma_{k}(\Fn); \mathbb{Z}_{\ell})$, the result follows.
\end{proof}

\section{Computation of dimension of $\pi_3(\widehat{K}_k^{(\ell)})$} \label{sec:6}

In this section, we compute the rank of $\pi_3(\widehat{K}_k^{(\ell)})$. The computation is very similar to the computation of $\pi_3(K_k)$ in \cite{IO}. That is, we show the existence of an isomorphism between $\pi_3(\widehat{K}_k^{(\ell)})$ and the third homology group $H_3(\xi_k^{(\ell)}; \mathbb{Z}_{\ell})$ of a pro-$\ell$ torus bundle  $\xi_k^{(\ell)}$ over $K_k^{(\ell)}$ and reduce the computation to that of $H_3(\xi_k^{(\ell)}; \mathbb{Z}_{\ell})$. However, to apply their techniques to our situation, we will need more subtle arguments than the ones in [loc.cit.]

\subsection{Isomorphism $\pi_3(\widehat{K}_k^{(\ell)}) \simeq H_3(\xi_k^{(\ell)}; \mathbb{Z}_{\ell})$}

To begin with, let us show that there is an isomorphism $\pi_3(\widehat{K}_k^{(\ell)}) \simeq H_3(\xi_k^{(\ell)}; \mathbb{Z}_{\ell})$ to reduce the computation of $\pi_3(\widehat{K}_k^{(\ell)})$ to that of $H_3(\xi_k^{(\ell)}; \mathbb{Z}_{\ell})$.

For this, we describe the behaviour of universal circle bundle $\mathbb{S}^{\infty} \rightarrow \mathbb{CP}^{\infty}$ under pro-$\ell$ completion.

\begin{lemma}\label{lem:7.1.1} 
Let $\widehat{\mathbb{S}}^{\infty}{}^{(\ell)}  \rightarrow  \widehat{\mathbb{C} \mathbb{P}}{}^{\infty}{}^{(\ell)}$ be the pro-$\ell$ completion of  the universal circle bundle $\mathbb{S}^{\infty} \rightarrow \mathbb{C} \mathbb{P}^{\infty}$. Then the following statements hold:
\begin{enumerate}[label=$(\arabic{enumi})$]
	\item The fibre of $\widehat{\mathbb{S}}^{\infty}{}^{(\ell)}  \rightarrow  \widehat{\mathbb{C} \mathbb{P}}{}^{\infty}{}^{(\ell)}$ is weak equivalent to the pro-$\ell$ completion  $\widehat{S}^1{}^{(\ell)}$ of $S^1$. 
	\item The pro-$\ell$ completion $\widehat{\mathbb{C} \mathbb{P}}{}^{\infty}{}^{(\ell)}$ is weak equivalent to $K(\mathbb{Z}_{\ell}, 2)$.
\end{enumerate}

\end{lemma}

\begin{proof}
(1) By applying \cite[Lemma(5.3)]{AM}, we have the following commutative diagram:
\begin{center}
\begin{tikzcd}
S^1 \arrow[r] \arrow[d] & \mathbb{S}^{\infty} \arrow[r] \arrow[d] &  \mathbb{C} \mathbb{P}^{\infty} \arrow[d] \\
\bar{F} \arrow[r] & \widehat{\mathbb{S}}^{\infty}{}^{(\ell)}  \arrow[r] & \widehat{\mathbb{C} \mathbb{P}}{}^{\infty}{}^{(\ell)}
\end{tikzcd}
\end{center}
where $\bar{F}$ is the fibres of $\widehat{\mathbb{S}}^{\infty}{}^{(\ell)}  \rightarrow  \widehat{\mathbb{C} \mathbb{P}}{}^{\infty}{}^{(\ell)}$. Since $\mathbb{C}\mathbb{P}^{\infty}\sim K(\mathbb{Z}, 2)$ is simply connected, so is $\widehat{\mathbb{C} \mathbb{P}}{}^{\infty}{}^{(\ell)}$. Hence, by \cite[Theorem(5.9)]{AM} and \cite[Corollary(4.4)]{AM}, we see that $\pi_i(\widehat{S}^1{}^{(\ell)}) \simeq \pi_i(\bar{F})$ for each $i$. Therefore, the statement follows from the uniqueness of Eilenberg-MacLane (pro-)spaces \cite[Corollary 4.14]{AM}.

(2) Applying homotopy exact sequence of the fibre space to $\widehat{\mathbb{S}}^{\infty}{}^{(\ell)}  \rightarrow  \widehat{\mathbb{C} \mathbb{P}}{}^{\infty}{}^{(\ell)}$ and using statement (1), we obtain the result from the uniqueness of Eilenberg-MacLane (pro-)spaces.
\end{proof}

\begin{remark}
Here, we give an another proof of  $\widehat{\mathbb{C}\mathbb{P}}^{\infty}{}^{(\ell)} \sim K(\mathbb{Z}_{\ell},2)$ in terms of the relation of completion and fibrations, but it also follows from the fact that $\mathbb{Z}$ is ``good'' since $\mathbb{Z}$ is finitely generated Abelian group (cf. Example \ref{ex:2.3.1}).	
\end{remark}

In terms of above Lemma \ref{lem:7.1.1}, we define a pro-$\ell$ torus bundle $\xi_k^{(\ell)} \rightarrow K_k^{(\ell)}$ in the following manner. The central extension
\begin{equation} \label{eq:7.1.2} 
	0 \rightarrow \Gamma_{k}(F_n) / \Gamma_{k+1}(F_n)\simeq \mathbb{Z}^{\oplus N_k} \rightarrow F_n / \Gamma_{k+1}(F_n) \rightarrow F_n / \Gamma_{k}(F_n) \rightarrow 0,
\end{equation}
gives rise to a fibre bundle  $K(k+1) \rightarrow K(k)$ whose fibre is $N_k$-dimensional torus denoted by $T_k$. Since $\mathbb{S}^{\infty} \rightarrow \mathbb{C} \mathbb{P}^{\infty}$ is the universal circle bundle, the bundle $K(k+1) \rightarrow K(k)$ is classified by a map
\begin{equation}\label{eq:7.1.3}
	K(k) \rightarrow \prod_{i=1}^{N_k} (\mathbb{C}\mathbb{P}^{\infty})_i \sim  K(H_2(F_n/\Gamma_{k}(F_n); \mathbb{Z}), 2) \end{equation}
where $(\mathbb{C}\mathbb{P}^{\infty})_i$ is just a copy of  $\mathbb{C}\mathbb{P}^{\infty}$. Note that the above map $K(k) \rightarrow K(H_2(F_n/\Gamma_{k}(F_n); \mathbb{Z}), 2)$ is isomorphic on second homology group with coefficients in $\mathbb{Z}$.

By taking the pro-$\ell$ completion, we get the induced map
\begin{equation}\label{eq:7.1.4}
	K^{(\ell)}(k) \rightarrow \prod_{i=1}^{N_k} (\widehat{\mathbb{C}\mathbb{P}}^{\infty}{}^{(\ell)})_i \sim K(H_2(\Fn/\Gamma_k(\Fn); \mathbb{Z}_{\ell}), 2).
\end{equation}
Here, we note that $F_n/\Gamma_k(F_n)$ is ``$\ell$-good'' in the sense of Serre by central extension \eqref{eq:7.1.2} (cf.\cite[2.6 Excercises 2)(c)(d)]{Ser2}). Thus, $\widehat{K(k)}^{(\ell)} \sim K(\widehat{F_n/\Gamma_k(F_n)}^{(\ell)}, 1)$. By universal property of pro-$\ell$ completion,  there is a map $\Fn / \Gamma_k(\Fn) \rightarrow \widehat{F_n / \Gamma_k(F_n)}^{(\ell)}$, and thus, the corresponding map $K^{(\ell)}(k) \rightarrow \widehat{K(k)}^{(\ell)}$.

Similar to the above, the fibre of $K^{(\ell)}(k+1) \rightarrow K^{(\ell)}(k)$ is $\widehat{T}_k{}^{(\ell)} \sim K(\mathbb{Z}_{\ell}^{\oplus N_k}, 1)$ since we have
\begin{equation}\label{eq:7.1.5}
	0 \rightarrow \Gamma_{k}(\Fn) / \Gamma_{k+1}(\Fn)  \rightarrow \Fn / \Gamma_{k+1}(\Fn) \rightarrow \Fn / \Gamma_k (\Fn) \rightarrow 0
\end{equation}
and $\Gamma_{k}(\Fn) / \Gamma_{k+1}(\Fn) \simeq \widehat{\Gamma_{k} (F_n)/\Gamma_{k+1}(F_n)}{}^{(\ell)} \simeq \mathbb{Z}_{\ell}^{\oplus N_k} $. In terms of Hurewicz Theorem in the pro-category \cite[Corollary(4.5)]{AM}, the map $K^{(\ell)}(k) \rightarrow \prod_{i=1}^{N_k}(\widehat{\mathbb{C}\mathbb{P}}^{\infty}{}^{(\ell)})_i$ induces isomorphism on second homology group with coefficients in $\mathbb{Z}_{\ell}$.

Since $K(H_2(\Fn / \Gamma_{k}(\Fn); \mathbb{Z}_{\ell}), 2)$ is simply connected, the map $K^{(\ell)}(k) \rightarrow \prod_{i=1}^{N_k}(\widehat{\mathbb{C}\mathbb{P}}^{\infty}{}^{(\ell)})_i$ factors through the the pro-$\ell$ Orr space $K^{(\ell)}_k$. We denote by $\xi^{(\ell)}_k \rightarrow K^{(\ell)}_k$ the pullback bundle obtained from $\widehat{\mathbb{S}}^{\infty}{}^{(\ell)}  \rightarrow  \widehat{\mathbb{C} \mathbb{P}}{}^{\infty}{}^{(\ell)}$ via pulling back by $K^{(\ell)}_k \rightarrow \prod_{i=1}^{N_k}(\widehat{\mathbb{C}\mathbb{P}}^{\infty}{}^{(\ell)})_i$. Thus, we obtain the following commutative diagram:
\begin{center}
\begin{tikzcd}
\vee \widehat{S}^1{}^{(\ell)} \times \widehat{T}_k^{(\ell)} \arrow[r] \arrow[d] & K^{(\ell)}(k+1) \arrow[r] \arrow[d] & \xi_k^{(\ell)} \arrow[r] \arrow[d] & \prod_{i=1}^{N_k} \widehat{\mathbb{S}}^{\infty}{}^{(\ell)} \arrow[d]\\
\vee \widehat{S}^1{}^{(\ell)} \arrow[r] & K^{(\ell)}(k) \arrow[r]& K^{(\ell)}_k \arrow[r] &	\prod_{i=1}^{N_k} (\widehat{\mathbb{C}\mathbb{P}}^{\infty}{}^{(\ell)})_i\end{tikzcd}
	
\end{center}

Since $K_k^{(\ell)}$ can be thought of as pushout (fiber coproduct) of morphisms $\vee \widehat{S}^1{}^{(\ell)} \rightarrow K^{(\ell)}(k)$ and $\vee \widehat{S}^1{}^{(\ell)} \rightarrow \ast$ by definition of $K_k^{(\ell)}$, we obtain the following commutative diagram compatible with the above one:

\begin{center}
\begin{tikzcd}[row sep=scriptsize, column sep=scriptsize]
& \vee \widehat{S}^1{}^{(\ell)} \times \widehat{T}_k^{(\ell)}  \arrow[dr] \arrow[rr] \arrow[dd] & & K^{(\ell)}(k+1)\arrow[dr] \arrow[dd] \\
 & & \widehat{T}_k^{(\ell)} \arrow[rr, crossing over]& & \xi^{(\ell)}_k \arrow[dd] \\
& \vee \widehat{S}^1{}^{(\ell)} \arrow[dr] \arrow[rr] & & K^{(\ell)}(k) \arrow[dr] \\
 & & \ast \arrow[from=uu, crossing over] \arrow[rr]&  & K^{(\ell)}_k
\end{tikzcd}	
\end{center}
Here, the bottom (and also the top) square is a pushout diagram, and the vertical maps are bundle maps with fiber the pro-$\ell$ completion of a torus 
$\widehat{T}_k^{(\ell)}$.
\begin{lemma}\label{lem:7.1.6}
For each $k \geq 2$, there is a map $\gamma_{k+1, k}: \xi_k^{(\ell)} \rightarrow \xi_k^{(\ell)}$ such that it induces the following commutative diagram:
\begin{center}
\begin{tikzcd}
\pi_3(\xi^{(\ell)}_{k+1}) \arrow[r, "\gamma_{k+1, k}{}_{\ast}"] \arrow[d] &  \pi_3(\xi_k^{(\ell)}) 	\arrow[d] \\
\pi_3(K_{k+1}^{(\ell)}) \arrow[r, "\psi_{k+1, k}"] & \pi_3(K^{(\ell)}_k).
\end{tikzcd}
\end{center}
	Moreover, the vertical homomorphisms $\pi_3(\xi^{(\ell)}_k) \rightarrow \pi_3(K^{(\ell)}_k)$ are isomorphisms for each $k \geq 2$.
\end{lemma}

\begin{proof}
Let us consider the following commutative diagram:
\begin{center}
\begin{tikzcd}
K^{(\ell)}(k+2) \arrow[r] & K^{(\ell)}(k+1) \arrow[r, "="] & K^{(\ell)}(k+1)	 \arrow[r] & K^{(\ell)}(k) \\
\vee \widehat{S}^1{}^{(\ell)} \times \widehat{T}_{k+1}^{(\ell)} \arrow[u] \arrow[r, "proj"] \arrow[d] & \vee \widehat{S}^1{}^{(\ell)} \arrow[u] \arrow[r, "id \times \ast"] \arrow[u]  \arrow[d] & \vee \widehat{S}^1{}^{(\ell)} \times \widehat{T}_{k+1}^{(\ell)} \arrow[r, "proj"]  \arrow[d] \arrow[u] & \vee  \widehat{S}^1{}^{(\ell)} \arrow[u]  \arrow[d]\\
\widehat{T}_{k+1}^{(\ell)} \arrow[r] & \ast \arrow[r] &\widehat{T}_{k+1}^{(\ell)} \arrow[r] & \ast 
\end{tikzcd}
\end{center}
Here, the symbol $\ast$ denotes the constant map or the base point depending on context. By taking the pushout for each vertical two morphisms, the above commutative diagram induces the following sequence of four pushouts:
\begin{equation}\label{eq:7.1.7}
	\xi_{k+1}^{(\ell)} \rightarrow K_{k+1}^{(\ell)} \rightarrow \xi_{k}^{(\ell)} \rightarrow K_k^{(\ell)}
\end{equation}
Note that this map is well-defined up to homotopy since the spaces appearing here are all $K(\pi, 1)$ spaces for some group $\pi$.

Denotes $\gamma_{k+1,k} : \xi_{k+1}^{(\ell)} \rightarrow K_{k+1}^{(\ell)} \rightarrow \xi_{k}$ and $\psi_{k+1,k} : K_{k+1}^{(\ell)} \rightarrow \xi_{k}^{(\ell)} \rightarrow K_k^{(\ell)}$ the composition maps in the above sequence. Then, by inserting $\gamma_{k+1,k}$ and $\psi_{k+1,k}$ in the sequence above, one obtains the following commutative diagram:
	
\begin{center}
\begin{tikzcd}
\xi^{(\ell)}_{k+1} \arrow[r, "\gamma_{k+1,k}"] \arrow[d] & \xi^{(\ell)}_{k} \arrow[d]\\
K^{(\ell)}_{k+1} \arrow[ur] \arrow[r, "\psi_{k+1,k}"] & K^{(\ell)}_k	
\end{tikzcd}
\end{center}

Henceforth, by taking homotopy functor $\pi_3$ to the above diagram, we get the desired commutative diagram. The second statement follows from homotopy exact sequence of the fibration $\xi_{k} \rightarrow K^{(\ell)}_k$ because its fibre $\widehat{T}_k^{(\ell)}$ is $K(\pi, 1)$ and, in particular, $\pi_2(\widehat{T}_k^{(\ell)})$ and $\pi_3(\widehat{T}_k^{(\ell)})$ vanish. This completes the proof.
\end{proof}

\begin{lemma} \label{eq:7.1.8}
For each $k \geq 2$, the pullback bundle $\xi_k^{(\ell)} \rightarrow K^{(\ell)}_k$ is 2-connected. In particular, the Hurewicz homomorphism $\pi_3(\xi_k^{(\ell)}) \rightarrow H_3(\xi^{(\ell)}_{k})$ is isomorphism.
\end{lemma}

\begin{proof}
Consider the homotopy exact sequence for the fibration $\prod \widehat{\mathbb{S}}^{\infty}{}^{(\ell)} \rightarrow \prod \widehat{\mathbb{CP}}^{\infty}{}^{(\ell)}$, we get
\begin{equation}\label{eq:7.1.9}
	\stackbelow{\pi_2(\widehat{\mathbb{S}}^{\infty}{}^{(\ell)})}{0} \rightarrow \pi_2(\prod \widehat{\mathbb{CP}}^{\infty}{}^{(\ell)}) \rightarrow \pi_1(\widehat{T}_k^{(\ell)}) \rightarrow \stackbelow{\pi_1(\prod \widehat{\mathbb{S}}^{\infty}{}^{(\ell)})}{0}
\end{equation} 
since $\widehat{\mathbb{S}}^{\infty}{}^{(\ell)}$ is contractible. By comparing this with $\xi_k^{(\ell)} \rightarrow K_k^{(\ell)}$, one sees that  $\pi_2(K_k^{(\ell)}) \simeq \pi_1(\widehat{T}_k^{(\ell)})$. Noting that $\widehat{T}_k^{(\ell)}$ is $K(\pi,1)$ space  and $\pi_1(K_k^{(\ell)})=0$, the assertion follows from the following homotopy exact sequence for the fibration $\xi_k^{(\ell)} \rightarrow K_k^{(\ell)}$;
\begin{equation}\label{eq:7.1.10}
	\pi_2(\widehat{T}^{(\ell)}_k) \rightarrow \pi_2(\xi^{(\ell)}) \rightarrow \pi_2(K^{(\ell)}_k) \rightarrow \pi_1(\widehat{T}_k^{(\ell)}) \rightarrow \pi_1(\xi_k^{(\ell)}) \rightarrow \pi_1(K^{(\ell)}_k).
\end{equation}		
\end{proof}

In short, we get the following.

\begin{lemma}\label{lem:7.1.10}
We have an isomorphism
\begin{equation} \label{eq:7.1.11}
	\pi_3(K_k^{(\ell)}) \simeq H_3(\xi_k^{(\ell)}).
\end{equation}
In particular, we have
\begin{equation}\label{eq:7.1.12}
	\pi_3(\widehat{K}_k^{(\ell)}) \simeq  \pi_3(K_k^{(\ell)})\otimes_{\mathbb{Z}} \mathbb{Z}_{\ell} \simeq H_3(\xi_k^{(\ell)}) \otimes_{\mathbb{Z}} \mathbb{Z}_{\ell} \simeq H_3(\xi_k^{(\ell)}; \mathbb{Z}_{\ell}) 
\end{equation}
\end{lemma}
\begin{proof}
The former statement is a consequence of the above discussion. The latter follows from Proposition \ref{lem:4.2.2} (3).	
\end{proof}

Hence, the computation of $\pi_3(\widehat{K}_k^{(\ell)})$ reduces to that of $H_3(\xi_k^{(\ell)};\mathbb{Z}_{\ell})$.

\subsection{Computation of dimension of $H_3(\xi^{(\ell)}_k; \mathbb{Z}_{\ell})$}

In previous section, we have shown that $\pi_3(\widehat{K}_k^{(\ell)}) \simeq H_3(\xi_k^{(\ell)}; \mathbb{Z}_{\ell})$ and the computation of $\pi_3(\widehat{K}_k^{(\ell)})$ reduces to that of $H_3(\xi_k^{(\ell)}; \mathbb{Z}_{\ell})$. To know the rank of $\pi_3(\widehat{K}_k^{(\ell)})$, this section computes $H_3(\xi_k^{(\ell)}; \mathbb{Z}_{\ell})$. 

By taking $G$ as  $\Fn/\Gamma_{k+1}(\Fn)$ in Theorem \ref{thm:6.2.13}, we obatain the chain complex 
\begin{equation} \label{eq:7.2.1}
	D_{\ast}(K^{(\ell)}(k+1); \mathbb{Z}_{\ell}) = C_{\ast}(\Fn/\Gamma_{k+1}(\Fn)) \otimes_{\llbracket \mathbb{Z}_{\ell} \Fn/\Gamma_{k+1}(\Fn) \rrbracket} \mathbb{Z}_{\ell}
\end{equation}
where we consider $\mathbb{Z}_{\ell}$ as trivial $\llbracket \mathbb{Z}_{\ell} \Fn/\Gamma_{k+1}(\Fn) \rrbracket$-module. Note that this identification between chain complex of $K(\pi, 1)$ pro-space and one of $\pi$ is deduced from Lemma \ref{lem:2.3.3}. We consider the subcomplex of $D_{\ast}(K^{(\ell)}(k+1); \mathbb{Z}_{\ell})$ which is generated by the terms $\langle \alpha, \beta_1,\ldots, \beta_{n-1} \rangle$ with $\alpha  \in \mathcal{B}_1$ and $\beta_1,\ldots, \beta_{n-1} \in \mathcal{B}_{k}$. Then, this chain complex can be identified with the chain complex $D_{\ast}(\vee \widehat{S}^1{}^{(\ell)} \times \widehat{T}_k^{(\ell)})$ with induced weight filtration. Note that the subcomplex is chain complex with zero differentials. Then, we have the following weight preserving pushout diagram of chain complexes over $\mathbb{Z}_{\ell}$:
\begin{center}
\begin{tikzcd}
D_{\ast}(\vee \widehat{S}^1{}^{(\ell)} \times \widehat{T}_k^{(\ell)}; \mathbb{Z}_{\ell}) \arrow[r] \arrow[d] & D_{\ast}(K^{(\ell)}(k+1); \mathbb{Z}_{\ell}) \arrow[d] \\
D_{\ast}(\widehat{T}_k^{(\ell)}; \mathbb{Z}_{\ell}) \arrow[r] & D_{\ast}(\xi_k^{(\ell)}; \mathbb{Z}_{\ell})	
\end{tikzcd}
	
\end{center}

Note that from the above pushout diagram, $D_{\ast}(\xi_k^{(\ell)}; \mathbb{Z}_{\ell})$ is endowed with weight filtration.

\begin{lemma} \label{lem:7.2.2} The inclusion  of chain complex $D_{\ast}(\vee \widehat{S}^1{}^{(\ell)} \times \widehat{T}_k^{(\ell)}; \mathbb{Z}_{\ell}) \rightarrow D_{\ast}(K^{(\ell)}(k+1); \mathbb{Z}_{\ell})$ induces the  homomorphism 
\begin{equation} \label{eq:7.2.3}
	H_3(\vee \widehat{S}^1{}^{(\ell)}\times \widehat{T}_k^{(\ell)}; \mathbb{Z}_{\ell}) \rightarrow H_3(K^{(\ell)}(k+1); \mathbb{Z}_{\ell})
\end{equation}
whose image is equal to the $2k+1$-th term of filtration
\begin{equation} \label{eq:7.2.4}
	\mathcal{F}^{2k+1} H_3(K^{(\ell)}(k+1); \mathbb{Z}_{\ell}) \simeq \mathcal{F}^{2k+1} H_3(\Fn / \Gamma_{k+1}(\Fn); \mathbb{Z}_{\ell}).
\end{equation}

Moreover, the composition 
\begin{equation} \label{eq:7.2.5}
	H_3(\widehat{T}_k^{(\ell)}; \mathbb{Z}_{\ell}) \rightarrow  H_3(\vee \widehat{S}^1{}^{(\ell)} \times \widehat{T}_k^{(\ell)}; \mathbb{Z}_{\ell}) \rightarrow H_3(K^{(\ell)}(k+1); \mathbb{Z}_{\ell})
\end{equation}
is the zero homomorphism.
\end{lemma}

\begin{proof}
Noting that  $D_{\ast}(\vee \widehat{S}^1{}^{(\ell)} \times \widehat{T}_k^{(\ell)}; \mathbb{Z}_{\ell})$ and $D_{\ast}(K^{(\ell)}(k+1); \mathbb{Z}_{\ell})$ have weight filtration, consider the May like spectral sequence $'E=\{'E^{r}_{p,q}\}$ and $''E=\{''E^r_{p,q}\}$ associated with filtered chain complexes respcetively. Consider homomorphisms between them induced by the inclusion $D_{\ast}(\vee \widehat{S}^1{}^{(\ell)} \times \widehat{T}_k^{(\ell)}; \mathbb{Z}_{\ell}) \rightarrow D_{\ast}(K^{(\ell)}(k+1); \mathbb{Z}_{\ell})$, then we have the following commutative diagram:
\begin{center}
\begin{tikzcd}
\mathcal{F}^{2k+1}H_3(\vee \widehat{S}^1{}^{(\ell)} \times \widehat{T}_k^{(\ell)}); \mathbb{Z}_{\ell}) \arrow[r] \arrow[d] & \mathcal{F}^{2k+1} H_3(K^{(\ell)}(k+1); \mathbb{Z}_{\ell}) \arrow[d] & H_3(\mathcal{F}^{2k+1}D_{\ast}(K^{(\ell)}(k+1); \mathbb{Z}_{\ell}) \arrow[d] \arrow[l]\\
'E^{\infty}_{2k+1, 3-2k-1} \arrow[r] & ''E^{\infty}_{2k+1, 3-2k-1} \arrow[r] & H_3^{(2k+1)}(\mathcal{L}/\mathcal{L}_{\geq k}; \mathbb{Z}_{\ell})	
\end{tikzcd}
\end{center}
Here, the top right homomorphism is isomorphism by definition of the filtration $\mathcal{F}$ on the third homology group and the right most vertical homomorphism is induced by quotient morphism $\mathcal{F}^{2k+1}D_{\ast}(K^{(\ell)}(k+1) \rightarrow \Lambda^{(2k+1)}(\mathcal{L} / \mathcal{L}_{\geq k}; \mathbb{Z}_{\ell})$ given by taking associated graded Lie algebra.

Since the May like spectral sequences $'E$ and $''E$ converge to $H_{\ast}(\vee \widehat{S}^1{}^{(\ell)} \times \widehat{T}_k^{(\ell)}; \mathbb{Z}_{\ell})$ and $H_{\ast}(K^{(\ell)}(k+1); \mathbb{Z}_{\ell})$ respectively, the left most and the middle vertical homomorphism is surjective. By Corollary \ref{eq:6.3.9} and \eqref{eq:6.3.8}, we know that the bottom right homomorphism is isomorphism. Thus, the right most vertical homomorphism must be surjective.

Noting that $D_{\ast}(\vee \widehat{S}^1{}^{(\ell)} \times \widehat{T}_k^{(\ell)}; \mathbb{Z}_{\ell})$ is concentrated on weight $2k+1$ part, we have
\begin{equation} \label{eq:7.2.6}
	H_3(\vee \widehat{S}^1{}^{(\ell)} \times \widehat{T}_k^{(\ell)}; \mathbb{Z}_{\ell}) \simeq H_3(\mathcal{F}^{2k+1}D_{\ast}(\vee \widehat{S}^1{}^{(\ell)} \times \widehat{T}_k^{(\ell)}); \mathbb{Z}_{\ell}).
\end{equation}
Therefore, from the above arguments, it is enough to show that the inclusion $\mathcal{F}^{2k+1} D_{\ast}(\vee \widehat{S}^1{}^{(\ell)} \times \widehat{T}_k^{(\ell)}; \mathbb{Z}_{\ell}) \rightarrow \Lambda_{\ast}^{2k+1}(\mathcal{L}/\mathcal{L}_{k+1}; \mathbb{Z}_{\ell})$ induces a surjective  homomorphism on third homology groups because their third homology groups are free module with same rank.

For this, we compare two spectral sequences associated to $D_{\ast}(\vee \widehat{S}^1{}^{(\ell)} \times \widehat{T}_k^{(\ell)}; \mathbb{Z}_{\ell})$ and central extension $\mathcal{L}/\mathcal{L}_{\geq k+1} \rightarrow \mathcal{L} / \mathcal{L}_{\geq k}$ of Lie algebra as in \eqref{eq:5.5.3} and \eqref{eq:5.5.4}. Regarding $D_{\ast}(\vee \widehat{S}^1{}^{(\ell)} \times \widehat{T}_k^{(\ell)}; \mathbb{Z}_{\ell})$ as double complex concentrated on degree $(1, 2k)$ with zero differentials, we obtain a first quadrant spectral sequence $'''E=\{'''E_{p,q}^r\}$ with
\begin{equation}\label{eq:7.2.7}
	'''E^2_{p,q} = H_p^{(w-qk)}(\vee \widehat{S}^1{}^{(\ell)}; \mathbb{Z}_{\ell}) \otimes_{\mathbb{Z}_{\ell}} \Lambda_{q}(\Fn/ \Gamma_{k+1}(\Fn)), \quad '''E^2_{p,q} \Rightarrow H_{p+q}^{(w)}(\vee \widehat{S}^1{}^{(\ell)} \times \widehat{T}_k^{(\ell)}; \mathbb{Z}_{\ell}).
\end{equation}
Then, by construction, we see that this spectral sequence collapses, that is $'''E^{\infty}_{p,q} \simeq '''E^2_{p,q}$, and so
\begin{equation}\label{eq:7.2.8}
	'''E^{\infty}_{p,q} = H_{p+q}^{(w)}(\vee \widehat{S}^1{}^{(\ell)} \times \widehat{T}_k^{(\ell)}; \mathbb{Z}_{\ell}) \simeq H_p^{(w -qk)}(\vee \widehat{S}^1{}^{(\ell)}; \mathbb{Z}_{\ell}) \otimes_{\mathbb{Z}_{\ell}} \Lambda_q(\Fn/\Gamma_{k+1}(\Fn); \mathbb{Z}_{\ell})='''E^2_{p,q}.
\end{equation}
	On the other hand, we have the Hochschild-Serre spectral sequence $''''E=\{''''E^r_{p,q} \}$ converging to $H_{p+q}(\mathcal{L} /\mathcal{L}_{\geq k+1}; \mathbb{Z}_{\ell})$ given in \eqref{eq:5.5.3} and \eqref{eq:5.5.4}. Let us compare the spectral sequences for  weight $2k+1$ part. We see that the only $E^2_{1,2}$ terms are non-zero in both spectral sequences with weight $2k+1$. Moreover, by comparing explicit form of $E^2_{p,q}$ terms, one sees that they are isomorphic. Since the former spectral sequence $'''E$ collapses, the homomorphism $'''E^{\infty}_{1,2} \rightarrow ''''E^{\infty}_{1,2}$ is surjective. As we see in \eqref{eq:5.5.3} and \eqref{eq:5.5.4}, we have $''''E^{\infty}_{1,2} = H^{(2k+1)}_3(\mathcal{L}/\mathcal{L}_{\geq k+1}; \mathbb{Z}_{\ell})$, so we conclude that $H_3(\mathcal{F}^{2k+1} D_{\ast}(\vee \widehat{S}^1{}^{(\ell)} \times \widehat{T}_k^{(\ell)}); \mathbb{Z}_{\ell}) \rightarrow H^{(2k+1)}_3(\mathcal{L}/\mathcal{L}_{\geq k+1}; \mathbb{Z}_{\ell})$ is surjective and we get the desired isomorphism \eqref{eq:7.2.5}.
	
	Since the split subgroup $H_3(\widehat{T}_k^{(\ell)}; \mathbb{Z}_{\ell})$ in $H_3(\vee \widehat{S}^1{}^{(\ell)} \times \widehat{T}_k^{(\ell)}; \mathbb{Z}_{\ell})$  concentrates on weight $3k$ part, by Corollary \ref{eq:6.3.9} $\mathcal{F}^{2k+2} H_3(K^{(\ell)}(k+1); \mathbb{Z}_{\ell}) =0$, and inclusions induce weight preserving homomorphisms on homology groups,  we conclude that the image of the composition
	\begin{equation} \label{eq:7.2.9}
		H_3(\widehat{T}_k^{(\ell)}; \mathbb{Z}_{\ell}) \rightarrow  H_3(\vee \widehat{S}^1 \times \widehat{T}_k^{(\ell)}, \mathbb{Z}_{\ell}) \rightarrow H_3(K^{(\ell)}(k+1); \mathbb{Z}_{\ell})
	\end{equation}
	must be zero.
\end{proof}

\begin{proposition} \label{prop:7.2.10}
The cokernel of the induced homomorphism $H_3(K^{(\ell)}(k+1); \mathbb{Z}_{\ell}) \rightarrow H_3(\xi^{(\ell)}_k; \mathbb{Z}_{\ell})$ is free $\mathbb{Z}_{\ell}$-module of rank $nN_k - N_{k+1}$;
\begin{equation} \label{eq:7.2.11}
	\coker(H_3(K^{(\ell)}(k+1); \mathbb{Z}_{\ell}) \rightarrow H_3(\xi^{(\ell)}_k; \mathbb{Z}_{\ell})) \simeq \mathbb{Z}_{\ell}^{\oplus (nN_k - N_{k+1})}.
\end{equation}

In addition, the third homotopy group $\pi_3(\widehat{K}_k^{(\ell)})$ of pro-$\ell$ (complete) Orr space is free $\mathbb{Z}_{\ell}$-module  whose rank is computed as 
\begin{equation}\label{eq:7.2.12}
	\pi_3(\widehat{K}_k^{(\ell)}) \simeq H_3(\xi^{(\ell)}_k; \mathbb{Z}_{\ell}) \simeq \bigoplus_{i=k}^{2k-1} \mathbb{Z}_{\ell}^{\oplus (nN_i - N_{i+1})}
\end{equation}
\end{proposition}

\begin{proof}
By considering $\xi_k^{(\ell)}$ as the pushout of $ \vee \widehat{S}^1{}^{(\ell)} \times \widehat{T}_k^{(\ell)} \rightarrow K^{(\ell)}(k+1)$ and $\vee \widehat{S}^1{}^{(\ell)} \times \widehat{T}_k^{(\ell)}	\rightarrow T_k^{(\ell)}$, its associated Mayer-Vietoris exact sequence is given as follows.
\begin{center}
\begin{tikzcd}
H_3(\vee \widehat{S}^1{}^{(\ell)} \times \widehat{T}_k^{(\ell)}; \mathbb{Z}_{\ell}) \arrow[r]
& H_3(K^{(\ell)}(k+1); \mathbb{Z}_{\ell}) \oplus H_3(\widehat{T}_k^{(\ell)}; \mathbb{Z}_{\ell}) \arrow[r]
\arrow[d, phantom, ""{coordinate, name=Z}]
& H_3(\xi_k^{(\ell)}; \mathbb{Z}_{\ell}) \arrow[dll,
rounded corners=3mm,
to path={ -- ([xshift=2ex]\tikztostart.east)
|- (Z) [near end]\tikztonodes
-| ([xshift=-2ex]\tikztotarget.west)
-- (\tikztotarget)}] \\
H_2(\vee \widehat{S}^1{}^{(\ell)} \times \widehat{T}_k^{(\ell)}; \mathbb{Z}_{\ell}) \arrow[r]
& H_2(K^{(\ell)}(k+1); \mathbb{Z}_{\ell}) \oplus H_2(\widehat{T}_k^{(\ell)}; \mathbb{Z}_{\ell}) \arrow[r]
& 0
\end{tikzcd}
\end{center}
Here, note that $H_2(\xi_k^{(\ell)}; \mathbb{Z}_{\ell})=0$ since $\xi_k^{(\ell)}$ is 2-connected. Also, note that morphisms in the above exact sequence preserve filtration of homology groups since morphisms of chain complexes preserve filtration. 

Noting that the map $\vee \widehat{S}^1{}^{(\ell)} \times \widehat{T}_k^{(\ell)} \rightarrow \widehat{T}_k^{(\ell)}$ split and using Lemma \ref{lem:7.2.2}, the above Mayer-Vietoris exact sequence induces the following commutative diagram whose bottom row is exact:

\begin{center}
\begin{tikzcd}
\mathcal{F}^{2k+1} H_3(\vee \widehat{S}^1{}^{(\ell)} \times \widehat{T}_k^{(\ell)}; \mathbb{Z}_{\ell}) \ar[r]	\ar[d, "="'] & \mathcal{F}^{2k+1}H_3(K^{(\ell)}(k+1); \mathbb{Z}_{\ell}) \ar[d] & & \\
H_3(\vee \widehat{S}^1{}^{(\ell)} \times \widehat{T}_k^{(\ell)}; \mathbb{Z}_{\ell})  \ar[r] & H_3(K^{(\ell)}(k+1); \mathbb{Z}_{\ell}) \ar[r]& H_3(\xi_k^{(\ell)}; \mathbb{Z}_{\ell}) \ar[r] & \cdots
\end{tikzcd}
\end{center}

Since the top horizontal homomorphism is onto by Lemma \ref{lem:7.2.2}, noting again that the map $\vee \widehat{S}^1{}^{(\ell)} \times \widehat{T}_k^{(\ell)} \rightarrow \widehat{T}_k^{(\ell)}$ split, we obtain the following exact sequence:

\begin{center}
\begin{tikzcd}
0 \arrow[r]
& \displaystyle\frac{H_3(K^{(\ell)}(k+1); \mathbb{Z}_{\ell})}{\mathcal{F}^{2k+1}H_3(K^{(\ell)}(k+1); \mathbb{Z}_{\ell})} \arrow[r]
\arrow[d, phantom, ""{coordinate, name=Z}]
& H_3(\xi_k^{(\ell)}; \mathbb{Z}_{\ell}) \arrow[dll,
rounded corners=3mm,
to path={ -- ([xshift=2ex]\tikztostart.east)
|- (Z) [near end]\tikztonodes
-| ([xshift=-2ex]\tikztotarget.west)
-- (\tikztotarget)}] \\
H_1(\widehat{T}_k^{(\ell)}; \mathbb{Z}_{\ell})\otimes_{\mathbb{Z}_{\ell}} \mathbb{Z}_{\ell}^{\oplus n} \arrow[r]
& H_2(K^{(\ell)}(k+1); \mathbb{Z}_{\ell}) \arrow[r]
& 0.
\end{tikzcd}
\end{center}
Here, we use the fact that $H_1(\Fn;  \mathbb{Z}_{\ell}) \simeq \mathbb{Z}_{\ell}^{\oplus n}$ and $H^i(\Fn; \mathbb{Z}_{\ell})=0$ for $i \geq 2$. The former is \cite[Lemma 6.8.6 (b)]{RZ} and the latter follows from the fact that the cohomological $\ell$-dimension of $\Fn$ is 1 and the existence of duality between homology groups and cohomology groups (\cite[Proposition 6.3.6]{RZ}). Using the result of Corollary \ref{eq:6.3.9}, the sequence leads to the following exact sequence
\begin{equation} \label{eq:7.2.13}
	0 \rightarrow \bigoplus_{i=k+1}^{2k-1} \mathbb{Z}_{\ell}^{\oplus (nN_i - N_{i+1})} \rightarrow H_3(\xi^{(\ell)}_k; \mathbb{Z}_{\ell}) \rightarrow \mathbb{Z}_{\ell}^{\oplus n N_k} \rightarrow \mathbb{Z}_{\ell}^{\oplus N_{k+1}} \rightarrow 0.
\end{equation}
Since all the relevant groups in the above sequence are free $\mathbb{Z}_{\ell}$-modules, by dimension counting, the assertion follows.
\end{proof}

\begin{lemma}\label{lem:7.2.14}
The  homomorphism $\gamma_{k+1,k}: H_3(\xi_{k+1}^{(\ell)}; \mathbb{Z}_{\ell}) \rightarrow H_3(\xi^{(\ell)}_k; \mathbb{Z}_{\ell})$ factors through a homomorphism $H_3(K^{(\ell)}(k+1); \mathbb{Z}_{\ell}) \rightarrow H_3(\xi^{(\ell)}_k; \mathbb{Z}_{\ell})$. Moreover, the resulting homomorphism $H_3(\xi^{(\ell)}_{k+1}; \mathbb{Z}_{\ell}) \rightarrow H_3(K^{(\ell)}(k+1); \mathbb{Z}_{\ell})$ is surjective. That is to say, we have the following commutative diagram:
\begin{center}
\begin{tikzcd}
H_3(\xi^{(\ell)}_{k+1}; \mathbb{Z}_{\ell}) \ar[r, "\gamma_{k+1,k}"] \ar[d, twoheadrightarrow, dashed] & H_3(\xi^{(\ell)}_{k}; \mathbb{Z}_{\ell}) \\
H_3(K^{(\ell)}(k+1); \mathbb{Z}_{\ell}) \ar[ur] & 
\end{tikzcd}
	
\end{center}

\end{lemma}

\begin{proof}
	For a chain complex $C_{\ast}$, we denote by $\sigma_{\geq 2}C_{\ast}$ the ``stupid'' truncation $\sigma_{\geq 2}$ of $C_{\ast}$, that is, we set
	\begin{equation} \label{eq:7.2.15}
		\sigma_{\geq 2}C_{n} = \begin{cases}
 C_{n} & (n \geq 2) \\
 0 & (n \leq 1 ).
 \end{cases}
\end{equation}
Then, by definition, we have $H_{i}(C_{\ast}) \simeq  H_i(\sigma_{\geq 2}C_{\ast})$ for $i \geq 3$. Since the quotient by $K(\Fn, 1)$ does note affect to $D_{i}(K^{(\ell)}(k); \mathbb{Z}_{\ell})$ with degree $i \geq 2$, we have $\sigma_{\geq 2}D_{\ast}(K^{(\ell)}(k); \mathbb{Z}_{\ell}) \simeq \sigma_{\geq 2} D_{\ast}(K^{(\ell)}_k; \mathbb{Z}_{\ell})$. Henceforth, the morphisms of pro-spaces $K^{(\ell)}(k+1) \rightarrow \xi^{(\ell)}_{k+1}$, $\xi^{(\ell)}_{k+1} \rightarrow K^{(\ell)}_{k+1}$ and $K^{(\ell)}(k+1) \rightarrow K^{(\ell)}_{k+1}$ induces the morphisms of chain complexes as follows.
\begin{equation}\label{eq:7.2.16}
	\sigma_{\geq 2}D_{\ast}(\xi^{(\ell)}_{k+1}; \mathbb{Z}_{\ell}) \rightarrow \sigma_{\geq 2} D_{\ast}(K^{(\ell)}_{k+1}; \mathbb{Z}_{\ell}) \rightarrow \sigma_{\geq 2}D_{\ast}(K^{(\ell)}(k+1); \mathbb{Z}_{\ell}) \rightarrow \sigma_{\geq 2}D_{\ast}(\xi^{(\ell)}_k; \mathbb{Z}_{\ell}).
\end{equation}
Since the middle homomorphism is isomorphism noted as above, by taking homology group functor, we get the desired commutative diagram. The surjectivity of $H_3(\xi^{(\ell)}_{k+1}; \mathbb{Z}_{\ell}) \rightarrow H_3(K^{(\ell)}(k+1); \mathbb{Z}_{\ell})$ follows from that of $\xi^{(\ell)}_{k+1} \rightarrow K^{(\ell)}_{k+1}$. Thus, the assertion has been proved.
\end{proof}

\begin{lemma} \label{lem:7.2.17}
For homomorphisms $\gamma_{k+1,k} : \pi_3(\xi^{(\ell)}_{k+1}) \rightarrow \pi_3(\xi^{(\ell)}_k)$ as in Lemma \ref{lem:7.2.14}, we have
\begin{equation}\label{eq:7.2.18}
	\ker(\gamma_{k+1, k}) \subset \Im(\gamma_{k+2, k+1})
\end{equation}
\end{lemma}

\begin{proof}
	Consider the following commutative diagram whose horizontal sequences are exact:
	\begin{center}
	\begin{tikzcd}
		0 \ar[r] & \displaystyle\frac{H_3(K^{(\ell)}(k+2); \mathbb{Z}_{\ell})}{\mathcal{F}^{2k+3}H_3(K^{(\ell)}(k+2); \mathbb{Z}_{\ell})} \ar[r, "\alpha"] \ar[d, "\epsilon"']  & H_3(\xi^{(\ell)}_{k+1}; \mathbb{Z}_{\ell}) \ar[r, "i"] \ar[dl, "\gamma", dashed] \ar[d, "\gamma_{k+1,k}"]  & \coker(\alpha) \ar[r] \ar[d, "\rho"]& 0\\
		0 \ar[r] & \displaystyle\frac{H_3(K^{(\ell)}(k+1); \mathbb{Z}_{\ell})}{\mathcal{F}^{2k+1} H_3(K^{(\ell)}(k+1); \mathbb{Z}_{\ell})} \ar[r, "\beta"] & H_3(\xi^{(\ell)}_k; \mathbb{Z}_{\ell}) \ar[r, "j"] & \coker(\beta) \ar[r] & 0.
	\end{tikzcd}	
	\end{center}
	Here, $\rho$ is the induced homomorphism from the left most square of the diagram and $\gamma$ is the induced map from the factorization of $\gamma_{k+1,k}$ in Lemma \ref{lem:7.2.14}. Applying the results of Corollary \ref{eq:6.3.9} and Proposition \ref{prop:7.2.10} to the top and bottom short exact sequences, we can compute the rank of $\coker(\alpha)$ and $\coker(\beta)$ as 
	\begin{eqnarray*}
		\coker(\alpha)& \simeq &\mathbb{Z}_{\ell}^{\oplus (nN_{k+1} - N_{k+2})}, \\
		 \coker(\beta)& \simeq & \mathbb{Z}_{\ell}^{\oplus (nN_{k} - N_{k+1})}.
	\end{eqnarray*}
Now, the homomorphism $\rho$ is zero homomorphism, since $\gamma$ is onto by Lemma \ref{lem:7.2.14}. In fact, for any element $x \in \coker(\alpha)$ there is an element $c \in H_3(\xi_{k+1}^{(\ell)}; \mathbb{Z}_{\ell})$. In fact, since $\rho \circ i = j \circ \gamma_{k+1,k}$ and $\Im (\gamma_{k+1, k}) \simeq H_3(\xi^{(\ell)}_{k+1}; \mathbb{Z}_{\ell})/\ker (\gamma_{k+1,k}) \simeq H_3(\xi^{(\ell)}_{k+1}; \mathbb{Z}_{\ell})/\ker (\gamma) \simeq \Im(\gamma)$, so $\Im(\gamma_{k+1,k}) \simeq \Im (\beta \circ \gamma)$ and therefore $\rho \circ i = j \circ \beta \circ \gamma = 0$. This implies $\rho = 0$.

Then, by snake lemma, we have the following exact sequence

\begin{center}
\begin{tikzcd}
0 \arrow[r]
& \ker(\epsilon) \arrow[r]
& \ker(\gamma_{k+1,k}) \arrow[r] & \coker(\alpha)  \arrow[dlll,
rounded corners=2mm,"\delta",
to path={ -- ([xshift=2ex]\tikztostart.east)
|- (Z) [near end]\tikztonodes
-| ([xshift=-2ex]\tikztotarget.west)
-- (\tikztotarget)}] \\
\coker(\epsilon) \arrow[r]
& \coker(\gamma_{k+1,k}) \arrow[r]
& \coker(\beta) \arrow[r] & 0
\end{tikzcd}
\end{center}
where we use $\ker(\rho)=\coker(\alpha)$ mentioned as above. Since we have $\coker(\alpha) = \coker(\epsilon)$ and $\coker(\gamma_{k+1,k})=\coker(\beta)$ by counting dimension, we conclude that $\ker(\gamma_{k+1,k}) \subset \Im(\alpha)$. Hence, this completes the proof.
\end{proof}

We introduce a (descending) filtration $\mathcal{G}^{\bullet}$ to $H_3(\xi^{(\ell)}_k; \mathbb{Z}_{\ell})$ by setting
\begin{equation}\label{eq:7.2.19}
	\mathcal{G}^l H_3(\xi^{(\ell)}_k; \mathbb{Z}_{\ell}) = \Im(\gamma_{l-1,k} : H_3(\xi^{(\ell)}_{l-1}; \mathbb{Z}_{\ell}) \rightarrow H_3(\xi^{(\ell)}_k; \mathbb{Z}_{\ell}))
\end{equation}
where the homomorphism $\gamma_{l-1, k}$ is defined as composition $\gamma_{l-1,k}=\gamma_{l-1, l-2}\circ \cdots \circ \gamma_{k+1, k}$. Note that the isomorphism $\pi_3(\xi^{(\ell)}_k) \rightarrow \pi_3(K^{(\ell)}_k)$ in Lemma maps the $l$-th term of the filtration $\mathcal{G}^l H_3(\xi^{(\ell)}_k; \mathbb{Z}_{\ell})$ to the $l$-th term of the filtration $\mathcal{G}^l\pi_3(\widehat{K}^{(\ell)}_k)$. So, we have
\begin{equation}\label{eq:7.2.20}
	\coker(\gamma_{k+1,k}) \simeq \coker(\psi_{k+1,k}) \simeq \mathbb{Z}_{\ell}^{\oplus(nN_k - N_{k+1})}.
\end{equation}

\begin{theorem} \label{thm:7.2.21}
Let $\widehat{K}^{(\ell)}_k$ be the complete pro-$\ell$ Orr space. Then, the following statements hold:

\begin{enumerate}[label=$(\arabic{enumi})$]
\item For the filtration $\mathcal{G}^l \widehat{K}^{(\ell)}_k$ is actually given as follows:
\begin{equation} \label{eq:7.2.22}
	0= \mathcal{G}^{2k+1} \pi_3(\widehat{K}^{(\ell)}_k) \subset \cdots \subset \mathcal{G}^{k+1} \pi_3(\widehat{K}^{(\ell)}_k) = \pi_3(\widehat{K}^{(\ell)}_k).
\end{equation}
\item  For each integer $k \leq l \leq 2k-1$, the associated graded quotient is isomorphic to rank $nN_l - N_{l+1}$ free $\mathbb{Z}_{\ell}$-module, i.e., there is an isomorphism
\begin{equation}\label{eq:7.2.23}
	\Phi_l : \frac{\mathcal{F}^{l+1} \pi_3(\widehat{K}^{(\ell)}_k)}{\mathcal{F}^{l+2} \pi_3(\widehat{K}^{(\ell)}_k)} \overset{\sim}{\longrightarrow} \mathbb{Z}_{\ell}^{\oplus (nN_l - N_{l+1})}. 
\end{equation}
\item The homomorphisms
\begin{equation}\label{eq:7.2.24}
	\psi_{k+1, k} : \pi_3(\widehat{K}^{(\ell)}_{k+1}) \rightarrow \pi_3(\widehat{K}^{(\ell)}_k)
\end{equation}
preserve the filtration $\mathcal{G}^{\bullet}\pi_3(K^{(\ell)}_{\ast})$, and   induced homomorphism on the associated graded quotients 
\begin{equation}\label{eq:7.2.25}
	\frac{\mathcal{G}^l \pi_3(\widehat{K}^{(\ell)}_{k+1})}{\mathcal{G}^{l+1} \pi_3(\widehat{K}^{(\ell)}_{k+1})} \rightarrow \frac{\mathcal{G}^l \pi_3(\widehat{K}^{(\ell)}_{k})}{\mathcal{G}^{l+1} \pi_3(\widehat{K}^{(\ell)}_k)}
\end{equation}
	are isomorphisms for $k+2 \leq l \leq 2k$ and zero homomorphisms otherwise.
\item In particular, for each $k\geq 2$, there is an isomorphism
\begin{equation} \label{eq:7.2.26}
	\pi_3(\widehat{K}^{(\ell)}_k) \simeq \bigoplus_{i=k}^{2k-1} \mathbb{Z}_{\ell}^{\oplus (nN_i - N_{i+1})}
\end{equation}
and
\begin{equation}\label{eq:7.2.27}
	x \in \ker(\psi_{l,k}) \text{if and only if} \ x \in \Im(\psi_{2k,l}).
\end{equation}
\end{enumerate}
\end{theorem}

\begin{proof}
	Consider the following exact sequence
	\begin{equation} \label{eq:7.2.28}
		0 \rightarrow \ker(\gamma_{k+1,k}) \rightarrow H_3(\xi^{(\ell)}_{k+1}; \mathbb{Z}_{\ell}) \overset{\gamma_{k+1,k}}{\rightarrow} H_3(\xi^{(\ell)}_k; \mathbb{Z}_{\ell}) \rightarrow \coker(\gamma_{k+1,k}) \rightarrow 0.
	\end{equation}
	By inserting the result of Proposition \ref{prop:7.2.10} and $\coker(\gamma_{k+1,k}) \simeq \mathbb{Z}_{\ell}^{\oplus(nN_k - N_{k+1})}$, the above sequence becomes
	\begin{equation} \label{eq:7.2.29}
		0 \rightarrow \ker(\gamma_{k+1, k}) \rightarrow \bigoplus_{i=k+1}^{2k+1} \mathbb{Z}_{\ell}^{\oplus(nN_i - N_{i+1})} \rightarrow \bigoplus_{i=k}^{2k-1} \mathbb{Z}_{\ell}^{\oplus (nN_i - N_{i+1})} \rightarrow \mathbb{Z}_{\ell}^{\oplus(nN_{k} - N_{k+1})} \rightarrow 0.
	\end{equation}
	Therefore, dimension counting implies that
	\begin{equation} \label{eq:7.2.30}
		\ker(\gamma_{k+1,k}) \simeq \bigoplus_{i=2k}^{2k+1} \mathbb{Z}_{\ell}^{\oplus(nN_i - N_{i+1})}.
	\end{equation}
	Since all homology groups $H_3(\xi^{(\ell)}_m; \mathbb{Z}_{\ell})$ are free $\mathbb{Z}_{\ell}$-modules, in terms of Lemma \ref{lem:7.2.17}, we have
	\begin{equation} \label{eq:7.2.32}
		\ker(\gamma_{l,k}) \simeq \bigoplus_{i=2k}^{2l-1} \mathbb{Z}_{\ell}^{\oplus (nN_i - N_{i+1})}.
	\end{equation}
	Hence, by group isomorphism theorem, we get
	\begin{equation} \label{eq:7.2.33}
		\Im(\gamma_{l,k}) \simeq \bigoplus_{i=l}^{2k-1} \mathbb{Z}_{\ell}^{\oplus (nN_i - N_{i+1})}.
	\end{equation}
	By definition of the filtration $\mathcal{G}^{\bullet}$, its associated graded quotient is given as the form
	\begin{equation} \label{eq:7.2.34}
		\frac{\mathcal{G}^{l}H_3(\xi_{k}^{(\ell)}; \mathbb{Z}_{\ell})}{\mathcal{G}^{l+1}H_3(\xi_{k}^{(\ell)}; \mathbb{Z}_{\ell})} \simeq \mathbb{Z}_{\ell}^{\oplus (nN_{l-1} - N_{l})} \oplus \text{torsion}.
	\end{equation}
	Now, we know that 
	\begin{equation} \label{eq:7.2.35}
		\frac{\mathcal{G}^{l}H_3(\xi_{l-1}^{(\ell)}; \mathbb{Z}_{\ell})}{\mathcal{G}^{l+1}H_3(\xi_{l-1}^{(\ell)}; \mathbb{Z}_{\ell})} \simeq \mathbb{Z}_{\ell}^{\oplus (nN_{l-1} - N_{l})}
	\end{equation}
	and $\mathcal{G}^l H_3(\xi^{(\ell)}_{l-1}; \mathbb{Z}_{\ell}) \rightarrow \mathcal{G}^{l}H_3(\xi^{(\ell)}_k; \mathbb{Z}_{\ell})$ is surjective by definition, so the torsion part of equation \eqref{eq:7.2.34} must be zero. So we conclude that
	\begin{equation} \label{eq:7.2.36}
		\frac{\mathcal{G}^{l}H_3(\xi_{k}^{(\ell)}; \mathbb{Z}_{\ell})}{\mathcal{G}^{l+1}H_3(\xi_{k}^{(\ell)}; \mathbb{Z}_{\ell})} \simeq \mathbb{Z}_{\ell}^{\oplus (nN_{l-1} - N_{l})}.
	\end{equation}
	By proposition \ref{prop:7.2.10} and counting dimension, we have $\mathcal{G}^{2k+1}H_3(\xi^{(\ell)}_k; \mathbb{Z}_{\ell}) \simeq 0$.
\end{proof}

\section{Massey products and  $H^{2}(\Fn/\Gamma_k(\Fn); \mathbb{F})$} \label{sec:7}

In this section, we recall the notion of Massey products for pro-$\ell$ group cohomology. Then, we describe a basis of  the dual module $\Hom_{\mathbb{Z}_{\ell}}(H_2(\Fn/\Gamma_k(\Fn); \mathbb{Z}_{\ell}), \mathbb{Z}_{\ell})$ and the second cohomology group $H^2(\Fn/\Gamma_k(\Fn); \mathbb{F})$ via  Massey products. Here, $\mathbb{F}$ denotes the field of $\ell$-adic numbers $\mathbb{Q}_{\ell}$ or finite field $\mathbb{F}_{\ell}$.

\subsection{Higher Massey products for pro-$\ell$ group cohomology}

Here, we recall the basics of higher Massey product for pro-$\ell$ group cohomology following \cite{Dwy} and \cite{Mo1}.

Let $G$ be a profinite group and let $A$ be a discrete $G$-module. Let us consider the inhomogeneous cochain complex of $G$ with coefficients in $A$ as in \cite[Section 2]{NSW} (see also  \cite[Section 6.4]{RZ}), that is, for $k \geq 0$, the module $C^k(G;A)$ of $k$-cochains is given by the module of all continuous maps $\mathrm{Map}(G^k, A)$ and coboundary operator $d: C^k(G;A) \rightarrow C^{k+1}(G; A)$ is given by
\begin{eqnarray} \label{eq:-1.1.1}
	&&(d f)(g_1, \ldots, g_{k+1}) = g_1 f(g_2, \ldots, g_{k+1}) \nonumber  \\
	&& \quad +\sum_{i=1}^{k} (-1)^i f(g_1,\ldots, g_{i-1}, g_i g_{i+1}, g_{i+2}, \ldots, g_{k+1}) + (-1)^{k+1} f(g_1, \ldots, g_{k}).
\end{eqnarray}
The cup product $\smile : C^p(G; A) \times C^q(G; A) \rightarrow C^{p+q}(G; A)$ is defined as, for $u \in C^p(G;A)$ and $v \in C^q(G;A)$, 
\begin{equation} \label{eq:-1.1.2}
	(u\smile v)(g_1, \ldots, g_{p+q}):= (-1)^{pq} u(g_1, \ldots, g_p)\cdot v(g_{p+1}, \ldots, g_{p+q}) \in A.
\end{equation}
Let $\alpha_1, \ldots, \alpha_m \in H^1(G; A)$ be first cohomology classes. A {\it Massey product}  $\langle \alpha_1, \ldots, \alpha_m \rangle$ is said to be {\it defined} if there is an array $\mathcal{A}$
\begin{equation} \label{eq:-1.1.3}
	\mathcal{A} = \{ a_{ij} \in C^1(G; A) \mid 1 \leq i < j \leq m+1,\ (i,j) \neq (1, m+1) \}
\end{equation}
such that
\begin{equation} \label{eq:-1.1.4}
	\begin{cases}
	[a_{i,i+1}] = \alpha_i & (1\leq i \leq m) \\
	da_{ij} = \sum_{k=i+1}^{j-1} a_{ik} \smile a_{kj} & (j \neq i+1).
\end{cases}
\end{equation} 

We call such an array $\mathcal{A}$ a {\it defining system} for $\langle \alpha_1, \ldots, \alpha_m \rangle$. Then, for a defining system $\mathcal{A}$, the cohomology class $\langle \alpha_1, \ldots, \alpha_m \rangle_{\mathcal{A}}$ in $H^2(G; A)$ is defined as the cohomology class represented by the 2-cocycle
\begin{equation} \label{eq:-1.1.5}
	\sum_{k=2}^m a_{1k} \smile a_{k, m+1}.
\end{equation}
Then, we define a Massey product of $\alpha_1, \ldots, \alpha_m$ as the subset of $H^2(G; A)$ given by
\begin{equation} \label{eq:-1.1.6}
	\langle \alpha_1, \ldots, \alpha_m \rangle := \{ \langle \alpha_1, \ldots, \alpha_m \rangle_{\mathcal{A}} \in H^2(G; A)\  | \ \mathcal{A} \ \text{is defining system for } \langle \alpha_1, \ldots, \alpha_m\rangle \}.
\end{equation}

\begin{remark} \label{rem:-1.1.7}
	(1) The Massey product $\langle \alpha_1 \rangle$ is just $\alpha_1$ and its defining system $\mathcal{A}$ consists of any 1-cocycle representing $\alpha_1$. The Massey product $\langle \alpha_1, \alpha_2 \rangle$ is the cup product $\alpha_1 \smile \alpha_2$.  \\
	(2) (Uniqueness) For any integer $m \geq 3$, the Massey product $\langle \alpha_1, \ldots, \alpha_m \rangle$ is defined and uniquely defined if $\langle \alpha_{i_1}, \ldots, \alpha_{i_r} \rangle = 0$ for any proper subset $\{ i_1, \ldots, i_r \}$ of $\{1, \ldots, m\}$ with $r \geq 1$.\\
	(3) (Naturality) Let $G$ and $G'$ be profinite groups and $f : G \rightarrow G'$ be a continuous homomorphism. Assume that $\langle \alpha_1, \ldots, \alpha_m \rangle$ is defined for $\alpha_1, \ldots, \alpha_m \in H^1(G'; A)$ with defining system $\mathcal{A} = (a_{ij})$. Then, the Massey product $\langle f^{\ast}(\alpha_1), \ldots, f^{\ast}(\alpha_m) \rangle$ is defined for $f^{\ast}(\alpha_1), \ldots, f^{\ast}(\alpha_m) \in H^1(G; A)$ with the defining system $\mathcal{A}^{\ast} = (f^{\ast}(a_{ij}))$ and we have $f^{\ast}(\langle \alpha_1, \ldots, \alpha_m \rangle) \subset \langle f^{\ast}(\alpha_1), \ldots, f^{\ast}(\alpha_m) \rangle$.
\end{remark}

Next, let us recall interpretation of a defining system for the Massey product as uni-triangular matrix (cf. \cite{Dwy} and \cite{St}). Let $U_{m+1}( A)$ denote the multiplicative group of upper triangular matrices coefficients in $A$ whose diagonal entries are identity. Similarly, we denote by $Z_{m+1}(A)$ the subgroup of $U_{m+1}(A)$ whose non-zero entries are only diagonal entries and $(1, m+1)$-entry. Then, noting that $Z_{m+1}(A)$ is identified with $A$, we have the following exact sequence
\begin{equation} \label{eq:-1.1.8}
	1 \rightarrow A \rightarrow U_{m+1}(A) \rightarrow \overline{U}_{m+1}(A) \rightarrow 1
\end{equation}
where $\overline{U}_{m+1}(A)$ is the quotient group $U_{m+1}(A)/Z_{m+1}(A)$. Now, for the defining system $\mathcal{A}=(a_{ij})$ for $\langle \alpha_1, \ldots, \alpha_m \rangle$ can be identified with a homomorphism
\begin{equation} \label{eq:-1.1.9}
	\phi_{\mathcal{A}} : G \rightarrow \overline{U}_{m+1}(A)
\end{equation}
which sends $g \in G$ to $\phi_{\mathcal{A}}(g):=(\phi_{ij}(g)):=(-a_{ij}(g))$ where  we set $-a_{ii}(g)=1$ for each $1\leq i \leq m$. Indeed, by the condition \eqref{eq:-1.1.4}, one obtains
\begin{equation} \label{eq:-1.1.10}
	d a_{ij}(g_1 g_2) = a_{ij}(g_1) + a_{ij}(g_2) - a_{ij}(g_1g_2)= \sum_{k=i+1}^{j-1} a_{ik}(g_1)a_{kj}(g_2)
\end{equation}
and so
\begin{equation} \label{eq:-1.1.11}
	\phi_{ij}(g_1 g_2) = \phi_{ij}(g_1) + \phi_{ij}(g_2) + \sum_{k=i+1}^{j-1} \phi_{ik}(g_1) \phi_{kj}(g_2).
\end{equation}
Besides, from this interpretation, one easily sees that the Massey product $\langle \alpha_1, \ldots, \alpha_m \rangle$ plays a role of obstruction class for existence of a lift of $\phi_{\mathcal{A}} : G \rightarrow \overline{U}_{m+1}(A)$ to $U_{m+1}(A)$. More precisely, the cohomology class $\langle \alpha_1, \ldots, \alpha_m \rangle_{\mathcal{A}} =0$ if and only if $\phi_{\mathcal{A}} : G \rightarrow \overline{U}_{m+1}(A)$ extends to a homomorphism $G \rightarrow U_{m+1}(A)$.

Finally, we recall the relation between Massey products and Magnus coefficients.

Let $G$ denote a pro-$\ell$ group with minimal presentation given by
\begin{equation} \label{eq:-1.1.12}
	1 \rightarrow R \rightarrow \Fn \overset{\pi}{\rightarrow} G \rightarrow 1.
\end{equation} 
We set $\pi(x_i)=:g_i$ for $1 \leq i \leq n$. Then, we recall the transgression map
\begin{equation} \label{eq:-1.1.13}
	\mathrm{tg} : H^1(R; A)^G \rightarrow H^2(G; A)
\end{equation}
associated with Hochschild-Serre spectral sequence defined as follows. For a 1-cocycle $a$ of $H^1(R; A)^G$, there exists a 1-cochain $b \in Z^1(\Fn; A)$ with $b|_R = a$ and $b(\sigma \tau) = b(\tau \sigma) = b(\sigma) + b(\tau)$ for $\sigma \in \Fn$ and $\tau \in R$. Then, the value $db(z_1 z_2)$ depends only cosets $z_1 \mod R$ and $z_2 \mod R$, so there is a 2-cocycle $c \in Z^1(G; A)$ such that the pullback is $\pi^{\ast}(c) = db$. Then, the transgression $\mathrm{tg}$ is defined as $\mathrm{tg}([a]) = [c]$. Since $H^1(G; A) \simeq  H^1(\Fn;A)$ and $H^2(\Fn; A) = 0$, the transgression map is an isomorphism. Dually we have the isomorphism $\mathrm{tg}^{\vee}$, called {\it Hopf isomorphism},
\begin{equation} \label{eq:-1.1.14}
	\mathrm{tg}^{\vee} : H_2(G; A) \rightarrow H_1(R; A)_G \simeq R \cap [\Fn, \Fn] / [R, \Fn].
\end{equation}
Then, an explicit formula that relates Massey products and Magnus coefficients is given as follows (cf. \cite[Lemma 1.5]{St}).

\begin{proposition} \label{prop:-1.1.1}
	Notations being as above, let $\alpha_1, \ldots, \alpha_m \in  H_1(G; A)$ and let $\mathcal{A} = (a_{ij})$ be a defining system for the Massey product $\langle \alpha_1, \ldots, \alpha_m \rangle$. For $f \in R \cap [\Fn, \Fn]$, we  set $\eta := (\mathrm{tg}^{\vee})^{-1}(f \mod [R,\Fn])$. Then, we have
	\begin{eqnarray*}
		&&\langle \alpha_1, \ldots, \alpha_m \rangle_{\mathcal{A}}(\eta) \\
		&=&\sum_{r=1}^m (-1)^{r+1} \sum_{c_1 + \cdots + c_r = m} \sum_{1 \leq i_1, \ldots, i_r \leq n} a_{1,1+c_1}(g_{i_1}) a_{1+c_1, 1+c_1 + c_2}(g_{i_2}) \cdots a_{m+1-c_j,m+1}(g_{i_r})\mu(I;f)
	\end{eqnarray*}
where $c_i$ $(1 \leq  i \leq  j)$ runs over positive integers with $c_1 + \cdots + c_j = m$ and $g_i := \pi(x_i)$ $(1 \leq  i \leq n)$ and $\mu(i_1 \cdots i_j;f)$ is the Magnus coefficient of $f$ with  respect to $I = (i_1 \cdots i_j)$. 

	\end{proposition}
	
Recalling the fact that $\mu(I;f)=0$ for $f \in \Gamma_k(\Fn)$ with respect to $I=(i_1 \ldots i_r)$ with $r \leq k-1$, we immediately get the following corollary.

\begin{corollary} \label{cor:-1.1.2}
	With the same notation as in Proposition \ref{prop:-1.1.1}, for $f \in R \cap \Gamma_k(\Fn)$, we have
\begin{equation} \label{eq:-1.1.15}
	\langle \alpha_1, \ldots, \alpha_k \rangle_{\mathcal{A}}(\eta) = (-1)^{k+1} \sum_{I = (i_1 \cdots i_k)} \alpha_1(g_{i_1}) \cdots \alpha_k(g_{i_k}) \mu(I; f)
\end{equation}
\end{corollary}

\begin{remark} \label{rem:-1.1.3}
Here, we follow the sign convention in \cite{Dwy}, \cite{Mo1} and \cite{W1}. This is different from \cite{May} and \cite{Kr}.
\end{remark}

\subsection{A basis of $H^2(\Fn/\Gamma_k(\Fn); \mathbb{F})$ in terms of Massey products}

We describe a basis of $\Hom_{\mathbb{Z}_{\ell}}(H_2(\Fn/\Gamma_k(\Fn); \mathbb{Z}_{\ell}), \mathbb{Z}_{\ell})$ and $H^2(\Fn/\Gamma_k(\Fn); \mathbb{F})$ via Massey products. Here $\mathbb{F}=\mathbb{Q}_{\ell}$ or $\mathbb{F}_{\ell}$. 

To begin with, let us recall the notion of {\it Lyndon basis} or {\it standard basis} of $\mathcal{L}_k\simeq \Gamma_k(F_n)/\Gamma_{k+1}(F_n)$ following \cite{CFL}.

Let $\mathcal{I}$ be the set of all multi-indices $I=(i_1 \cdots i_m)$ with $1 \leq i_1, \ldots, i_k \leq n$. Suppose that $\mathcal{I}$ is endowed with lexicographic order induced from that of $\{1, \ldots, n\}$. An {\it end} of a multi-index $I=(i_1 \cdots i_k) \in \mathcal{I}$  is proper subindex $I_m :=(i_m i_{m+1} \cdots i_k)$ for each $m \geq 2$. A multi-index $I \in \mathcal{I}$ is said to be {\it standard} if $I < I_m$ for any $m \geq 2$ with respect to lexicographic order of $\mathcal{I}$. Then, any standard multi-index $I$ factorise into two standard indices $I=I_1 I_2$ so that $I_2$ is the longest standard end. We denote it by $I=(I_1, I_2)$. By iterating this factorisation, we obtain the indices with bracketing such that each index in bracket is standard. A standard multi-index $I=(i_1\cdots i_k)$ is also called a {\it Lyndon word} of length $k$ over $\{1,\ldots, n\}$.\footnote{Lyndon words can be computable in terms of module {\tt LyndonWords} in Sage Mathematics Software \cite{sage}} Let $LW_k$ be the set of Lyndon words of length $k$ over $\{1,\ldots, n\}$.
\begin{example}\label{ex:-1.2.1}
(1) The multi-indices $(122)$ and $(1122)$ are Lyndon words, but $(121)$ and $(1312)$ are not. (2) The Lyndon words $(122)$ and   $(1112)$ factorise into $(122)= ((12)2)$ and $(1112)=(1(1(12)))$ respectively. 
\end{example}
Inductively, we define a map
\begin{equation} \label{eq:-1.2.1}
	e : LW_{k} \rightarrow \mathcal{L}_k
\end{equation}
which assign elements of $\mathcal{L}_k$ corresponding to  Lyndon words as follows: For $k=1$ set $e(i)=X_i$, and for $k > 1$, set $e(I) = [e(I_1), e(I_2)]$ corresponding to standard factorisation of a Lyndon word $I=(I_1, I_2)$. It turns out that the image of $e$ forms a basis of $\mathcal{L}_k$. We  denote it by $LB_k$ and call it  {\it Lyndon basis} or {\it standard commutator} of $\mathcal{L}_k$. 
\begin{example} \label{ex:-1.2.2}
	 For Lyndon words $(122)$ and $(1112)$ as above, the corresponding elements of Lyndon basis  are given by  $e((122))=[[X_1, X_2] X_2]$ and $e((1112))=[X_1[X_1[X_1, X_2]]]$.
\end{example}

Next, we consider a defining system of Massey product obtained from Magnus coefficients as follows:

Let $x_i^{\ast} : F_n \rightarrow \mathbb{Z}$ be a homomorphism given by $x_i^{\ast}(x_j) = \delta_{ij}$ for $1 \leq i,j \leq n$. Here $\delta_{ij}$ denotes the Kronecker delta, that is, $\delta_{ij} = 1$ if $i=j$ and $\delta_{ij}=0$ otherwise. By abuse of notation, we denote its extension to $\Fn \rightarrow \mathbb{Z}_{\ell}$ by the same $x_i^{\ast}$.

As in equation \eqref{eq:2.8.5}, we consider the Magnus coefficient with respect to a multi-index $I$ as a homomorphism
\begin{equation} \label{eq:-1.2.2}
	\mu(I; -) : \Fn \rightarrow \mathbb{Z}_{\ell}
\end{equation}
and also as a 1-cocycle in $Z^1(\Fn, \mathbb{Z}_{\ell})$. Then, for a multi-index $I=(i_1 \cdots i_{m})$ of length $m$ and for $1 \leq i < j \leq m+1$, we set
\begin{equation} \label{eq:-1.2.3}
	\mu_{ij}:= \mu((i_ii_{i+1} \cdots i_{j-1}); -) : \Fn/\Gamma_k(\Fn)  \rightarrow \mathbb{Z}_{\ell}
\end{equation}
the induced 1-cocycle on $k$-th nilpotent quotient $\Fn/\Gamma_k(\Fn)$. Then, we denote by $\mathcal{A}$ the array consisting of $-\mu_{ij}$ with $1 \leq i < j \leq m+1$. Then, we see that $\mathcal{A}=(-\mu_{ij})$ forms a defining system  for $\langle -x^{\ast}_{i_1}, \ldots, -x^{\ast}_{i_{m}} \rangle$. In fact $\mu_{i,i+1} = \mu(i_i; -) = x^{\ast}_{i_i}$ and we can check
\begin{equation} \label{eq:-1.2.4}
	d \mu_{ij} = \sum_{m=i+1}^{j-1} \mu_{ik} \smile \mu_{k j}
\end{equation}
for $j \neq i+1$ since we have
\begin{equation} \label{eq:-1.2.5}
\mu(I; g_1 g_2) = \sum_{(I_1 I_2) = I} \mu(I_1; g_1) \mu(I_2; g_2).	
\end{equation}
Therefore, the resulting Massey product
\begin{equation} \label{eq:-1.2.6}
	\langle -x^{\ast}_{i_1}, \ldots, -x^{\ast}_{i_{m}} \rangle_{\mathcal{A}} \in H^2(\Fn/\Gamma_k(\Fn); \mathbb{Z}_{\ell})
\end{equation}
is represented by the following 2-cocycle:
\begin{equation} \label{eq:-1.2.7}
	\sum_{1\leq l \leq m} \mu((i_1i_2 \cdots i_l); -) \smile \mu((i_{l+1} \cdots i_m); -).
\end{equation}

Then, we have the following pro-$\ell$ version of result in \cite{O3}.
\begin{theorem}	\label{thm:-1.2.3}
Let $k \geq 2$ be a fixed integer. Then, following statements hold:
\begin{enumerate}[label=$(\arabic{enumi})$]
	\item Any  Massey product $\langle -x_{i_1}^{\ast}, \ldots, -x_{i_m}^{\ast} \rangle$ with $m < k$ vanishes. In particular, any  Massey product $\langle -x_{i_1}^{\ast}, \ldots, -x_{i_k}^{\ast} \rangle \in H^2(\Fn/\Gamma_{k}(\Fn); \mathbb{Z}_{\ell})$ is uniquely defined for each standard index $I=(i_1 \cdots i_k)$.
	\item The Massey products $\langle -x_{i_1}^{\ast}, \ldots, -x_{i_k}^{\ast} \rangle$ with standard indices $I=(i_1 \cdots i_k)$ form a $\mathbb{Z}_{\ell}$-basis for the dual module of the second homology group $\Hom_{\mathbb{Z}_{\ell}}(H_2(\Fn/\Gamma_{k}(\Fn);\mathbb{Z}_{\ell}), \mathbb{Z}_{\ell})$ under the surjective homomorphism
	\begin{equation} \label{eq:-1.2.8}
		H^2(\Fn/\Gamma_k(\Fn); \mathbb{Z}_{\ell}) \rightarrow \Hom_{\mathbb{Z}_{\ell}}(H_2(\Fn/\Gamma_{k}(\Fn);\mathbb{Z}_{\ell}), \mathbb{Z}_{\ell}).
	\end{equation}
\end{enumerate}
\end{theorem}

\begin{proof}
(1) By Corollary \ref{eq:6.3.9}, we have $H_2(\Fn/\Gamma_k(\Fn); \mathbb{Z}_{\ell}) \simeq \Gamma_k(\Fn)/\Gamma_{k+1}(\Fn)$. We know that for such an element $g$, $\mu(I; g) = 0$ with any multi-index $I$ of length $<k$ by Remark \ref{rem:2.8.6} (2). Therefore, any Massey products with any multi-index $I$ of length $m<k$ vanish by Proposition \ref{prop:-1.1.1}. The uniqueness follows from Remark \ref{rem:-1.1.7} (2).\\
(2) In terms of the Lyndon basis, there is a homomorphism
\begin{equation} \label{eq:-1.2.9}
	\bigoplus_{I \in LW_k} \mu(I;-) : \Gamma_k(F_n) / \Gamma_{k+1}(F_n) \rightarrow \mathbb{Z}^{\oplus N_k}
\end{equation}
given by 
\begin{equation} \label{eq:-1.2.10}
	 \Gamma_k(F_n) / \Gamma_{k+1}(F_n) \ni g=\sum_{I \in LW_k} \mu(I; g) \cdot e(I) \longmapsto (\mu(I; g))_{I \in LW_k} \in \mathbb{Z}^{\oplus N_k}.
\end{equation}
It is known that $\bigoplus_{I \in LW_k}$ is surjective and so isomorphism (cf. \cite[Theorem 3.5 and 3.9]{CFL}). By abuse of notation, we denote by the same $\bigoplus_{I \in LW_k} \mu(I;-)$ the extension
\begin{equation} \label{eq:-1.2.11}
	\bigoplus_{I \in LW_k} \mu(I;-) : \Gamma_k(\Fn) / \Gamma_{k+1}(\Fn) \rightarrow \mathbb{Z}_{\ell}^{\oplus N_k},
\end{equation}
and we see that this also gives an isomorphism. In fact, this follows from  $\Gamma_k(F_n) / \Gamma_{k+1}(F_n) \widehat{\ }^{(\ell)} \simeq \Gamma_k(\Fn) / \Gamma_{k+1}(\Fn)$. 

Let  $e^{\ast}(I)$ denote the Kronecker dual to $e(I)$ for each standard indices $I=(i_1 \cdots i_k)$. Then, by applying  Corollary \ref{cor:-1.1.2} to $\langle - x^{\ast}_{i_1}, \ldots, -x^{\ast}_{i_k} \rangle_{\mathcal{A}}$, for $g=\sum_{J \in LW_k} \mu(J; g) \cdot e(J)$, we have
\begin{eqnarray*}
	\langle - x^{\ast}_{i_1}, \ldots, -x^{\ast}_{i_k} \rangle_{\mathcal{A}}(g) &=& (-1)^{k+1} \sum_{J=(j_1 \cdots j_k)} (x^{\ast}_{i_1}(x_{j_1}) \cdots x^{\ast}_{i_k}(x_{j_k}) \mu(J; g) \\
	&=& (-1)^{k+1} \mu((i_1 \cdots i_k; g)\\
	&=& (-1)^{k+1} \langle e^{\ast}(I), g \rangle.
\end{eqnarray*}
Here, $\langle -, - \rangle$ denotes the Kronecker paring. Therefore, as an element of $\Hom_{\mathbb{Z}_{\ell}}(H_2(\Fn/\Gamma_{k}(\Fn); \mathbb{Z}_{\ell}), \mathbb{Z}_{\ell})$, the 2-cocycle $\langle - x^{\ast}_{i_1}, \ldots, -x^{\ast}_{i_k} \rangle_{\mathcal{A}}$ is identified with the Kronecker dual $(-1)^{k+1} e^{\ast}(I)$. Thus, the 2-cocycles $\{ \langle -x^{\ast}_{i_1}, \ldots, -x^{\ast}_{i_k}\rangle | I \in LW_k\}$ forms a basis of  $\Hom_{\mathbb{Z}_{\ell}}(H_2(\Fn/\Gamma_{k}(\Fn); \mathbb{Z}_{\ell}), \mathbb{Z}_{\ell})$. This completes the proof.
\end{proof}

We complete this section by giving Massey products description of a basis for the second cohomology groups $H^{2}(\Fn/\Gamma_{k}(\Fn); \mathbb{F})$ with coefficients in $\mathbb{F}=\mathbb{Q}_{\ell}$ or $\mathbb{F}_{\ell}$, as a corollary of above discussions. 
\begin{corollary} \label{cor:7.2.6}
Let $\mathbb{F}$ denote $\mathbb{Q}_{\ell}$ or $\mathbb{F}_{\ell}$.
	Then, the following statement holds:
		The  Massey products $\langle -x_{i_1}^{\ast}, \ldots, -x_{i_k} \rangle$ with standard indices $I=(i_1 \cdots i_k)$ form a $\mathbb{F}$-basis for the second cohomology group  $H^2(\Fn/\Gamma_k(\Fn); \mathbb{F})$ with coefficients in $\mathbb{F}$.

\end{corollary}
\begin{proof}
It is enough to show that $\dim_{\mathbb{F}}(H^{\bullet}(\Fn/\Gamma_k(\Fn); \mathbb{F})) = \mathrm{rank}_{\mathbb{Z}_{\ell}}(H_{\bullet}(\Fn/\Gamma_k(\Fn); \mathbb{Z}_{\ell}))$. 
	When $\mathbb{F}=\mathbb{F}_{\ell}$, this follows from Pontryagin duality  $H^{\bullet}(\Fn/\Gamma_k(\Fn); \mathbb{F}_{\ell}) \simeq H_{\bullet}(\Fn/\Gamma_k(\Fn); \mathbb{F}_{\ell})$. For the case of $\mathbb{F} = \mathbb{Q}_{\ell}$, by the universal coefficient theorem, we have
	\begin{equation} \label{eq:7.2.25}
		H^{\bullet}(\Fn/\Gamma_k(\Fn); \mathbb{Q}_{\ell}) \simeq \Hom_{\mathbb{Z}_{\ell}}(H_{\bullet}(\Fn/\Gamma_k(\Fn); \mathbb{Z}_{\ell}), \mathbb{Q}_{\ell})
	\end{equation}
	since $\mathbb{Q}_{\ell}$ is injective $\mathbb{Z}_{\ell}$-module and so $\Ext^1_{\mathbb{Z_{\ell}}}(-, \mathbb{Q}_{\ell}) = 0$.
\end{proof}

\section{Properties of pro-$\ell$ Orr invariants.} \label{sec:8}

This section examines some (algebraic) properties of pro-$\ell$ Orr invariants $\theta_k^{(\ell)}$ and its image under Hurewicz homomorphism $\tau_k^{(\ell)}$ defined in Section \ref{sec:2} by applying the results in Section \ref{sec:6} and \ref{sec:7},

\subsection{Properties of $\theta_k^{(\ell)}(\sigma,\tau)$}

To begin with,  we show the following properties of pro-$\ell$ Orr invariants.
\begin{theorem}[properties of pro-$\ell$ Orr invariants] \label{thm:8.1.1}
	Let $\mathcal{T}$ be a set of basing and let $G_K[k]$ be the $k$-th Johnson subgroup of absolute Galois group $G_K$ of a number field $K$. Suppose that, for each $\tau \in \mathcal{T}$, we have a representation $G_K \rightarrow \Aut(\Fn)$ given as in section 2.7. Then, the pro-$\ell$ Orr invariant $\theta_k^{(\ell)}(\sigma, \tau) \in \pi_3(K^{(\ell)}_k)$ for based Galois element $(\sigma, \tau) \in G_K[k] \times \mathcal{T}$ satisfies the following properties:
	\begin{enumerate}[label=$(\arabic{enumi})$]
		\item The invariant $\theta_k^{(\ell)}(\sigma, \tau)$ is defined if and only if the length $\leq k$ $\ell$-adic Milnor invariants of $\sigma$ vanish.
		\item The map $\theta^{(\ell)}_k : G_K[k] \times \mathcal{T} \rightarrow \pi_3(K^{(\ell)}_k)$ is additive with respect to the product of $(E_{\sigma}, \tau)$, i.e., we have
		\begin{equation}
			\theta_{k}^{(\ell)}((\sigma_1, \tau_1) \circ (\sigma_2, \tau_2))=\theta_k^{(\ell)}(\sigma_1, \tau_1) + \theta^{(\ell)}_k(\sigma_2, \tau_2).
		\end{equation}
		for $(\sigma_1, \tau_1), (\sigma_2, \tau_2) \in G_K[k] \times \mathcal{T}$. Here, $\theta_{k}^{(\ell)}((\sigma_1, \tau_1) \circ (\sigma_2, \tau_2))$ denotes the pro-$\ell$ Orr invariant obtained from $(E_{\sigma_1}, \tau_1) \circ (E_{\sigma_2}, \tau_2)$.
		In particular, for fixed basing $\tau$, we have
		\begin{equation}
			\theta_k^{(\ell)}(\sigma_1 \circ \sigma_2, \tau) = \theta_k^{(\ell)}(\sigma_1, \tau) + \theta_k^{(\ell)}(\sigma_2, \tau).
		\end{equation}
		\item $\psi_{l,k}(\theta_l^{(\ell)}(\sigma, \tau)) = \theta_k^{(\ell)}(\sigma, \tau)$
	\item $\theta_k^{(\ell)}(\sigma, \tau) \in \Im(\psi_{k+1,k})$ if and only if all length $k+1$ $\ell$-adic Milnor invariants of $\sigma$ vanish.
	\item Let $\tau, \tau'$ be two basing. Then, $\theta_k(\sigma, \tau) \equiv \theta_k(\sigma, \tau')  \mod \Im(\psi_{k+1, k})$.
	\end{enumerate}
\end{theorem}

\begin{proof}
(1) Assume that the invariant $\theta^{(\ell)}(\sigma, \tau)$ is defined. Then, by assumption, we know that each $y_i(\sigma)$ lies in $\Gamma_k(\Fn)$, so its Magnus coefficients $\mu(I; y_i(\sigma))$ are all zero with multi-indices of length $\leq k$. Conversely, assume that all the Magnus coefficients $\mu(I; y_i(\sigma)) =0$ for any multi-index $I$ of length $\leq k$, then we see $y_i(\sigma) \in \Gamma_k(\Fn)$. Thus, the invariant $\theta_k^{(\ell)}(\sigma, \tau)$ is defined. \\
(2) This is Theorem 3.4.4 and Corollary 3.4.5. \\
(3) Assume that $\theta_l^{(\ell)}(\sigma, \tau)$ is defined. Then, degree $k$ to $l-1$ part of $\theta_k^{(\ell)}(\sigma, \tau)$ must be zero. This means $\theta_k^{(\ell)}(\sigma, \tau) = \psi_{l,k}(\theta_l^{(\ell)}(\sigma,\tau))$.\\
(4) Suppose that all Magnus coefficients $\mu(I; y_i(\sigma))$ for multi-indices of length $k+1$ vanish. Then, by (1) the invariant $\theta_{k+1}^{(\ell)}(\sigma,\tau)$ is defined. Therefore, applying (3), we have $\theta_k^{(\ell)}(\sigma, \tau) = \psi_{k+1,k}(\theta_{k+1}^{(\ell)}(\sigma, \tau))$. Conversely, if $\theta_{k}^{(\ell)}(\sigma, \tau) \in \Im(\psi_{k+1,k})$, then this means that $\theta_{k+1}^{(\ell)}(\sigma, \tau)$ is defined. By (1), we conclude that  all Magnus coefficients $\mu(I; y_i(\sigma))$ for multi-indices of length $k+1$ vanish. \\
(5) By (2), $\theta_k(\sigma, \tau) - \theta_k(\sigma, \tau') = \theta_k(h^{-1} \circ \sigma^{-1} \circ h \circ \sigma, \tau)$. Here, $h \in \Aut(\Fn)$ such that $\tau = \tau' \circ h$. Since the $\ell$-adic Milnor invariants of length $\leq k$ vanish for $\sigma$ and the first non-vanishing Milnor invariants are additive invariants, length $\leq k+1$ Milnor invariants of $(h^{-1} \circ \sigma^{-1} \circ h \circ \sigma, \tau)$ vanish. In fact, we obtain

\begin{eqnarray*}
	&&\mu(h^{-1} \circ \sigma^{-1} \circ h \circ \sigma; (i_1 \cdots i_{k+1}))\\
	& =& \mu(h^{-1} \circ \sigma^{-1} \circ h; (i_1 \cdots i_{k+1})) + \mu(\sigma; (i_1 \cdots i_{k+1})) \\
	& =& \mu(\sigma^{-1}; (i_1 \cdots i_{k+1})) + \mu(\sigma; (i_1 \cdots i_{k+1})) \\
	& =& 0.
\end{eqnarray*}
The second follows from the fact that $h$ which acts on generators of $\Fn$ by conjugation and the first non-vanishing Milnor invariants are invariant under conjugation of each generator.

Therefore, $\theta_{k+1}^{(\ell)}(h^{-1} \circ \sigma^{-1} \circ h \circ \sigma, \tau)$ is defined. So, by (3), $\psi_{k+1, k}(\theta_{k+1}^{(\ell)}(h^{-1} \circ \sigma^{-1} \circ h \circ \sigma, \tau)) = \theta_{k}^{(\ell)}(h^{-1} \circ \sigma^{-1} \circ h \circ \sigma, \tau) = \theta_k^{(\ell)}(\sigma, \tau) - \theta_k(\sigma, \tau')$.
\end{proof}

\begin{remark}
 Note that in \cite{O1} he proved some geometric properties of his invariant such as realizability (\cite[\S 2]{O1}) in addition to the above properties.
\end{remark}

\subsection{Properties of $\tau_k^{(\ell)}(\sigma,\tau)$}

Next, we study some property of $\tau_k^{(\ell)}(\sigma,\tau)$.

\begin{theorem}[Vanishing condition of $\tau_k^{(\ell)}(\sigma, \tau)$] \label{thm:8.2.1}
	With the same notation as in Theorem 8.1.1, let $\sigma$ be an element in $G_K[k]$. Then, $\tau_k^{(\ell)}(\sigma, \tau)=0$ if and only if $\sigma \in G_K[2k-1]$.
\end{theorem}

\begin{proof}
Since $\tau_k^{(\ell)}(\sigma, \tau)$ is the image of $\theta_k^{(\ell)}(\sigma, \tau)$ under the Hurewicz homomorphism, if $\tau_k^{(\ell)}(\sigma, \tau)$ is defined, then $\theta_k^{(\ell)}(\sigma, \tau)$ is also defined. Note that by Lemma 3.5.1, Corollary 5.3.2 and Theorem 6.2.5, we obtain the following exact sequence
\begin{equation}
	0 \rightarrow \mathbb{Z}_{\ell}^{\oplus n N_{2k-1} - N_{2k}} \rightarrow \pi_3(\widehat{K}_k^{(\ell)}) \rightarrow H_3(\Fn/\Gamma_k(\Fn)) \rightarrow 0.
\end{equation}
Thus, the vanishing of $\tau_k^{(\ell)}(\sigma, \tau)$ is equivalent to  $\theta_{k}^{(\ell)}(\sigma, \tau) \in \Im(\psi_{2k-1,k})$. Assume that $\theta_{k}^{(\ell)}(\sigma, \tau) \in \Im(\psi_{2k-1,k})$. Then, by Theorem 8.1.1 (4), we see that  all length $l$ $(l \leq 2k-1)$ $\ell$-adic Milnor invariants of $\sigma$ vanish. This means $\sigma \in G_K[2k-1]$.

Conversely, assume that $\sigma \in G_K[2k-1]$. Then, by Theorem 8.1.1 (3), $\theta_{2k-1}^{(\ell)}(\sigma, \tau)$ is defined. Now, by Theorem 8.1.1 (3), $\theta_{k}^{(\ell)}(\sigma, \tau) \in \Im(\psi_{2k-1, k})$. Thus, $\tau_k^{(\ell)}(\sigma, \tau)=0$.
\end{proof}

As corollary of Theorem \ref{thm:8.2.1} and Theorem 7.2.3, vanishing of $\tau_k^{(\ell)}(\sigma, \tau)$ can be described by Massey products as follows.
\begin{corollary}
	With notations as in Theorem 7.1.2 and Theorem 8.2.1, let $\sigma$ be an element of $G_K[k]$. Let $y_1(\sigma), \ldots, y_n(\sigma)$ be pro-$\ell$ words given as in Section 2.9. Then, $\tau_k^{(\ell)}(\sigma, \tau)=0$ if and only if, for any $m \leq 2k-1$,  the evaluations of Massey products $\langle -x_{i_1}^{\ast}, \ldots, -x_{i_m}^{\ast} \rangle \in H^2(\pi_1(E_g); \mathbb{Z}_{\ell})$ at $\eta_i(\sigma) := (\mathrm{tg}^{\vee})^{-1}([x_i, y_i(\sigma)] \mod [R, \Fn])$ $(1\leq i \leq n)$,
	\begin{equation}
		\langle -x_{i_1}^{\ast}, \ldots, -x_{i_m}^{\ast} \rangle (\eta_i(\sigma)) = 0
	\end{equation}
	for any index $I=(i_1 \cdots i_m)$ of length $m$.
		\end{corollary}

\section{Applications} \label{sec:9}

In this section, we consider pro-$\ell$ Orr invariants in the context of  Grothendieck's section conjecture \cite{Gr}. In particular, we study its relation with Ellenberg obstruction introduced by Ellenberg in \cite{E} and further studied by Wickelgren in \cite{W1}, \cite{W2} and \cite{W3}. In addition, we give examples of $\ell$-adic Milnor invariants which are equivalent to the vanishing of  $\tau_k^{(\ell)}(\sigma, \tau)$ for $X=\mathbb{P}^1_{\mathbb{Q}}\setminus \{0,1,\infty\}$.

\subsection{Review of Ellenberg's obstruction to $\pi_1$ sections}

To begin with, let us recall the notion of Ellenberg's obstruction introduced by Jordan Ellenberg in \cite{E}.

Let $X \rightarrow \Spec(K)$ be a geometrically connected curve over a number field $K$. Here, a curve over $K$ means a finite type, separated, reduced scheme of dimension 1 over $K$. Fix an embedding $K \hookrightarrow \mathbb{C}$. Let $\overline{K}$ be a fixed algebraic closure of $K$ in $\mathbb{C}$. We set  $X_{\overline{K}}=X \times_{K} \Spec(\overline{K})$ and take   a rational point or rational tangential base point $b$. Let $X(K)$ denote the set of $K$-rational points of $X$. Then, we consider non-abelian Kummer map
\begin{equation}
	\kappa=\kappa_{(X,b)} : X(K) \rightarrow H^1(G_K; \Piet(X_{\overline{K}}, b))
\end{equation}
given by, for $z \in X(K)$, $\sigma \in G_K$ and $\gamma_z \in \Path(b,z)$,
\begin{equation}
	\kappa(z)(\sigma) := [\gamma_z^{-1}\cdot (\sigma(\gamma_z))] \in \Piet(X_{\overline{K}}, b).
\end{equation}
We can easily show that $\kappa(z)$ is 1-cocycle and does not depend on the choice of the path $\gamma_z \in \Path(b,z)$, so is well-defined. It is known that $\kappa$ is an injection. The section conjecture states that for a smooth proper curve $X$ of genus at least 2  over a number field $K$, the non-abelian Kummer map is not only injection but also surjection, that is, a bijection. 

One of the difficult points to study this conjecture 
 comes from fact that  the $\Piet(X_{\overline{K}}, b)$ is huge and complicated group. To remedy this,  Ellenberg considered its approximation using  the tower of nilpotent quotients of $\Piet:=\Piet(X_{\overline{K}}, b)$:
\begin{center}
\begin{tikzcd}
	& \vdots \ar[d]\\
	& H^1(G_K; \Piet / \Gamma_4(\Piet)) \ar[d] \\
	& H^1(G_K; \Piet / \Gamma_3(\Piet)) \ar[d]\\
	X(K) \ar[r] \ar[ur] \ar[uur]& H^1(G_K; \Piet / \Gamma_2(\Piet))
\end{tikzcd}	
\end{center}
Note that with the choice of the base point $b$, $X$ embedded in to its generalized Jacobian $\mathrm{Jac}(X)$ (for generalized Jacobian, see \cite[Chapter V]{Ser1})
\begin{equation}
	X \rightarrow \mathrm{Jac}(X)
\end{equation}
by Abel-Jacobi map and its induced homomorphism on \'etale fundamental groups
\begin{equation}
	\Piet(X_{\overline{K}},b) \rightarrow \Piet(\mathrm{Jac}(X)_{\overline{K}}, b)=\Piet/\Gamma_2(\Piet).
\end{equation}
Since $X$ is embedded into $\mathrm{Jac}(X)$, we have $X(K) \subset \mathrm{Jac}(X)(K)$. Thus, one sees that $H^1(G_K; \Piet/\Gamma_2(\Piet))$ contains a homotopy section comes not from a $K$-rational point in $X(K)$  but from that in $\mathrm{Jac}(X)(K)$.

For each $k \geq 2$, the central extension of profinite groups
\begin{equation}
	1 \rightarrow \Gamma_k(\Piet)/\Gamma_{k+1}(\Piet) \rightarrow \Piet / \Gamma_{k+1}(\Piet) \rightarrow \Piet / \Gamma_k(\Piet) \rightarrow 1
\end{equation}
give rise to a continuous boundary homomorphism $\delta_k$ on cohomology groups 
\begin{equation}
	\delta_k : H^1(G_K; \Piet/\Gamma_{k}(\Piet)) \rightarrow H^2(G_k; \Gamma_k(\Piet) / \Gamma_{k+1}(\Piet)).
\end{equation}

This boundary homomorphism $\delta_k$ is an obstruction to a homotopy section $H^1(G_K; \Piet / \Gamma_k(\Piet))$ to be the image of a homotopy section of $H^1(G_K; \Piet / \Gamma_{k+1}(\Piet))$. That is, for $f \in H^1(G_K; \Piet/\Gamma_k(\Piet))$, if $\delta_k(f)=0$ then there is a lift $\tilde{f}$ of $f$  in $H^1(G_k; \Piet/\Gamma_{k+1}(\Piet))$.  Thus, starting from a homotopy section $f \in H^1(G_K; \Piet/\Gamma_2(\Piet))$,  computing $\delta_2, \ldots, \delta_k, \ldots$ recursively, one can know whether $f$ comes from $K$-rational point of $X(K)$ or not. If there is a positive integer $k$, the value of $\delta_k$ of a lift of $f$ does not vanish, then $f$ must not come from a homotopy section of $H^1(G_K, \Piet)$. As we see in Section 7.1, the Ellenberg obstruction class $\delta_k$ is computable from Massey products. We note that, Wickelgren studied Ellenberg obstruction from this viewpoints in her series of articles \cite{W1}, \cite{W2} and \cite{W3}.

\subsection{Relation between pro-$\ell$ Orr invariants and Ellenberg's obstruction}

Here, we describe the relation between Ellenberg's obstruction and pro-$\ell$ Orr invariants.

Here, we consider the spacial case that $X=\mathbb{P}^1_{K} \setminus \{\infty, a_1, \ldots, a_n \} \rightarrow \Spec(K)$ with $a_1,\ldots, a_n \in K$ as in Section 2.7. Note that we use the same notation as in Section 2.7. 

Then, in terms of $G_K$-action $\varphi$ on the maximal pro-$\ell$ quotient $\Piet = \Piet(X_{\bar{K}}, b_{\infty})^{(\ell)}$ we define the Johnson filtration of $G_K$ with respect to $\varphi$, as
\begin{equation} \label{eq:9.2.1}
	G_K = G_K[0] \supset G_K[1] \supset G_K[2] \supset \cdots \supset G_K[m] \supset \cdots 
\end{equation}

Note that, for $m \leq l$,  the inclusion homomorphism $\iota:G_K[l] \hookrightarrow  G_K[m]$ induces the  homomorphism 
\begin{equation} \label{eq:9.2.2}
	H^1(G_K[m]; \Piet / \Gamma_k(\Piet)) \rightarrow H^1(G_K[l]; \Piet / \Gamma_k(\Piet))
\end{equation}
via pullback along $\iota$. By considering approximation not only via nilpotent quotient by Ellenberg but also via Johnson filtration, we obtain the following double indexed nilpotent tower:
\begin{center}
\begin{tikzcd}
	& \vdots \ar[d] & \vdots \ar[d] & \vdots \ar[d] & \\
	& H^1(G_K; \Piet / \Gamma_4(\Piet)) \ar[d] \ar[r] & H^1(G_K[2]; \Piet / \Gamma_4(\Piet)) \ar[d] \ar[r] & H^1(G_K[3]; \Piet / \Gamma_4(\Piet)) \ar[r] \ar[d]& \cdots \\
	& H^1(G_K; \Piet / \Gamma_3(\Piet)) \ar[d] \ar[r] & H^1(G_K[2]; \Piet / \Gamma_3(\Piet) \ar[r] \ar[d] & H^1(G_K[3]; \Piet / \Gamma_3(\Piet)) \ar[r] \ar[d] & \cdots \\
	X(k) \ar[r] \ar[ur] \ar[uur]& H^1(G_K; \Piet / \Gamma_2(\Piet)) \ar[r] & H^1(G_K[2]; \Piet / \Gamma_2(\Piet)) \ar[r] & H^1(G_K[3]; \Piet / \Gamma_2(\Piet)) \ar[r] & \cdots
\end{tikzcd}	
\end{center}

Corresponding to the double indexed nilpotent tower, there is the double indexed boundary map, for $k\geq 1$ and $m \geq 1$,
\begin{equation} \label{eq:9.2.3}
	\delta_{k, m} : H^1(G_K[m]; \Piet / \Gamma_k(\Piet)) \rightarrow H^2(G_K[m];  \Gamma_k(\Piet)/ \Gamma_{k+1}(\Piet))
\end{equation}
obtained by restricting the boundary map $\delta_k : H^1(G_K; \Piet / \Gamma_k(\Piet)) \rightarrow H^2(G_K;  \Gamma_k(\Piet)/ \Gamma_{k+1}(\Piet))$ to the $m$-th Johnson subgroup $G_K[m]$. 

Then, by definition of Johnson filtration and Ellenberg obstruction, we see the following properties:

\begin{lemma} \label{lem:9.2.1}
Notations being as above, let $y_i : G_K \rightarrow \Piet$ be 1-cocycles corresponding to $G_K$-action on $\Piet$ such that $\sigma(x_i) = y_i(\sigma)^{-1} x_i^{\chi(\sigma)} y_i(\sigma)$. Let $[y_i]_k \in H^1(G_K; \Piet/\Gamma_k(\Piet))$ be the image of $[y_i] \in H^1(G_K; \Piet)$ under the reduction map $H^1(G_K; \Piet) \rightarrow H^1(G_K; \Piet/\Gamma_k(\Piet))$. Then, the following statements hold:
\begin{enumerate}[label=$(\arabic{enumi})$]
	\item For any $k \geq 2$ we have
	\begin{equation} \label{eq:9.2.4}
		\delta_k([y_i]_k) = 0 , \quad (1 \leq i \leq n),
	\end{equation}
	\item Similarly, for any $m \geq 1$ and $k \geq 2$, we have
	\begin{equation} \label{eq:9.2.5}
		\delta_{k,m}([y_i|_{G_K[m]}]) = 0 , \quad (1 \leq i \leq n),
	\end{equation}
	\item For $m \geq 2$ and $m \geq k$, we have
	\begin{equation} \label{eq:9.2.6}
		[y_i|_{G_K[m]}] = 0 \in H^1(G_K[m];  \Piet / \Gamma_k(\Piet))
	\end{equation}
\end{enumerate}
\end{lemma}

\begin{proof}
(1) By definition of $[y_i]_k$, this is the image of $[y_i]_{k+1}$ under the reduction $H^1(G_K; \Piet/\Gamma_{k+1}(\Piet)) \rightarrow H^1(G_K; \Piet/\Gamma_k(\Piet))$. Thus, by definition of $\delta_k$, its evaluation $\delta_k([y_i]_k)$ must vanish.\\
(2) This follows from (1) since $G_K[m]$ is subgroup of $G_K$.\\
(3) By definition of Johnson filtration, the restriction $y_i|_{G_K[m]}$ is 1-cocycle   sending $\sigma \in G_K[m]$ to $y_i|_{G_K[m]}(\sigma) \in \Gamma_m(\Piet)$. Since $\Gamma_m(\Piet)/\Gamma_k(\Piet)=0$ for $m \geq k$, $y_i|_{G_K[m]}(\sigma) = 0 \in \Piet/\Gamma_k(\Piet)$ for $m \geq k$. Thus, the assertion follows.
\end{proof}

To relate pro-$\ell$ Orr invariants and Ellenberg obstruction, we give 1-cocycle presentation of them. For this, we prepare the following notation. Let $\mu(I; -) : \Fn \rightarrow \mathbb{Z}_{\ell}$ be the pro-$\ell$ Magnus coefficient function with respect to the multi-index $I$. Then, for a 1-cocycle $f : G_K \rightarrow \Fn$, we set $\mu(I; f): G_K \rightarrow \mathbb{Z}_{\ell}$ which sends any $\sigma \in G_K$ to the pro-$\ell$ Magnus coefficient $\mu(I; f(\sigma)) \in \mathbb{Z}_{\ell}$ of $f(\sigma) \in \Fn$ with respect to $I$. Then, we have the following proposition stating the 1-cocycle version of the relation between the first non-vanishing $\ell$-adic Milnor invariants and the Massey products.

\begin{proposition} \label{prop:9.2.9}
With notations as in the above, the following statement holds: For $m \geq 1$ and $k=m+1$, the restriction of  1-cocycle $y_i$ to $G_K[m]$
\begin{equation} \label{eq:9.2.7}
	[y_i|_{G_K[m]}] \in H^1(G_K[m]; \Piet / \Gamma_{m+1}(\Piet))
\end{equation}
actually lies on $H^1(G_K[m], \Gamma_{m}(\Piet)/\Gamma_{m+1}(\Piet))$ and is explicitly written as
\begin{equation} \label{eq:9.2.8}
	[y_i|_{G_K[m]}] = \bigoplus_{I \in LW_m} \mu(I;y_i|_{G_K[m]}) \cdot e(I) \in H^1(G_K[m]; \Gamma_{m}(\Piet) / \Gamma_{m+1}(\Piet)).
\end{equation}
In particular, the Kronecker pairing of the Massey product $\langle -x_{i_1}^{\ast}, \ldots, -x_{i_m}^{\ast} \rangle$ in $H^2(\Fn/\Gamma_m(\Fn); \mathbb{Z}_{\ell})$ with $[y_i|_{G_K[m]}]$ gives the 1-cocycle
\begin{equation} \label{eq:9.2.9}
	\langle -x_{i_1}^{\ast}, \ldots, -x_{i_m}^{\ast} \rangle (y_i|_{G_K[m]}) = \mu((i_1 \cdots i_m); y_i|_{G_K[m]}) \in H^1(G_K; \mathbb{Z}_{\ell})
\end{equation}
\end{proposition}

\begin{proof}
	By definition of Johnson subgroup $G_K[m]$, we know that $y_i(\sigma) \in \Gamma_m(\Fn)$ for $\sigma \in G_K[m]$. Thus, $[y_i|_{G_K[m]}] \in H^1(G_K; \Gamma_m(\Piet)/\Gamma_{m+1}(\Piet)) \subset H^1(G_K; \Piet/\Gamma_{m+1}(\Piet))$. As in equation \eqref{eq:-1.2.10}, any element $g$ in $\Gamma_m(\Fn)/\Gamma_{m+1}(\Fn)$, can be written as linear combination $g=\sum_{I \in LW_m} \mu(I; g)\cdot e(I)$ in terms of Magnus coefficients and standard basis $e(I)$. Therefore, from our definition of $\mu(I; y_i)$, we conclude that $[y_i|_{G_K[m]}]$ is represented by
	\begin{equation} \label{eq:9.2.10}
		y_i|_{G_K[m]} = \sum_{I \in LW_m} \mu(I; y_i|_{G_K[m]})\cdot e(I).
	\end{equation}
	The last statement follows from Theorem \ref{thm:-1.2.3} since $\langle -x_{i_1}^{\ast},\ldots, -x_{i_m}^{\ast} \rangle$ is the Kronecker dual of $e((i_1 \cdots i_m))$ for standard multi-index $I=(i_1\cdots i_m)$.
\end{proof}

Further, we describe the pro-$\ell$ Orr invariants as 1-cocycle as follows.

\begin{proposition} \label{prop:9.2.3}
	Let $\tau$ be a basing and fix it. Let $\varphi : G_K \rightarrow \Aut(\Piet)$ be a representation as above with $\varphi(\sigma)(x_i) = y_i(\sigma)^{-1} x_i^{\chi_{\ell}(\sigma)} y_i(\sigma)$ $(1 \leq i \leq n)$. We define the map
	\begin{equation} \label{eq:9.2.11}
		\theta_m^{(\ell)}: G_K \rightarrow K_m^{(\ell)}
	\end{equation}
	by $\sigma \mapsto \theta_m^{(\ell)}(\sigma, \tau)$ for $\sigma \in G_K$. 
	Then, the following statements about pro-$\ell$ Orr invariants hold:
	\begin{enumerate}[label=$(\arabic{enumi})$]
		\item The map $\theta_m^{(\ell)}$ is defined for $G_K[l]$ with $l \geq m$.
		\item For $l\geq m$, the map $\theta_m^{(\ell)}$ is 1-cocycle and gives a cohomology class
		\begin{equation} \label{eq:9.2.12}
			[\theta_m^{(\ell)}] \in H^1(G_K[l]; \pi_3(K_m^{(\ell)}))
		\end{equation}
		\item For $l\geq 2m$, the map  $\theta_m^{(\ell)}$ is zero map.
	\end{enumerate}
	In particular, for $l \geq m$, the cohomology class $[\theta_m^{(\ell)}]$ lies in
	\begin{equation} \label{eq:9.2.13}
		[\theta_m^{(\ell)}] \in \bigoplus_{k=l}^{2m-1} H^1(G_K[l]; \mathbb{Z}_{\ell}^{\oplus n N_k - N_{k-1}})
	\end{equation}
\end{proposition}

\begin{proof}
These are consequence of basic (algebraic) properties of pro-$\ell$ Orr invariants in Theorem \ref{thm:8.1.1} and Theorem \ref{thm:7.2.21}. In particular, note that $G_K[l]$ acts on $\pi_3(K_m^{(\ell)})$ trivially when $l \geq m$.
\end{proof}

Similarly, for $\tau_k^{(\ell)}(\sigma, \tau) \in H_3(\Fn/\Gamma_k(\Fn); \mathbb{Z}_{\ell})$, there is 1-cocycle presentation as follows.

\begin{proposition} \label{prop:9.2.14}
	With the same notation as in Proposition \ref{prop:9.2.3}, we define a map
	\begin{equation} \label{eq:9.2.14}
		\tau_m^{(\ell)} : G_K \rightarrow H_3(\Fn/\Gamma_m(\Fn); \mathbb{Z}_{\ell})
	\end{equation}
	by $\sigma \mapsto \tau_m^{(\ell)}(\sigma, \tau)$. Then, the following statements about $\tau_m^{(\ell)}$ hold:
	\begin{enumerate}[label=$(\arabic{enumi})$] 
		\item The map $\tau_k^{(\ell)}$ is defined for $G_K[l]$ with $l \geq m$.
		\item For $l \geq m$, the map $\tau_m^{(\ell)}$ is 1-cocycle and gives a cohomology class
		\begin{equation} \label{eq:9.2.15}
			[\tau_m^{(\ell)}] \in H^1(G_k[l]; H_3(\Fn/\Gamma_m(\Fn); \mathbb{Z}_{\ell})).
		\end{equation}
		\item For $l \geq 2m-1$, the map $\tau_m^{(\ell)}$ is zero map.
		\item For $m \leq l \leq 2m-2$, the map $\tau_m^{(\ell)}$ vanish if and only if the 1-cocyles
		\begin{equation} \label{eq:9.2.16}
			\langle -x_{i_1}^{\ast},\ldots, -x_{i_r}^{\ast}\rangle (y_i|_{G_K[m]}) = \mu((i_1\cdots i_r); y_i|_{G_K[l]})  =0
		\end{equation}
		for any standard index $I=(i_1 \cdots i_r)$ with $m \leq r \leq 2m-2$ and $1 \leq i \leq n$.
	\end{enumerate}
\end{proposition}

\begin{proof}
(1), (2) and (3) follows from Proposition \ref{prop:9.2.3}, since $\tau_m^{(\ell)}$ is given by composition of Hurewicz homomorphism and $\theta_m^{(\ell)}$.\\
(4) This is a consequence of proof of Theorem \ref{thm:8.2.1} and Corollary \ref{cor:-1.1.2}.	
\end{proof}

Now, we relate Ellenberg obstruction and pro-$\ell$ Orr invariants as follows:

\begin{theorem} \label{thm:9.2.5}
For a representation $\varphi : G_K \rightarrow \Aut(\Piet)$ with fixed basing $\tau$. Let $y_1, \ldots, y_n : G_K \rightarrow \Piet$ be the corresponding 1-cocycles. Let $G_K[m]$ be the Johnson subgroup defined by $\varphi$. Let 1-cocycle $\theta_m^{(\ell)}$ obtained from pro-$\ell$ Orr invariants associated with $\varphi$ representing a class in $H^1(G_K[l]; \Piet/\Gamma_m(\Piet))$ for $l \geq m$. Then, followings are equivalent:
\begin{enumerate}[label=$(\arabic{enumi})$]
	\item  The 1-cocycle $\theta_m^{(\ell)}$ vanishes.
	\item The successive computation of double indexed Ellenberg obstructions 
	\begin{equation} \label{eq:9.2.16}
		\delta_{m,l}, \delta_{m+1,l},\ldots, \delta_{l,l}, \delta_{l+1,l+1}, \ldots, \delta_{2m-1,2m-1}
	\end{equation}
	of $[y_1]_m,\ldots, [y_n]_m$ vanish.
\end{enumerate}
\end{theorem}
\begin{proof}
	We assume (1). Then, by Theorem \ref{thm:8.1.1}, we see that $[y_i]_m \in H^1(G_K[l]; \Piet/\Gamma_m(\Piet))$ has a lift to $H^1(G_K[2m]; \Piet/\Gamma_m(\Piet))$, and therefore to $H^1(G_K[2m]; \Piet/\Gamma_{2m}(\Piet))$. This leads to (2). Conversely, assume (2). Then, the reverse direction of the above arguments gives (1). Note that $\delta_{m,l},\ldots, \delta_{l,l}$ automatically vanish by Lemma \ref{lem:9.2.1} (3) and, in the next step,  the cocycles lift to an elements in $H^1(G_K[l]; \Piet/\Gamma_{l+1}(\Piet))$ which are not automatically 0 but whose pullbacks to $H^1(G_K[l+1]; \Piet/\Gamma_{l+1}(\Piet))$ become 0. This completes the proof.
\end{proof}

\begin{remark}
	One of the possible advantages of  Theorem \ref{thm:9.2.5} is as follows. Suppose that 1-cocycles which would correspond to a representation $\varphi': G_K \rightarrow \Aut(\Piet)$ with the condition to define $\theta_m^{(\ell)}$ for some $m$ are given. Then, one may know whether two Johnson subgroups associated with $\varphi$ and $\varphi'$ coincide in some range by computing (1) or (2).
\end{remark}

\subsection{Explicit vanishing condition of   $\tau_k^{(\ell)}(\sigma, \tau)$ for $X=\mathbb{P}_{\mathbb{Q}}^1 \setminus \{0,1, \infty\}$ in lower degrees}

In general, it is difficult to compute pro-$\ell$ Orr invariants $\theta_k^{(\ell)}(\sigma, \tau)$ and its image under Hurewicz homomorphism $\tau_k^{(\ell)}(\sigma, \tau)$ from our definition. Hence, this section gives vanishing condition of $\tau_k^{(\ell)}(\sigma, \tau)$ in terms of $\ell$-adic Milnor invariants for the case of $X=\mathbb{P}_{\mathbb{Q}}^1 \setminus \{0,1, \infty\}$ using  Ihara-Kaneko-Yukinari's computation as in \cite{I3} and \cite{IKY}. Moreover, we give some diagrammatic computation applying methods given in \cite[Section 6]{HaMa}. Throughout this section, we use the same notation as in Section 2.7.

Let $k$ be an integer with $\geq 1$. Let $G_{\mathbb{Q}}[k]$ be the $k$-th Johnson subgroup of $G_{\mathbb{Q}}$. We set $H_{\mathbb{Q}_{\ell}} := H_1(\Fn;\mathbb{Q}_{\ell})$ and, by abuse of notation, $\mathcal{L}_k := \mathcal{L}_k \otimes_{\mathbb{Z}_{\ell}} \mathbb{Q}_{\ell}$. Then, we have the following homomorphism
\begin{equation}
	\mu_k : G_{\mathbb{Q}}[k] \rightarrow H_{\mathbb{Q}_{\ell}} \otimes_{\mathbb{Q}_{\ell}} \mathcal{L}_k
\end{equation}
which sends each $\sigma \in G_{\mathbb{Q}}[k]$ to $[x_1] \otimes [y_1(\sigma)] + [x_2] \otimes [y_2(\sigma)]$.  Then, we know that $\Im(\mu_k) \in \Ker([-,-] : H_{\mathbb{Q}_{\ell}} \otimes_{\mathbb{Q}_{\ell}} \mathcal{L}_k \rightarrow \mathcal{L}_{k+1})=:D_k(H_{\mathbb{Q}_{\ell}})$ by \eqref{eq:2.8.-1}. According to \cite{HaMa} and \cite{O1}, dimension of $D_k(H_{\mathbb{Q}_{\ell}})$ is equal to $nN_k - N_{k+1}$. Moreover, the space $D_k(H_{\mathbb{Q}_{\ell}})$ is isomorphic to the $\mathbb{Q}_{\ell}$-vector space $\mathcal{C}_k^t$ of  tree  Jacobi diagrams of degree $k$ labelled by $\{1,2\}$ subject to AS and IHX relations: 
\begin{equation*}
\begin{tikzpicture}
\begin{scope} 
\begin{scope}[densely dashed]
\draw (0,0)--(0,-0.45);
\draw (0,0)--(0.424,0.424);
\draw (0,0)--(-0.424,0.424);
\draw (-1,0) node{$=$};
\draw (-0.6,0) node{$-$};
\draw (-1,-0.9) node{AS};
\node at (0.5,-0.7) {,};
\end{scope}
\begin{scope}[xshift=-1.7cm]
\draw[densely dashed] (0,0)--(0,-0.45);
\draw[densely dashed] (0,0) .. controls (0.3,0.1) and (0.3,0.25)..(-0.424,0.424);
\draw[densely dashed] (0,0) .. controls (-0.3,0.1) and (-0.3,0.25)..(0.424,0.424);
\end{scope}
\end{scope}
\begin{scope}[xshift=-0.5cm]
\begin{scope}[xshift=2cm,densely dashed]
\draw (-0.3, 0.4) --(0.3,0.4);
\draw (0,0.4)--(0,-0.4);
\draw (-0.3,-0.4) --(0.3,-0.4);
\end{scope}
\begin{scope}[xshift=3.5cm, densely dashed]
\draw (-0.3,0.4)--(-0.3,-0.4);
\draw (0.3,0.4)--(0.3,-0.4);
\draw (-0.3,0)--(0.3,0);
\draw (-0.8, 0) node{$-$};
\end{scope}
\begin{scope}[xshift=5cm, densely dashed]
\draw (-0.3,0.4)--(0.3,-0.4);
\draw (0.3,0.4)--(-0.3,-0.4);
\draw (-0.19,-0.26)--(0.19,-0.26);
\draw (-0.8, 0) node{$+$};
\draw (0.8,0) node{$=0$};
\draw (-1.5, -0.9) node{IHX};
\end{scope}
\end{scope}
\end{tikzpicture}
\end{equation*}
Here, by a Jacobi diagram $D$, we mean a uni-trivalent graph with a cyclic order at each trivalent vertex. The degree of $D$ is the half of the number of its vertices. Customary, Jacobi diagrams are depicted by dashed lines in picture. In plane diagram, we always assume that, for any Jacobi diagram, the cyclic order at each trivalent vertex is given by opposite clockwise direction. Let us consider a linear map $\phi : \mathcal{C}_k^t \rightarrow D_k(H_{\mathbb{Q}_{\ell}})$  given by, for each degree $k$ tree Jacobi diagram $D$, 
\begin{equation}\label{eq:268}
	\phi(D) = \sum_{v} X_{l(v)} \otimes L_v(D) \in H_{\mathbb{Q}_{\ell}} \otimes \mathcal{L}_k
\end{equation}
where $v$ runs over all univalent vertices of $D$, $l(v) \in \{1,2\}$ is the label at $v$, and $L_v(D) \in \mathcal{L}_k$ is defined by the following manner: For any univalent vertex $v' \neq v$, label the edge incident to $v'$ of $D$ by $X_{v'}$. Next, we assign the label $[a,b]$ to any edge meeting $a$-labelled edge and $b$-labelled edge at a trivalent vertex following the cyclic orientation. Finally, one obtains the label associated the the edge incident to $v$. This is the desired element $L_v(D)$. For example, the element $L_v(D)$ for the following degree 3 tree Jacobi diagram is given by $[[X_1, X_2], X_2]$.
\begin{equation*}
 	\begin{tikzpicture}
 	\node at (-0.8,0.5) {$D=$};
 	\node at (2.8,0.5) {$\overset{L_v}{\longmapsto} [[X_1, X_2], X_2]$};
	\draw[densely dashed] (0,0) -- (0,1);
	\draw[densely dashed] (0,0.5) -- (1,0.5);
	\draw[densely dashed] (1,0) -- (1,1);
	\node[label=above:$v$] at (0,1) {};
	\node[label=below:$2$] at (0,0) {};
	\node[label=above:$1$] at (1,1) {};
	\node[label=below:$2$] at (1,0) {};
\end{tikzpicture}
\end{equation*}
Then, it turns out that the map $\phi$ actually gives an isomorphism of vector spaces $\phi : \mathcal{C}_k^t \overset{\sim}{\rightarrow} D_k(H_{\mathbb{Q}_{\ell}})$. For more details, see \cite[Section 6]{HaMa}.

Next, we recall Ihara's result in our situation. For this, we fix some notation. Let us fix a system $(\zeta_{\ell^k})$ of prmitive $\ell^k$-th roots of unity $\zeta_{\ell^k}$ in $\overline{\mathbb{Q}}$ such that $\zeta_{\ell^{k+1}}^{\ell} = \zeta_{\ell^k}$ $(k \geq 1)$. For each positive integer $m$, we set
\begin{equation}
	\epsilon_k^{(m)} := \prod_{\substack{1 \leq a \leq \ell^k \\ (a, \ell)=1}} (\zeta_{\ell^k}^a - 1)^{\langle a^m - 1 \rangle}
\end{equation}
where $\langle a^{m - 1}\rangle$ is the representative of $a^{m-1} \mod \ell^k$ lying on the interval $[0, \ell^{k})$. Then, we define the homomorphism
\begin{equation}
	\chi_m : G_{\mathbb{Q}}[1]  \rightarrow \mathbb{Z}_{\ell}^{\times}
\end{equation}
such that
\begin{equation}
	\frac{((\epsilon_k^{(m)})^{\ell^k})^{\sigma}}{(\epsilon_k^{(m)})^{\ell^k}} = \zeta_{\ell^k}^{\chi_m(\sigma)}.
\end{equation}
Under this notation, we have the following.

\begin{proposition}[{\cite[(4.5)]{I3}}]
Let $m \geq 0$ be an integer. Let $(12^{m}1)$ denote the multi-index $(122\cdots 21)$ with $m$-times iteration of $2$. Then, following statements hold:
\begin{enumerate}[label=$(\arabic{enumi})$]
	\item Let $k \geq 0$ be an integer. For any $\sigma \in G_{\mathbb{Q}}[2k+2]$, we have
	\begin{equation}
		\mu(\sigma; (12^{2k+1}1)) = 0.
	\end{equation} 	
	\item Let $k \geq 1$ be an integer. For any $\sigma \in G_{\mathbb{Q}}[2k+1]$, we have
	\begin{equation}
		\mu(\sigma; (12^{2k}1)) = \frac{(1-\ell^{2k})^{-1}}{(2k)!} \chi_{2k+1}(\sigma).
	\end{equation} 	
\end{enumerate}
\end{proposition}

\begin{proof}
	Here, we give a simple diagrammatic proof of (1). Noting the isomorphism $\phi$ in \eqref{eq:268}, we see that the tree Jacobi diagram corresponding $\mu(\sigma; (12^{2k+1}1))$ is given by the following diagram.
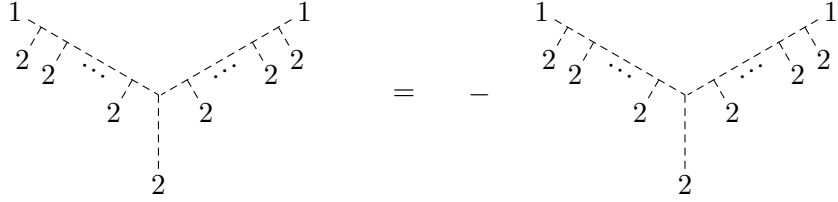
\begin{center}
\begin{tikzpicture}
\begin{scope}{
\begin{scope}[rotate=240]{
	\draw[densely dashed] (0,0) -- (0,2);
	\draw[densely dashed] (0,0.2) -- (0.3, 0.2);
	\draw[densely dashed] (0,0.6) -- (0.3, 0.6);
	\draw[densely dashed] (0,1.6) -- (0.3, 1.6);
	\node[rotate=240] at (0.15,1.2) {$\vdots$};
	\node at (0.5,0.2) {$2$};
	\node at (0.5,0.6) {$2$};
	\node at (0.5,1.6) {$2$};
	\node at (0,-0.2) {$1$};
	}
	\end{scope}
\begin{scope}[xshift=3.5cm, yshift=-0cm, rotate=-60,yscale=-1,xscale=1]{
	\draw[densely dashed] (0,0) -- (0,2);
	\draw[densely dashed] (0,0.2) -- (0.3, 0.2);
	\draw[densely dashed] (0,0.6) -- (0.3, 0.6);
	\draw[densely dashed] (0,1.6) -- (0.3, 1.6);
	\node[rotate=-60] at (0.15,1) {$\vdots$};
	\node at (0.5,0.2) {$2$};
	\node at (0.5,0.6) {$2$};
	\node at (0.5,1.6) {$2$};
	\node at (0,-0.2) {$1$};
	}
\end{scope}
\draw[densely dashed] ({sqrt(3)},-1) -- ({sqrt(3)},-2);
\node at ({sqrt(3)}, -2.2) {$2$};
}
\end{scope}

\node at (5,-1) {$=$};
\node at (6,-1) {$-$};

\begin{scope}[xshift=7cm]{
\begin{scope}[rotate=240]{
	\draw[densely dashed] (0,0) -- (0,2);
	\draw[densely dashed] (0,0.2) -- (0.3, 0.2);
	\draw[densely dashed] (0,0.6) -- (0.3, 0.6);
	\draw[densely dashed] (0,1.6) -- (0.3, 1.6);
	\node[rotate=240] at (0.15,1.2) {$\vdots$};
	\node at (0.5,0.2) {$2$};
	\node at (0.5,0.6) {$2$};
	\node at (0.5,1.6) {$2$};
	\node at (0,-0.2) {$1$};
	}
	\end{scope}
\begin{scope}[xshift=3.5cm, yshift=-0cm, rotate=-60,yscale=-1,xscale=1]{
	\draw[densely dashed] (0,0) -- (0,2);
	\draw[densely dashed] (0,0.2) -- (0.3, 0.2);
	\draw[densely dashed] (0,0.6) -- (0.3, 0.6);
	\draw[densely dashed] (0,1.6) -- (0.3, 1.6);
	\node[rotate=-60] at (0.15,1) {$\vdots$};
	\node at (0.5,0.2) {$2$};
	\node at (0.5,0.6) {$2$};
	\node at (0.5,1.6) {$2$};
	\node at (0,-0.2) {$1$};
	}
\end{scope}
\draw[densely dashed] ({sqrt(3)},-1) -- ({sqrt(3)},-2);
\node at ({sqrt(3)}, -2.2) {$2$};
}
\end{scope}
\end{tikzpicture}	
\end{center}
Here, there are $k$-times iterated vertices labelled by 2 on each side of the diagram, i.e., there are $2k+1$ vertices labelled by 2 on it. For convenience, let $D$ denote the above diagram. As we see in the above figure, the diagram $D$ is equal to $-D$ by AS relation. Therefore, over the field of characteristic 0, we have $D=0$. The assertion has been proved.
\end{proof}

Next, we give vanishing conditions of $\tau_k^{(\ell)}(\sigma, \tau)$ via $\ell$-adic Milnor invariants.
\begin{theorem}
With notations in above, the following statements hold:
\begin{enumerate}[label=$(\arabic{enumi})$]
\item For $\sigma \in G_{\mathbb{Q}}[2]$, the invariant $\tau_2^{(\ell)}(\sigma, \tau)$ always 0.
\item For $\sigma \in G_{\mathbb{Q}}[3]$, the invariant $\tau_3^{(\ell)}(\sigma, \tau)=0$ if and only if
\begin{equation} \label{eq:9.3.1}
	\mu(\sigma; (1221))=\frac{\chi_3(\sigma)}{2(1-\ell^2)} =0.
\end{equation}

\item For $\sigma \in G_{\mathbb{Q}}[4]$, the invariant $\tau_4^{(\ell)}(\sigma, \tau)=0$ if and only if
\begin{equation} \label{eq:9.3.2}
	\mu(\sigma; (111221))=0, \quad \mu(\sigma; (121221)=0, \quad \mu(\sigma;(122221))=\frac{\chi_5(\sigma)}{24(1-\ell^4)}=0.
\end{equation}
\end{enumerate}
\end{theorem}

\begin{proof}
 	(1) This follows from the Table 3 in A.3. Indeed, we have $H_3(\Fn/\Gamma_2(\Fn);\mathbb{Z}_{\ell})=0$.\\
 	(2) By Theorem 8.2.1, $\tau_3^{(\ell)}(\sigma, \tau)=0$ if and only if $\mu(\sigma; I)=0$ for any multi-index $I$ with lengths $\leq 4$. By Table 2 in A.2, we know that $\dim(\mathcal{C}_3^t)=1$ and $\dim(\mathcal{C}_4^t)=0$. By isomorphism $\phi$ in (269), the condition $\tau_3^{(\ell)}(\sigma, \tau)=0$ is equivalent to the vanishing of $\ell$-adic Milnor invariants of length 4 corresponding to the tree labeled Jacobi diagram below:
 	
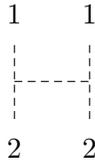
\begin{center}
 	\begin{tikzpicture}
	\draw[densely dashed] (0,0) -- (0,1);
	\draw[densely dashed] (0,0.5) -- (1,0.5);
	\draw[densely dashed] (1,0) -- (1,1);
	\node[label=above:$1$] at (0,1) {};
	\node[label=below:$2$] at (0,0) {};
	\node[label=above:$1$] at (1,1) {};
	\node[label=below:$2$] at (1,0) {};
\end{tikzpicture}
\end{center}
In particular, we can see that the $\ell$-adic Milnor invariants of length 4 are spanned by $\mu(\sigma; (1221))$. Therefore, we conclude that $\mu(\sigma; (1221))=0$ if and only if $\tau_3(\sigma, \tau)=0$. The equality \eqref{eq:9.3.1} follows from (273).\\
(3) Similarly, by Table 2 in A.2, we know that $\dim(\mathcal{C}_4^t)=0$, $\dim(\mathcal{C}_5^t)=3$, and $\dim(\mathcal{C}_6^t)=0$. By isomorphism $\phi$ in (269), the condition $\tau_4^{(\ell)}(\sigma, \tau)=0$ is equivalent to the vanishing of $\ell$-adic Milnor invariants of length 6 correspond to the following diagrams: 
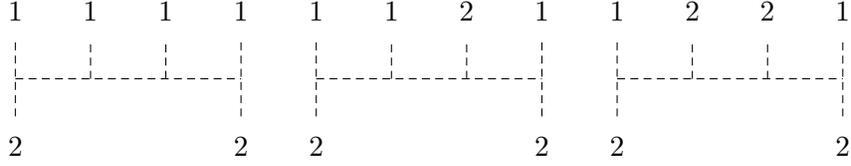
\begin{center}
	\begin{tikzpicture}
	\begin{scope}
		\draw[densely dashed] (0,0) -- (0,1);
		\draw[densely dashed] (0,0.5) -- (3,0.5);
		\foreach \x in {1,2} 
		{
			\draw[densely dashed] (\x,0.5) -- (\x,1);
		}
		\draw[densely dashed] (3,0) -- (3,1);
		\node[label=above:$1$] at (0,1) {};
		\node[label=below:$2$] at (0,0) {};
		\node[label=below:$2$] at (3,0) {};
		\node[label=above:$1$] at (3,1) {};
		\node[label=above:$1$] at (1,1) {};
		\node[label=above:$1$] at (2,1) {};
	\end{scope}
	
	\begin{scope}[xshift=4cm]
		\draw[densely dashed] (0,0) -- (0,1);
		\draw[densely dashed] (0,0.5) -- (3,0.5);
		\foreach \x in {1,2} 
		{
			\draw[densely dashed] (\x,0.5) -- (\x,1);
		}
		\draw[densely dashed] (3,0) -- (3,1);
		\node[label=above:$1$] at (0,1) {};
		\node[label=below:$2$] at (0,0) {};
		\node[label=below:$2$] at (3,0) {};
		\node[label=above:$1$] at (3,1) {};
		\node[label=above:$1$] at (1,1) {};
		\node[label=above:$2$] at (2,1) {};
	\end{scope}
	
	\begin{scope}[xshift=8cm]
		\draw[densely dashed] (0,0) -- (0,1);
		\draw[densely dashed] (0,0.5) -- (3,0.5);
		\foreach \x in {1,2} 
		{
			\draw[densely dashed] (\x,0.5) -- (\x,1);
		}
		\draw[densely dashed] (3,0) -- (3,1);
		\node[label=above:$1$] at (0,1) {};
		\node[label=below:$2$] at (0,0) {};
		\node[label=below:$2$] at (3,0) {};
		\node[label=above:$1$] at (3,1) {};
		\node[label=above:$2$] at (1,1) {};
		\node[label=above:$2$] at (2,1) {};
	\end{scope}
\end{tikzpicture}
\end{center}
More concretely, we can see that $\mu(\sigma; (111221))$, $\mu(\sigma; (121221)$, and $\mu(\sigma;(122221))$ are linear span of $\ell$-adic Milnor invariants. By (273), we get the right most equality \eqref{eq:9.3.2}.
\end{proof}

\subsection*{Acknowledgement}

The authors would like to thank James F. Davis, Rostislav Devyatov, Bingxiao Liu, Gw\'ena\"el Massuyeau, Pieter Moree, Masanori Morishita and Kent E. Orr for helpful communications. The authors would also like to thank Jordan Ellenberg for sending his unpublished preprint (\cite{E}) to them. H.K. is grateful to Max Planck Institute for Mathematics in Bonn for its hospitality and financial support. Y.T. is partially supported by JSPS KAKENHI Grant Numbers 17K05243 and JST CREST Grant Number JPMJCR14D6, Japan.

\appendix

\section{Table of $N_k$, $D_k$ and the rank of  $H_3(\Fn/\Gamma_k(\Fn); \mathbb{Z}_{\ell})$}
This appendix gives explicit computational table of ranks $N_k$, $D_k$ and $H_3(\Fn/\Gamma_k(\Fn); \mathbb{Z}_{\ell})$ for lower $n$ and $k$. In addition, we present the generating functions associated to them.
\subsection{Computation of $N_k$}
Let $N_k(n)$ be the rank of $\Gamma_k(F_n)/\Gamma_{k+1}(F_n)$ as $\mathbb{Z}$-module. The following is a table of  ranks $N_k(n)$ for $2 \leq n,k \leq 9$.

\begin{table}[hbt]
\begin{center}
\begin{tabular}{|l|l|l|l|l|l|l|l|l|} \hline
\diagbox[width=3em]{$n$}{$k$} & $2$ & $3$ & $4$ & $5$ & $6$ & $7$ & $8$ & $9$ \\ \hline
$2$ & $1$ & $2$ & $3$ & $6$ & $9$ & $18$ & $30$ & $56$ \\
$3$ & $3$ & $8$ & $18$ & $48$ & $116$ & $312$ & $810$ & $2184$ \\
$4$ & $6$ & $20$ & $60$ & $204$ & $670$ & $2340$ & $8160$ & $29120$ \\
$5$ & $10$ & $40$ & $150$ & $624$ & $2580$ & $11160$ & $48750$ & $217000$ \\
$6$ & $15$ & $70$ & $315$ & $1554$ & $7735$ & $39990$ & $209790$ & $1119720$ \\
$7$ & $21$ & $112$ & $588$ & $3360$ & $19544$ & $117648$ & $720300$ & $4483696$ \\
$8$ & $28$ & $168$ & $1008$ & $6552$ & $43596$ & $299592$ & $2096640$ & $14913024$ \\
$9$ & $36$ & $240$ & $1620$ & $11808$ & $88440$ & $683280$ & $5380020$ & $43046640$ \\ \hline
\end{tabular}
\caption{Table of $N_k(n)$}
\end{center}
\end{table}

It is known that the generating function associated to $N_k(n)$ is given by the cyclotomic identity
\begin{equation}\label{eq:A.1}
	\prod_{k=1}^{\infty}\left(\frac{1}{1-z^k}\right)^{N_k(n)} = \frac{1}{1-nz}
\end{equation}

\subsection{Computation of $D_k(n)$}
We set $D_k(n)= nN_k(n)-N_{k+1}(n)$. Then, the table of $D_k(n)$ for $2 \leq n \leq 9$ and $2 \leq k \leq 9$ is given as follows:
\begin{table}[hbt]
\begin{center}
\begin{tabular}{|l|l|l|l|l|l|l|l|l|} \hline
\diagbox[width=3em]{$n$}{$k$} & $2$ & $3$ & $4$ & $5$ & $6$ & $7$ & $8$ & $9$ \\ \hline
$2$ & $0$ & $1$ & $0$ & $3$ & $0$ & $6$ & $4$ & $13$ \\
$3$ & $1$ & $6$ & $6$ & $28$ & $36$ & $126$ & $246$ & $672$ \\
$4$ & $4$ & $20$ & $36$ & $146$ & $340$ & $1200$ & $3520$ & $11726$ \\
$5$ & $10$ & $50$ & $126$ & $540$ & $1740$ & $7050$ & $26750$ & $108752$ \\
$6$ & $20$ & $105$ & $336$ & $1589$ & $6420$ & $30150$ & $139020$ & $672483$ \\
$7$ & $35$ & $196$ & $756$ & $3976$ & $19160$ & $103236$ & $558404$ & $3140032$ \\
$8$ & $56$ & $336$ & $1512$ & $8820$ & $49176$ & $300096$ & $1860096$ & $11933292$ \\
$9$ & $84$ & $540$ & $2772$ & $17832$ & $112680$ & $769500$ & $5373540$ & $38747232$ \\ \hline
\end{tabular}
\end{center}
\caption{Table of $D_k(n)$}
\end{table}

In this case, similar to the case of $N_k(n)$, we have the following formula on generating function associated to $D_k(n)$.
\begin{lemma}
Notations being as above, we have 
\begin{equation}
	\prod_{k=1}^{\infty}\left(\frac{1}{1-z^k}\right)^{D_k(n)} = \frac{(1-z)^n}{(1-nz)^{n-1}}.
\end{equation}	
\end{lemma}
\begin{proof}
It follows from straight forward calculation by using 	\eqref{eq:A.1} and $N_1(n)=n$.
\end{proof}

\subsection{Computation of the rank of the third homology group of a free nilpotent group}
The table of the ranks of $H_3(\Fn/\Gamma_k(\Fn);\mathbb{Z}_{\ell})$ for $2 \leq n \leq 9$ and $2 \leq k \leq 5$  is given as follows:
\begin{table}[hbt]
\begin{center}
\begin{tabular}{|l|l|l|l|l|} \hline
\diagbox[width=3em]{$n$}{$k$} & $2$ & $3$ & $4$ & $5$ \\ \hline
$2$ & $0$ & $1\oplus0$ & $0\oplus3\oplus0$ & $3\oplus0\oplus6\oplus4$ \\
$3$ & $1$ & $6\oplus6$ & $6\oplus28\oplus36$ & $28\oplus36\oplus126\oplus246$ \\
$4$ & $4$ & $20\oplus36$ & $36\oplus146\oplus340$ & $146\oplus340\oplus1200\oplus3520$ \\
$5$ & $10$ & $50\oplus126$ & $126\oplus540\oplus1740$ & $540\oplus1740\oplus7050\oplus26750$ \\
$6$ & $20$ & $105\oplus336$ & $336\oplus1589\oplus6420$ & $1589\oplus6420\oplus30150\oplus139020$ \\
$7$ & $35$ & $196\oplus756$ & $756\oplus3976\oplus19160$ & $3976\oplus19160\oplus103236\oplus558404$ \\
$8$ & $56$ & $336\oplus1512$ & $1512\oplus8820\oplus49176$ & $8820\oplus49176\oplus300096\oplus1860096$ \\
$9$ & $84$ & $540\oplus2772$ & $2772\oplus17832\oplus112680$ & $17832\oplus112680\oplus769500\oplus5373540$ \\ \hline
\end{tabular}
\caption{table of the rank of $H_3(\Fn/\Gamma_k(\Fn); \mathbb{Z}_{\ell})$}
\end{center}
\end{table}

\Addresses

\end{document}